\documentclass[11pt]{amsart}
\usepackage{amsmath,amssymb,epsfig,color, amsthm, mathrsfs}
\usepackage{wrapfig,lipsum,floatrow,sidecap}
\usepackage{graphicx}
\usepackage{multirow}
\usepackage{hhline}
\usepackage{url}
\graphicspath{ {images/} }
\allowdisplaybreaks

\newtheorem{theorem}{Theorem}[section]

\newtheorem{definition}[theorem]{Definition}
\newtheorem{lemma}[theorem]{Lemma}
\newtheorem{proposition}[theorem]{Proposition}
\newtheorem{remark}[theorem]{Remark}
\newtheorem{example}[theorem]{Example}

\newcommand{\vanish}[1]{}\parskip=12pt

\newcommand{\sign}{\mbox{\rm sign}}

\def\p{\prime}
\def\b{\textbf{b}}

\numberwithin{equation}{section}

\begin{document}
\title{A Diagrammatic Approach for Determining the Braid Index of Alternating Links}
\author{Yuanan Diao$^\dagger$, Claus Ernst$^*$, Gabor Hetyei$^\dagger$ and Pengyu Liu$^\dagger$}
\address{$^\dagger$ Department of Mathematics and Statistics\\
University of North Carolina Charlotte\\
Charlotte, NC 28223}
\address{$^*$ Department of Mathematics\\
Western Kentucky University\\
Bowling Green, KY 42101, USA}
\email{}

\begin{abstract}
It is well known that the braid index of a link equals the minimum number of Seifert circles among all link diagrams representing it. For a link with a reduced alternating diagram $D$, $s(D)$, the number of Seifert circles in $D$, equals the braid index $\b(D)$ of $D$ if $D$ contains no {\em lone crossings} (a crossing in $D$ is called a {\em  lone crossing} if it is the only crossing between two Seifert circles in $D$). If $D$ contains lone crossings, then $\b(D)$ is strictly less than $s(D)$. However in general it is not known how $s(D)$ is related to $\b(D)$.  In this paper, we derive explicit formulas for many alternating links based on any minimum projections of these links. As an application of our results, we are able to determine the braid index for any alternating Montesinos link explicitly (which include all rational links and all alternating pretzel links).  
\end{abstract} 

\maketitle
\section{Introduction}

\medskip
It is well known that any oriented link can be represented by the closure of a braid. The minimum  number of strands needed in a braid whose closure represents a given link is called the braid index of the link. Although it is difficult to determine the braid index of a link in general, although there have been some successes. Examples include the closed positive braids with a full twist (in particular the torus links)~\cite{FW}, the 2-bridge links and fibered alternating links~\cite{Mu}, and a new class of links discussed in a more recent paper \cite{Lee}. For more readings on related topics, interested readers can refer to \cite{Bir, Crom, El, MS, Na1, Sto}. 

\medskip
Of the main results concerning braid index of a link, two of them are of particular relevance and importance to our paper. The first one relates the braid index of an oriented link $L$ to any given link diagram of $L$, and the second one relates the braid index of $L$ to its HOMFLY polynomial. While we will defer the discussion of the HOMFLY polynomial and how it is related to the braid index of $L$ to the next section, we outline the other result here.
For any given oriented link diagram $D$, a crossing in it is either a positive or a negative crossing as shown in Figure \ref{fig:cross} (marked by $D_+$ and $D_-$ respectively, and the summation of these signs is called the {\em writhe} of $D$). If we smooth all such crossings so that the diagram at each crossing looks like the one shown in Figure \ref{fig:cross} marked by $D_0$, then we obtain a link diagram that contains topological circles that do not intersect each other. These are called {\em Seifert circles} and the collection of these Seifert circles is called the {\em Seifert circle decomposition} of $D$. It turns out that the Seifert circle decomposition of $D$ is closely related to its braid index. If the Seifert circles of $D$ are concentric to each other, then $D$ is already in a closed braid form hence the number of Seifert circles in $D$ clearly gives an upper bound for the braid index of $D$ in this case. In fact, Yamada \cite{Ya} showed that one can obtain a closed braid presentation of any link diagram $D$ from its Seifert circle decomposition with the same number of Seifert circles (number of strings in the closed braid) without changing the writhe of the diagram. It follows immediately that the braid index of a link equals the minimum number of Seifert circles among all link diagrams of the link. 

\begin{figure}[htb!]
\includegraphics[scale=.6]{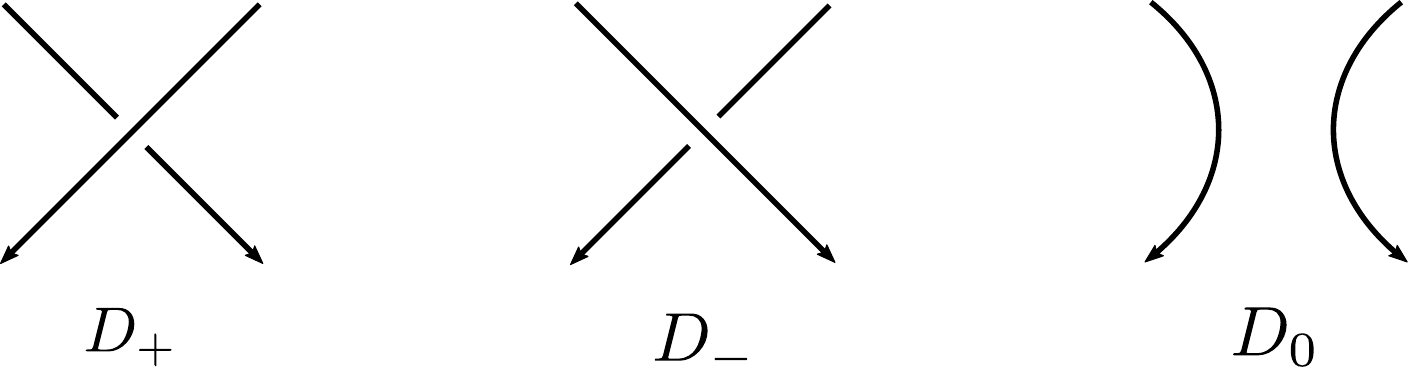}
\caption{The sign convention at a crossing of an oriented link and the splitting of the crossing: the crossing in $D_+$ ($D_-$) is positive (negative) and is assigned $+1$ ($-1$) in the calculation of the writhe of the link diagram.}
\label{fig:cross}
\end{figure}

\medskip
In this paper we are able to determine the braid index of many alternating links as an explicit function based on any minimum projection diagram of the link.
Let $D$ be a reduced alternating link diagram of some oriented alternating link $L$. A crossing in $D$ is called a {\em lone crossing} (or l-crossing for short) if it is the only crossing between two Seifert circles in $D$. If there are more than one crossings between two Seifert circles, then each of these crossings is called a {\em regular crossing} (or r-crossing for short). It is known \cite{DHL2017} that the braid index of $D$ equals the number of Seifert circles in $D$ if and only if $D$ contains no lone crossings. Thus, if $D$ contains no lone crossings, then we know its braid index is simply the number of Seifert circles in $D$. When $D$ contains lone crossings the problem of how to determine the braid index remains open in general, and this paper provides a solution for certain link families. Through careful examination of the structures of Seifert circle decompositions, we are able to explicitly determine the braid indices for many alternating links with simple formulas. In fact, our results are applicable to a very large class of links that include all alternating Montesinos links (which contain all two bridge links and alternating pretzel links).

\medskip
This paper is structured as follows. In Section~\ref{s2}, we introduce the HOMFLY polynomial and state several known results that will play key roles in our proofs later. In Section~\ref{s3} we introduce the concepts of {\em reduction numbers} and {\em base link diagrams}. In Section~\ref{s4}  we identify several classes of alternating link diagrams that are base link diagrams and show how a base link diagram can be obtained from another base link diagram. In
Section~\ref{s5}, we first derive the formulas that allow us to calculate the maximum and minimum powers of the variable $a$ in the HOMFLY polynomial for a rational link diagram based on a minimum projection of it. We then extend this result to derive a closed formula for the braid index of an alternating Montesinos link, that is also based only on a minimum projection of the Montesinos link. 

\section{Preparations and prior results}\label{s2}

\smallskip
For the sake of convenience, from this point on, when we talk about a
link diagram $D$, it is with the understanding that it is the link
diagram of some oriented link $L$. Since we will only be dealing with link
invariants such as braid index and the HOMFLY polynomial, it should not cause any confusion for
us to use $D$ as a link without mentioning $L$. Let $D_+$, $D_-$, and
$D_0$ be oriented link diagrams that coincide with each other except at a
small neighborhood of a crossing where the diagrams are presented as in
Figure~\ref{fig:cross}. We say the crossing presented in $D_+$ has a
{\em positive} sign and the crossing presented in $D_-$ has a {\em negative} sign. The following result appears in~\cite{Fr,Ja}.

\begin{proposition}\label{Ho}
There is a unique function that maps each oriented link diagram $D$ to a two-variable Laurent polynomial with integer coefficients $H(D,z,a)$ such that 
\begin{enumerate}
\item If $D_1$ and $D_2$ are ambient isotopic, then $H(D_1,z,a)=H(D_2,z,a)$.
\item $aH(D_+,z,a) - a^{-1}H(D_-,z,a) = zH(D_0,z,a)$. 
\item If $D$ is an unknot, then $H(D,z,a)=1$. 
\end{enumerate}
\end{proposition}

The Laurent polynomial $H(D,z,a)$ or $H(D)$  is called the {\em HOMFLY polynomial} or {\em HOMFLY-PT polynomial} 
of the oriented link $D$. The second condition in the proposition is
called the {\em skein relation} of the HOMFLY polynomial. With conditions (2) and (3) above, one can easily show that if  $D$ is a trivial link with $n$ connected components, then $H(D,z,a)=((a-a^{-1})z^{-1})^{n-1}$ (by applying these two conditions repeatedly to a simple closed curve with $n-1$ twists in its projection).
For our purposes, we will actually be using the following two equivalent forms of the skein relation:
\begin{eqnarray}
H(D_+,z,a)&=&a^{-2}H(D_-,z,a)+a^{-1}zH(D_0,z,a),\label{Skein1}\\
H(D_-,z,a)&=&a^2 H(D_+,z,a)-azH(D_0,z,a).\label{Skein2}
\end{eqnarray}

It is well known that 
\begin{equation}\label{connected_H}
H(D_1\#D_2)=H(D_1)H(D_2),
\end{equation} 
where $D_1\#D_2$ is the connected sum of the link diagrams $D_1$, $D_2$ and that 
\begin{equation}\label{mirro_H}
H(D,z,a)=H(D^\p,z,-a^{-1}),
\end{equation}
 where $D^\p$ is the mirror image of $D$. The following is a list of terms and notations used in this paper ($D$ stands for a link diagram). 

\noindent
\begin{itemize}
\item 
  $c(D)$: the number of crossings in $D$;
 \item   $c^-(D)$: the number of negative crossings in $D$.
\item 
  $s(D)$: the number of Seifert circles in $D$;
\item 
  $w(D)=c(D)-2c^{-}(D)$: the writhe of $D$;
\item 
  $E(D)$ and $e(D)$: the maximum and minimum powers of the variable $a$ in $H(D,z,a)$;
\item 
$E(P(z,a))$ and $e(P(z,a))$: the maximum and minimum powers of the variable $a$ in any Laurent polynomial $P(z,a)$;
\item 
  $H^h(D)$ and $H^\ell(D)$: the two terms corresponding to the highest and lowest powers of $a$ respectively when $H(D,z,a)$ is written as a polynomial of $a$;
\item   $p^h(D,z)$ and $p^\ell(D,z)$: the Laurent polynomials of $z$ serving as the coefficients of $a^{E(D)}$ and $a^{e(D)}$ in $H^h(D)$ and $H^\ell(D)$, respectively.
That is $H^h(D)=p^h(D,z)a^{E(D)}$ and $H^\ell(D)=p^\ell(D,z)a^{e(D)}$;
\item   $p_0^h(D)$ and $p_0^\ell(D)$: the highest power terms in $p^h(D,z)$ and $p^\ell(D,z)$ respectively. That is, $p_0^h(D)=b^hz^{{\rm deg}(p^h)}$ and $p_0^\ell(D)=b^\ell z^{{\rm deg}(p^\ell)}$ for some constants $b^h$ and $b^\ell$ respectively, where ${\rm deg}(p^h)$ and ${\rm deg}(p^l)$ denote the maximum power of $z$ in $p^h(D,z)$ and $p^\ell(D,z)$ respectively;
\item   $\sigma^+(D)$, $\sigma^-(D)$: the number of pairs of Seifert circles in $D$ that share multiple positive crossings and multiple negative crossings respectively;
\item We sometimes use the terms $p_0^h(P)$ and $p_0^\ell(P)$ where $P$ is not a diagram but a Laurent polynomial in the variables $a$ and $z$. Just as in the case of a diagram, our notation indicates the highest power terms in the $z$ variable of the Laurent polynomials of $z$ serving as the coefficients of the highest and lowest $a$ power in $P$.
\end{itemize}
 
The following result is well known:
\begin{theorem}\cite{Mo}
Let $D$ be any link diagram, then $E(D)\le s(D)-w(D)-1$ and $e(D)\ge -s(D)-w(D)+1$.
\end{theorem}

From the above theorem, it is immediate that $2s(D)-2\ge E(D)-e(D)$ hence $s(D)\ge (E(D)-e(D))/2+1$. It follows from the result of Yamada (as we mentioned in the introduction) that $\textbf{b}(D)\ge (E(D)-e(D))/2+1$. This last inequality is called Morton-William-Frank inequality (or MWF inequality for short). Clearly, if the equality in $s(D)\ge (E(D)-e(D))/2+1$ holds, then we must have $s(D)=\textbf{b}(D)$. This is the other important result that we mentioned in the introduction section.

\medskip
In the rest of this section, we will focus on the Seifert circle decomposition of an alternating link diagram $D$. Such a decomposition has a very special property as one can easily check: let $C$ be a Seifert circle in $D$, then the crossings that $C$ shares with other Seifert circles on one side of $C$ are either all positive or all negative, while the crossings that $C$ shares with other Seifert circles on the other side of $C$ have exactly the opposite sign. Thus if we smooth all positive crossings in $D$, we obtain a link diagram $D^{-}$ (with only negative crossings) that is still alternating, but may be consisting of disjoint link diagrams. Similarly, if we smooth all negative crossings in $D$, we obtain an alternating link diagram $D^{+}$ with only positive crossings and may be consisting of disjoint link diagrams. We will call these components the {$\partial^+S$-components} ({$\partial^-S$-components}) if they are obtained by smoothing all negative crossings (positive crossings)
of $D$, or just $\partial S$-components when there is no need to stress the signs. (The term $\partial S$ refers to ``partial Seifert circle decomposition" since only crossings of one sign are smoothed.) Notice that each $\partial S$-component is completely bounded within one Seifert circle with only one possible exception involving the unbounded region. (However there is no exception if we consider the diagram drawn on the 2-sphere.) The exception can be removed without creating any additional crossings if we use a diagram $D$ where the inside of an innermost Seifert circle becomes the unbounded region. 
See Figure \ref{Fig_ast} for an example of the three diagrams $D$, $D^{+}$ and $D^{-}$.
\begin{figure}[htb!]
\includegraphics[scale=.5]{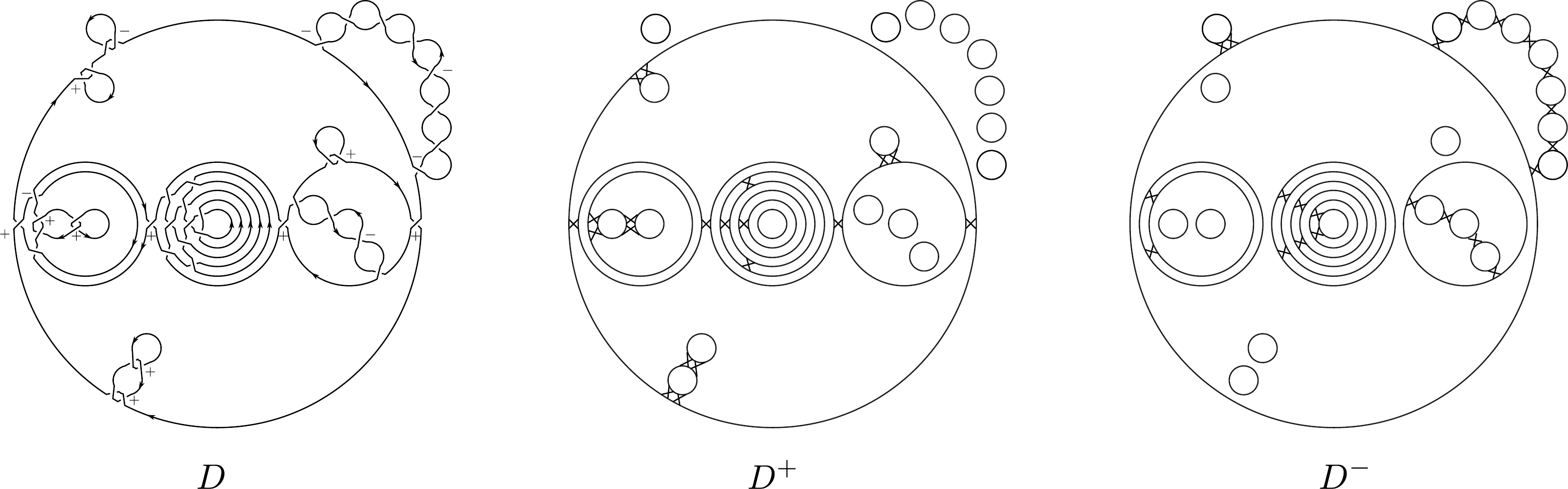} 
\caption{An alternating link diagram and its $\partial S$-components: the $\partial^+S$ components are in $D^+$ and the $\partial^-S$ components are in $D^-$. The Seifert circles in each component are connected by the remaining crossings. Notice that the $\partial^-S$-component containing the large Seifert circle is the exception where no Seifert circles in the component are contained in the interior of another in the component and it is apparent that a strand of this large Seifert circle can be rerouted (without causing any crossing changes) so the resulting (large) Seifert circle contains the other Seifert circles in this component. \label{Fig_ast}}
\end{figure}

\begin{definition}{\em 
A {\em cycle of Seifert circles} is a sequence of distinct Seifert circles $C_1$, $C_2$, ..., $C_n$ such that $C_j$ and $C_{j+1}$ share at least one crossing and $C_n$ and $C_1$ also share at least one crossing.}
\end{definition}

\begin{remark}{\em
In light of the above observations about the $\partial S$-components, and the fact that the strands over which two Seifert circles share crossings must be oriented in the same direction, we see that a cycle of Seifert circles can only occur within a $\partial S$-component hence all crossings involved must be of the same sign. Moreover the length of any cycle of Seifert circles must be even.}
\end{remark}

\begin{remark} {\em 
In \cite{M0} the concept of a {\em special} link diagram was introduced. Basically, a link diagram is called special if for every Seifert circle in the link diagram, either its interior or its exterior contains no other Seifert circles of the diagram. It is rather obvious that each $\partial S$-component of an alternating link diagram is a special link diagram.
}
\end{remark}

\section{Reduction numbers and base link diagrams}\label{s3}

Let us now consider a reduced alternating link diagram $D$ containing lone crossings. Then it is necessary that each lone crossing is part of a cycle of Seifert circles (otherwise the crossing is nugatory and $D$ is not reduced). In the case when $D$ consists of exactly one cycle of Seifert circles, such that the lone crossings occur in a consecutive manner, Figure \ref{Reduction} illustrates a systematic way to reduce the number of Seifert circles. 
The diagram on the left has eight Seifert circles. We now replace three short overpasses by the thin arcs as shown in Figure \ref{Reduction} on the left. This new diagram is isotopic to the original diagram and has only five Seifert circles as shown in Figure \ref{Reduction} on the right.

\begin{figure}[htb!]
\includegraphics[scale=.5]{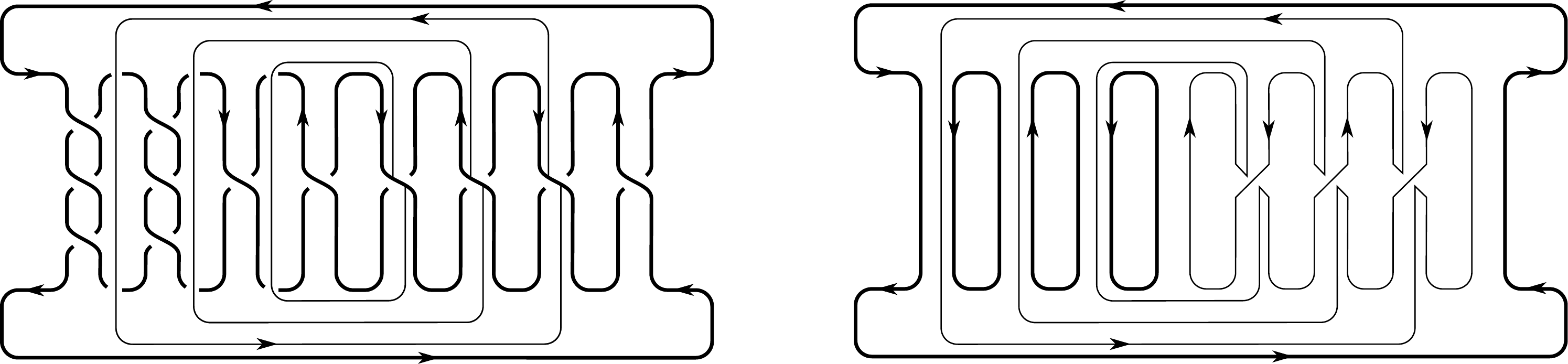}
\caption{A systematic way to reduce the number of Seifert circles in a cycle of Seifert circles. In this example, $2n=8$ is the length of cycle, $k=6$ is the number of lone crossings and the reduction number is $\min\{k,n-1\}=3$. Notice that in the figure on the right side 4 Seifert circles in the original diagram have been combined into one marked by the thin curve. \label{Reduction}}
\end{figure}

Let $2n$ be the length of the cycle and $k$ be the number of lone crossings. If the lone crossings are next to each other as shown in the figure, then in this particular case we see that we can reduce the number of Seifert circles in the diagram by $k$ if $k\le n-1$, and by $n-1$ if $k\ge n-1$. Thus we will call $\min\{n-1,k\}$ the (Seifert circle) {\em reduction number} associated to this string of lone crossings. With some modest effort, the reader can verify that (1) the same number of Seifert circles can be reduced even if the lone crossings in the cycle are not consecutive; (2) a similar operation can still be carried out if some of the Seifert circles in the cycle also share crossings with other Seifert circles in the same $\partial S$-component resulting in the same reduction number (of course, if there are two cycles in the same $\partial S$-component both containing lone crossings, then the reduction operations on one cycle may prevent the operations on the other);
 (3) this operation only involves Seifert circles in the $\partial S$-component containing the cycle. We can extend this concept to a general reduced alternating link diagram as follows: define $r^-(D)$ to be the maximum number of Seifert circles that can be reduced by rerouting strands at negative lone crossings, and define $r^+(D)$ to be the maximum number of Seifert circles that can be reduced by rerouting strands at positive lone crossings of $D$, and define $r(D)=r^-(D)+r^+(D)$ to be the (Seifert circle) reduction number of $D$. Our goal is to identify reduced alternating link diagrams $D$ satisfying the condition 
\begin{eqnarray}
E(D)&=&s(D)-w(D)-1-2r^-(D),\label{E1}\\
e(D)&=&-s(D)-w(D)+1+2r^+(D).\label{e1}
\end{eqnarray}
For if (\ref{E1}) and (\ref{e1}) hold, then we will have $s(D)-r(D)=(E(D)-e(D))/2+1\le \textbf{b}(D)$. However by the definition of $r(D)$ we also have $s(D)-r(D)\ge \textbf{b}(D)$. Thus for link diagrams satisfying (\ref{E1}) and (\ref{e1}), we have $s(D)-r(D)= \textbf{b}(D)$. 

\begin{remark}\label{r_est}{\em 
Throughout the paper, the reduction numbers are established as follows for a reduced alternating diagram $D$: First we identify $r^-(D)$ and $r^+(D)$ by showing that we can reduce the number of Seifert circles in $D$ by these numbers in the $\partial^-S$-components and $\partial^+S$-components of $D$ respectively. Then we establish that that (\ref{E1}) and (\ref{e1}) hold, that is $E(D)=s(D)-w(D)-1-2r^-(D)$ and $e(D)=-s(D)-w(D)+1+2r^+(D)$
for $r^-(D)$ and $r^+(D)$.}
\end{remark}

\begin{remark}{\em
The rerouting move used to reduce the number of Seifert circles in a link diagram at a lone crossing as described in the above is well known and is sometimes referred to as an M-P move \cite{MP}. A graph index based on the Seifert graph of a link diagram was introduced in \cite{MP} which is essentially the same as the reduction number defined here. Since we have chosen to use a diagrammatic approach in this paper, we have decided to adopt the reduction number terminology to avoid the technical details of graph theory which are unnecessary in this paper.
}
\end{remark}

\begin{remark}{\em
In general, it may be difficult to determine $r^-(D)$ and $r^+(D)$ directly. In this paper,  $r^-(D)$ and $r^+(D)$ are determined in the following way (which of course does not always apply). First we demonstrate that we can reduce the number of Seifert circles by some $r^-_0$ and $r^+_0$ in the $\partial^-S$-components and $\partial^+S$-components of $D$ respectively, then we establish that
$E(D)=s(D)-w(D)-1-2r_0^-$ and $e(D)=-s(D)-w(D)+1+2r_0^+$. The MFW inequality then implies that $\b(D)=s(D)-(r_0^-+r_0^+)$, which in turn implies that $r_0^-=r^-(D)$ and $r_0^+=r^+(D)$. Because of this, throughout the rest of the paper, we often do not distinguish $r^-_0$, $r^+_0$ and $r^-(D)$, $r^+(D)$, knowing that at the end they will be the same once we establish the equalities $E(D)=s(D)-w(D)-1-2r_0^-$ and $e(D)=-s(D)-w(D)+1+2r_0^+$.
}
\end{remark}

\begin{definition}{\em 
Let $D$ be a reduced alternating link diagram. We say that $D$ is a {\em base link diagram} if it satisfies (\ref{E1}) and (\ref{e1}).}
\end{definition}

The following theorem states that base link diagrams are additive under the connected sum operation.

\begin{theorem}\label{connectedsumtheorem}
The reduction numbers are additive under the connected sum operation and the connected sum of two base link diagrams is again a base link diagram.
\end{theorem}

\begin{proof}
Let $D_1$ and $D_2$ be two base link diagrams. It is obvious that $r^-(D_1\#D_2)\ge r^-(D_1)+r^-(D_2)$ and $r^+(D_1\#D_2)\ge r^+(D_1)+r^+(D_2)$ since the reduction can be performed on each diagram first before they are connected. It follows that 
\begin{eqnarray*}
\textbf{b}(D_1\#D_2)&\le &s(D_1\#D_2)-(r^-(D_1\#D_2)+r^+(D_1\#D_2))\\
&\le & s(D_1\#D_2)-(r^-(D_1)+r^-(D_2)+r^+(D_1)+r^+(D_2))\\
&=& s(D_1\#D_2)-(r(D_1)+r(D_2)).
\end{eqnarray*}
On the other hand,  $E(D_1\#D_2)=E(D_1)+E(D_2)$ and $e(D_1\#D_2)=e(D_1)+e(D_2)$ since $H(D_1\#D_2)=H(D_1)H(D_2)$. By the facts that $s(D_1\#D_2)=s(D_1)+s(D_2)-1$, $w(D_1\#D_2)=w(D_1)+w(D_2)$ and that $D_1$, $D_2$ are base link diagrams, we have
\begin{eqnarray*}
E(D_1\#D_2)&=& s(D_1\#D_2)-w(D_1\#D_2)-1-2(r^-(D_1)+r^-(D_2)),\\
e(D_1\#D_2)&=& -s(D_1\#D_2)-w(D_1\#D_2)+1+2(r^+(D_1)+r^+(D_2)).
\end{eqnarray*} 
Thus $E(D_1\#D_2)-e(D_1\#D_2)=2s(D_1\#D_2)-2-2(r(D_1)+r(D_2))$ hence $s(D_1\#D_2)-(r(D_1)+r(D_2))=(E(D_1\#D_2)-e(D_1\#D_2))/2+1\le \textbf{b}(D_1\#D_2)$ by the MWF inequality.
Combining this with the other inequality established earlier, we see that $r^-(D_1\#D_2)+r^+(D_1\#D_2)= r^-(D_1)+r^-(D_2)+ r^+(D_1)+r^+(D_2)$. Since $r^-(D_1\#D_2)\ge r^-(D_1)+r^-(D_2)$ and $r^+(D_1\#D_2)\ge r^+(D_1)+r^+(D_2)$, we must have $r^-(D_1\#D_2)= r^-(D_1)+r^-(D_2)$ and $r^+(D_1\#D_2)= r^+(D_1)+r^+(D_2)$. Hence $D_1\#D_2$ is a base link diagram.
\end{proof}

\section{Several families of base link diagrams}\label{s4}

In this section, we introduce several families of base link
diagrams. While these link families are already quite large themselves,
we can use them as building blocks to construct even more base link
diagrams in other constructions that are additional to the connected sum
operation. We shall demonstrate this in the next section. The names of
these link diagrams are not significant at the moment, they will be
needed in Section \ref{s5} and some of reasons for these names will
become clear. 

\medskip
\subsection{Type A link diagrams}

\smallskip
A {\em Type A link diagram} is an alternating link diagram that contains no lone crossings. The reduction numbers $r^+$ and $r^-$ are both zero. This type of link diagrams qualify for a base link diagram since (\ref{E1}) and (\ref{e1}) hold due to the following theorem. The additional details in the statement of the theorem are needed in the proofs of theorems stated later in this paper.

\begin{theorem}\label{main_lemma}\cite{DHL2017}
If $D$ is a Type A base link diagram, then $E(D)= s(D)-w(D)-1$ and $e(D)= -s(D)-w(D)+1$. Furthermore, the degree and sign of $p_0^h(D,z)$ are $c(D)-2\sigma^-(D)-s(D)+1$ and $(-1)^{c^-(D)}$ respectively. On the other hand, the degree and sign of $p_0^\ell(D,z)$ are $c(D)-2\sigma^+(D)-s(D)+1$ and $(-1)^{c^-(D)+s(D)-1}$ respectively.  
\end{theorem}

\subsection{Type B link diagrams} Let us start with a definition.

\begin{definition}{\em 
Let $D_1$, $D_2$ be two disjoint oriented link diagrams and $C_1$, $C_2$ be two Seifert circles in $D_1$ and $D_2$ respectively. Let $T_m$ be an oriented elementary torus link with parallel orientation and $m\ge 1$ crossings. We say that $C_1$ is {\em properly attached} to $C_2$  with $m$ crossings if the following hold:

(i) The attachment forms the  link $D_1\#T_m\#D_2$.

(ii) The connected sum operation of $T_m$ with $D_1$ and $D_2$ involves a subarc of $C_1$ and $C_2$ respectively, that is the Seifert circles $C_1$ and $C_2$ now meet along the $m$ crossings of $T_m$

(iii) The connected sum operation respects the orientations of  $C_1$, $C_2$ and $T_m$.

}
\end{definition}

Let $D_j$, $1\le j\le 2n$ ($n\ge 2$) be disjoint Type A link diagrams and $C_j$ be a Seifert circle in $D_j$ ($1\le j\le 2n$) that is not contained in the interior of any other Seifert circle in $D_j$. If we can properly attach each $C_j$ to $C_{j+1}$ with $m_j$ (multiple) crossings  ($1\le j\le 2n$ with $C_{2n+1}=C_1$) without creating any other crossings, then the resulting new link diagram $D$ is called a {\em Type B link diagram}. Since the original link diagrams $D_j$ are of Type A, all lone crossings in a Type B link diagram must lie on the cycle of Seifert circles $C_1C_2\cdots C_{2n}$.
It is necessary in this case that all crossings used for the attachment operation to create a Type B link diagram have the same sign. Moreover all these crossings and belong to the same $\partial S$-component of $D$. See Figure \ref{TypeB} for an illustration of a Type B link diagram.

\begin{figure}[htb!]
\includegraphics[scale=1]{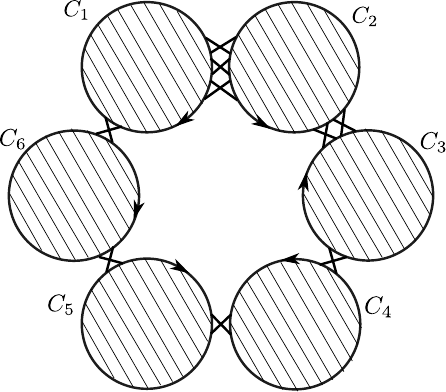}\qquad \includegraphics[scale=.85]{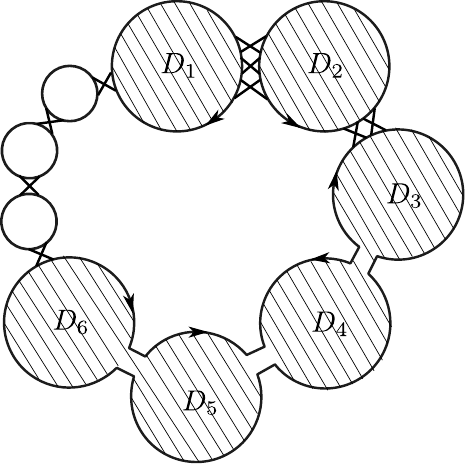}
\caption{Left: A Type B link diagram $D$ with $n=3$. Each $D_j$ is not necessarily bounded within $C_j$, the $C_j$'s are shaded in the figure only to indicate that they are part of the diagram $D_j$'s; Right: The re-arranged diagram $D$ after the flype moves, where $D_1^\p=D_1$, $D_2^\p=D_2$ and $D_3^\p=D_3\#D_4\#D_5\#D_6$. \label{TypeB}}
\end{figure}

\medskip
\begin{theorem}\label{cycle_theorem}{
Let $D$ be a Type B link diagram, then $D$ is a base link diagram. Furthermore, if the cycle of Seifert circles used to define $D$ has length $2n$ and $k$ lone crossings, then $r^-(D)=\min\{k,n-1\}$ and $r^+(D)=0$ if the lone crossings are of negative sign and $r^+(D)=\min\{k,n-1\}$ and $r^-(D)=0$ if the lone crossings are of positive sign.}
\end{theorem}

Let us point out that in the theorems of this article, whenever there are two symmetric cases, one involving the positive crossings and the other involving the negative crossings, we will always only give the proof for the positive case. The case for the negative crossings can be obtained from the positive one by (\ref{mirro_H}). The reader needs to keep this in mind when reading the proofs: in the proofs we always assume that we are dealing with the positive case, even though in the statement of the theorems we mention both cases. 

\medskip
Before we proceed to prove the theorem, let us state and prove the following lemma. This lemma will also be needed later in the next section.

\begin{lemma}\label{lowerbound_lemma}
Let $D$ be a reduced alternating link diagram with reduction numbers $r^-(D)$ and $r^+(D)>0$, then we have $E(D)\le s(D)-w(D)-1-2r^-(D)$ and $e(D)\ge -s(D)-w(D)+1+2r^+(D)$. 
\end{lemma}

\begin{remark}\label{positive_remark}
{\em
Before we proceed to prove Lemma \ref{lowerbound_lemma}, let us make the following observation. In the case when a link diagram $D$ (not necessarily alternating) has only positive crossings (called a {\em positive diagram} in the literature), if we apply Algorithm P as defined in  \cite{DHL2017} to $D$, then it is not hard to see that $E(D)=s(D)-w(D)-1$ with a single leaf vertex in the resulting resolving tree obtained by smoothing all crossings in $D$ contributing uniquely to $p^h(D,z)$. Similarly, if $D$ contains only negative crossings, then we have $e(D)=-s(D)-w(D)+1$.}
\end{remark}

\begin{proof} We will first resolve $D$ by applying (\ref{Skein2}) to every negative crossing in $D$. This leads us to 
\begin{equation}\label{H_exp_0}
H(D,z,a)=\sum W(D^j)H(D^j)
\end{equation} 
where each $D^j$ is a positive link diagram and its weight $W(D^j)$ is a
monomial obtained from the powers of the $a$ and $z$ variables in (\ref{Skein2}).
If $D^j$ is obtained by smoothing $m_j$ negative crossings in $D$ and flipping the rest of the negative crossings (there are $c^-(D)-m_j$ of them) then $W(D^j)=(-az)^{m_j}\cdot a^{2(c^-(D)-m_j)}=(-z)^{m_j}a^{w(D^j)-w(D)}$. By the Remark \ref{positive_remark}, we have $E(D^j)=s(D^j)-w(D^j)-1=s(D)-w(D^j)-1$. Since smoothing or flipping a negative crossing does not affect how we reduce the number of Seifert circles using the original positive lone crossings in $D$, we can still reduce the number of Seifert circles in $D^j$ by at least $r^+(D)$. Hence $\textbf{b}(D^j)\le s(D^j)-r^+(D)$. By the MWF inequality $E(D^j)-e(D^j)\le 2(\textbf{b}(D^j)-1)$, we have 
$
e(D^j)
\ge 
E(D^j)+2-2\textbf{b}(D^j)
\ge s(D^j)-w(D^j)-1+2-2s(D^j)+2r^+(D)
=-s(D)-w(D^j)+1+2r^+(D).$
It follows that $e(W(D^j)H(D^j))\ge w(D^j)-w(D)+(-s(D)-w(D^j)+1+2r^+(D))=-s(D)-w(D)+1+2r^+(D)$. This shows that the lowest $a$ power of each summand $W(D^j)H(D^j)$ in $\sum W(D^j)H(D^j)$ is at least $-s(D)-w(D)+1+2r^+(D)$, therefore $e(D)\ge -s(D)-w(D)+1+2r^+(D)$ as desired. The inequality $E(D)\le s(D)-w(D)-1-2r^-(D)$ can be proven analogously and is left to the reader.
\end{proof}

\begin{proof} (of Theorem \ref{cycle_theorem}.) As we mentioned before, we will only give the proof for the case where the crossings involved in the attachment operations are positive. Let $k$ be the number of lone crossings in $D$. Using flype moves, we can arrange the lone crossings in $D$ as a single group of $k$ half-twists. That is $D$ can be re-arranged into a diagram that is obtained by attaching a string of $k$ lone crossings to $C_1^\p$ and $C_{2n-k+1}^\p$ in $D_1^\p\#T_1\#D_2^\p\#\cdots\#T_{q}\#D_{2n-k+1}^\p=\hat{D}\#T_1\#T_2\#\cdots\#T_{2n-k}$, where each $D_j^\p$ is either one of the original $D_i$'s or a connected sum of some of them (which is a Type A link diagram), each $T_j$ is an elementary torus link with $m_j\ge 2$ crossings, and the $C_i^\p$ are a subset of the original Seifert circles $C_1,C_2,\cdots ,C_{2n}$.
We denote
 $\hat{D}=D_1^\p\#D_2^\p\#\cdots\#D_{2n-k+1}^\p=D_1\#D_2\#\cdots\# D_{2n}$, and assume that $C_1^\p$ and $C_{2n-k+1}^\p$ are Seifert circles in $D_1^\p$ and $D_{2n-k+1}^\p$ respectively, see Figure \ref{TypeB}.

Part 1. In this part we will show that $E(D)=s(D)-w(D)-1$. We will do this by induction on $k$, the number of lone crossings. If $k=0$, $D$ is a Type A link diagram and the statement is true by Theorem \ref{main_lemma}. If $k=2$, apply (\ref{Skein1}) to one of the two lone crossings. For $D_-$, a Reidemeister II move reduces it to a Type A link diagram $\tilde{D}_-$ with $s(\tilde{D}_-)=s(D)-2$, $w(\tilde{D}_-)=w(D)-2$. For $D_0$, a Reidemeister I move reduces it to the Type A link diagram $\tilde{D}_0=\hat{D}\#T_1\#T_2\#\cdots\#T_{2n-2}$ with $s(\tilde{D}_0)=s(D)-1$ and $w(\tilde{D}_-)=w(D)-2$. It follows that
$-2+E(D_-)=-2+E(\tilde{D}_-)=-2+s(D)-2-(w(D)-2)-1=s(D)-w(D)-3$ and $-1+E(D_0)=-1+E(\tilde{D}_0)=-1+s(D)-1-(w(D)-2)-1=s(D)-w(D)-1$. Thus we have $E(D)=s(D)-w(D)-1$ and $p_0^h(D)=zp_0^h(\tilde{D}_0)$.
If $k=1$, wlog let us assume the lone crossing is between $C_1$ and $C_{2n}$, and we will use induction on $q$, the number crossings between $C_1$ and $C_2$. We note that the case $q=1$ corresponds to the case $k=2$ dealt with previously. If $q=2$, apply (\ref{Skein1}) to one of the two crossings between $C_1$ and $C_2$. $D_-$ reduces to $\tilde{D}_0$ in the above case $k=2$ and $D_0$ is $D$ in the case $k=2$ (or $q=1$) already discussed. We have $-2+E(D_-)=-2+s(\tilde{D}_0)-w(\tilde{D}_0)-1=-2+s(D)-(w(D)-2)-1=s(D)-w(D)-1$ with $p_0^h(a^{-2}H(D_-))= p_0^h(\tilde{D}_0)$, and $-1+E(D_0)=-1+s(D)-(w(D)-1)-1=s(D)-w(D)-1$ with $p_0^h(a^{-1}zH({D}_0))=z^2p_0^h(\tilde{D}_0)$. It follows that $E(D)=s(D)-w(D)-1$ with $p_0^h(D)=z^2p_0^h(\tilde{D}_0)$. Assume that for some $q_0\ge 2$ we have $E(D)=s(D)-w(D)-1$ with $p_0^h(D)=z^qp_0^h(\tilde{D}_0)$ for any $q$ such that $1\le q\le q_0$. Then for $q=q_0+1$, $D_-$ corresponds to the case $q=q_0-1$ and $D_0$ corresponds to the case $q=q_0$ so the induction hypothesis applies. It is then straightforward to check that we have $E(D)=s(D)-w(D)-1$ with $p_0^h(D)=z^{q_0+1}p_0^h(\tilde{D}_0)$. This complete the proof for the case of $k=1$. 

Now assume that for some $k_0\ge 2$, the statement $E(D)=s(D)-w(D)-1$ holds for any $k$ such that $1\le k\le k_0$ and let us consider the case $k=k_0+1$. Again apply (\ref{Skein1}) to one of the lone crossings. $D_-$ corresponds to the case $k=k_0-1$ and $D_0$ reduces to a Type A link diagram $\tilde{D}^\p_0$, which is similar to $\tilde{D}_0$ as discussed in the case of $k=2$. We have $s({D}_-)=s(D)-2$, $w(\tilde{D}_-)=w(D)-2$, $s(\tilde{D}^\p_0)=s(D)-k_0$ and $w(\tilde{D}^\p_0)=w(D)-k_0-1$. It follows that $-2+E(D_-)=-2+s(D)-2-(w(D)-2)-1=s(D)-w(D)-3-2r^-(D)$ and $-1+E(D_0)=-1+s(D)-k_0-(w(D)-k_0-1)-1=s(D)-w(D)-1$. Thus $E(D)=s(D)-w(D)-1$ as desired.

\begin{figure}[htb!]
\includegraphics[scale=.45]{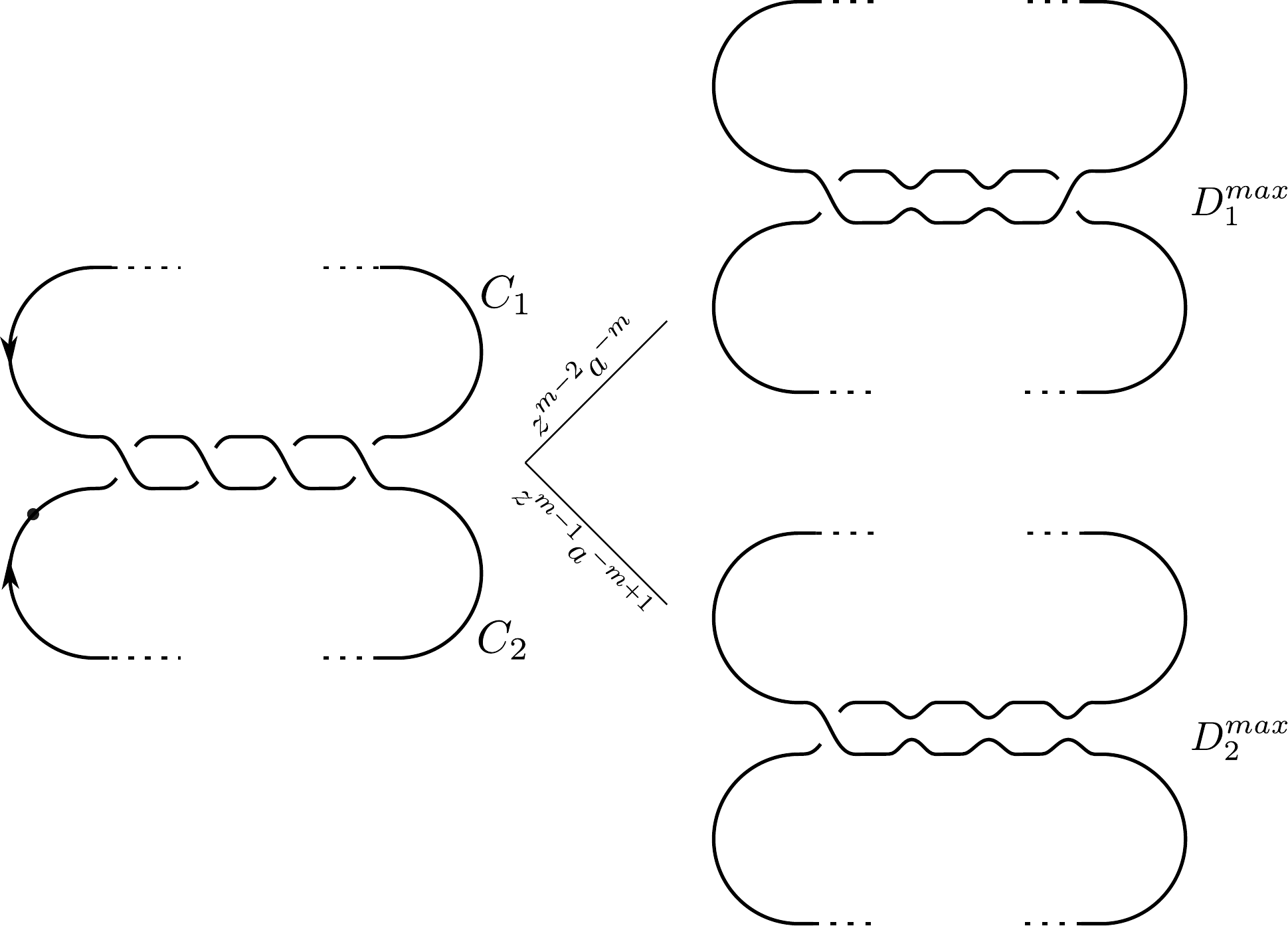}
\caption{Branching at the m-crossings between two Seifert circles corresponding to two adjacent vertices in a cycle. Shown in the figure are the representatives (one from each group) with the maximum number of crossings smoothed. }
\label{Case1Figure}
\end{figure}

\smallskip
Part 2. In this part we show that $e(D)=-s(D)-w(D)+1+2r^+(D)$. A big difference between the proof presented here and the proof in \cite{DHL2017} is that we will not be examining $H(D,z,a)$ using a complete resolving tree. Instead, we will only take several steps in that branching process using the same branching algorithm. Let $D^j$ be one of the link diagrams at the end of our branching process and let $\tilde{D}^j$ be the reduced link diagram of $D^j$. Let $W(D^j)=z^{t(D^j)}a^{w(D^j)-w(D)}$ be the combined weights through the Skein relation path leading to $D^j$ (where $t(D^j)$ is the number of crossings being smoothed along this path). Similar to (\ref{H_exp_0}), by (\ref{Skein1}) we have 
\begin{equation}\label{H_exp}
H(D,z,a)=\sum W(D^j)H(D^j)=\sum W(D^j)H(\tilde{D}^j). 
\end{equation} 
By Lemma \ref{lowerbound_lemma}, $e(D)\ge -s(D)-w(D)+1+2r^+(D)$. Thus we are done if we can demonstrate that one of the terms in the above summation attains $a^p$ as its lowest $a$ power where $p=-s(D)-w(D)+1+2r^+(D)$. Moreover among all the terms in the summation that have the same $a^p$ power, there is one term with the highest $z$ power in its corresponding $p_0^\ell$ function.
 
Case 1. $k<n-1$ so $r^+(D)=k$. If $k=0$ then $D$ is a Type A link diagram and the statement holds. For $k\ge 1$, let $C_1$ and $C_2$ be two Seifert circles in $D$ sharing $m\ge 2$ crossings. Figure \ref{Case1Figure} is an illustration of the situation with $m=4$. We will choose a starting point as shown in Figure  \ref{Case1Figure} and apply Algorithm N as defined in  \cite{DHL2017}. This means we will apply either the ascending or the descending algorithm, whichever does not allow smoothing nor flipping at the first  crossing encountered. For example in the case of Figure \ref{Case1Figure}, the ascending algorithm will be applied. We will end our branching process when we have traversed all crossings between $C_1$ and $C_2$. At this point the ${D}^j$'s fall into one of the following two groups. In the first group $C_1$ and $C_2$ detaches (via Type II Reidemeister moves) and all  lone crossings in $D^j$ become nugatory hence are removed in $\tilde{D}^j$, while in the second group $C_1$ and $C_2$ remains Seifert circles in  $\tilde{D}^j$ that share a lone crossing. Thus the $\tilde{D}^j$'s in the first group equal to the same Type A link diagram $\tilde{D}_1^{max}=\hat{D}\#T_2\#\cdots\#T_{2n-k}$ and $\tilde{D}^j$'s in the second group equal to the same link diagram $\tilde{D}_2^{max}$ which is similar to $D$ but with $k+1\le n-1$ lone crossings. The $D^j$ in each corresponding group with the maximum number of crossings smoothed is illustrated in Figure \ref{Case1Figure} which we will denote by $D_1^{max}$ and $D_2^{max}$ respectively. 

For $D_1^{max}$, we have $W(D_1^{max})=z^{m-2}a^{-m}$, $s(\tilde{D}_1^{max})=s(D)-k$, $w(\tilde{D}_1^{max})=w(D)-k-m$. For $D_2^{max}$, we have $W(D_2^{max})=z^{m-1}a^{-m+1}$, $s(\tilde{D}_2^{max})=s(D)$ and $w(\tilde{D}_2^{max})=w(D)-m+1$. 
Since $\tilde{D}_1^{max}$ is a Type A base link diagram, we have 
\begin{eqnarray*}
e(W(D_1^{max})H(D_1^{max}))&=&-m+e(\tilde{D}_1^{max})\\
&=&-m+1-s(\tilde{D}_1^{max})-w(\tilde{D}_1^{max})\\
&=&-m+1-s(D)+k-w(D)+k+m\\
&=&-s(D)-w(D)+1+2r^+(D)\ {\rm{and}}\\
p_0^\ell(W(D_1^{max})H(D_1^{max}))&=&z^{m-2}p_0^\ell(\tilde{D}_1^{max}).
\end{eqnarray*}
We observe that in the above if we replace two smoothings with the flipping of  a single crossing then we obtain the same diagram. Moreover in the term $W(D_1^{max})$ the power of $a$ remains unchanged, however the power of $z$ will be reduced by one. Thus there are multiple terms containing the same lowest $a$ power however the one obtained by the maximal number of smoothings exhibits the largest power of $z$.

On the other hand, $\tilde{D}_2^{max}$ contains $k+1\le n-1$ lone crossings, so $r^+(\tilde{D}_2^{max})=k+1$ hence by Lemma \ref{lowerbound_lemma}, $e(\tilde{D}_2^{max})\ge 1+2(k+1)-w(\tilde{D}_2^{max})-s(\tilde{D}_2^{max})=2+m+2r^+(D)-w(D)-s(D)$, thus $e(W(D_2^{max})H(D_2^{max}))\ge 3+2r^+(D)-w(D)-s(D)$.
Combining the two cases, we see that $W(D_1^{max})H(D_1^{max})$ is the unique term in the summation on the right side of Equation (\ref{H_exp}) making a contribution to the lowest $a$ power term in $H(D,z,a)$ with the highest $z$ degree in $p_0^\ell(D)$ which equals $z^{m-2}p_0^\ell(\tilde{D}_1^{max})=(-1)^{s(D)+c^-(D)}z^{c(D)-2\sigma^+(D)-s(D)+1}$. One should compare this with the similar result in Theorem \ref{main_lemma}. This proves the case $k<n-1$.

\smallskip
Case 2. $k\ge n$ and $r^+(D)=n-1$. We will use induction on $q=2n-k$. We have $n+1\ge q\ge 0$.  For $q=0$ (so $k=2n$), $\hat{D}^n=D$ is obtained by attaching a string of $2n$  lone crossings to the same Seifert circle in $\hat{D}$. Notice that the statement is true for $\hat{D}^1$  ($n=1$) since it is obtained by adding a Seifert circle to $\hat{D}$ that shares two crossings with a Seifert circle in $\hat{D}$ hence is still a Type A link diagram. 
The case $n\ge 2$ can then be proven inductively on $n$ in a manner similar to the case of $k\ge 2$ in Part 1. Furthermore, $p_0^\ell(D)=p_0^\ell(\hat{D}^1)$. The details are left to the reader. For $q=1$, let $C_1$, $C_2$ and $D_1^{max}$, $D_2^{max}$ be as described in Case 1 and assume that there are $m_1\ge 2$ crossings between $C_1$ and $C_2$. By similar discussion as in Case 1, we have $W(D_1^{max})=z^{m_1-2}a^{-m_1}$, $s(\tilde{D}_1^{max})=s(D)-k$ ($k=2n-1$), $w(\tilde{D}_1^j)=w(D)-k-m_1$, $W(D_2^{max})=z^{m_1-1}a^{-m_1+1}$, $s(\tilde{D}_2^{max})=s(D)$, $w(\tilde{D}_2^{max})=w(D)-m_1+1$ and $r^+(\tilde{D}_2^{max})=r^+(D)=n-1$. Thus
\begin{eqnarray*}
e(W(D_1^{max})H(D_1^{max}))&=&-m_1+e(\tilde{D}_1^{max})\\
&=&-m_1+1-s(\tilde{D}_1^{max})-w(\tilde{D}_1^{max})\\
&=&-m_1+1-s(D)+k-w(D)+k+m_1\\
&=&-s(D)-w(D)+1+2k
\\
&> & -s(D)-w(D)+1+2r^+(D),
\end{eqnarray*}
and 
\begin{eqnarray*}
e(W(D_2^{max})H(D_2^{max}))&=&-m_1+1+e(\tilde{D}_2^{max})\\
&=&-m_1+1-s(\tilde{D}_2^{max})-w(\tilde{D}_2^{max})+1+2r^+(\tilde{D}_2^{max})\\
&=&-m_1+1-s(D)-(w(D)-m_1+1)+1+2r^+(D)\\
&=&-s(D)-w(D)+1+2r^+(D).
\end{eqnarray*}
Thus $e(D)=-s(D)-w(D)+1+2r^+(D)$ with $p_0^\ell(D)=z^{m_1-1}p_0^\ell(\tilde{D}_2^{max})=z^{m_1-1}p_0^\ell(\hat{D}^1)$ since $\tilde{D}_2^{max}$ reduces to $\hat{D}^1$. Repeating the above argument, we can similarly show that
in general, $1\le q\le n$ (namely $k\ge n$), with $m_1\ge 2$, ..., $m_q\ge 2$ being the number of corresponding m-crossings, we have $e(D)=-s(D)-w(D)+1+2r^+(D)$ and $p_0^\ell(D)=z^{\sum_{1\le j\le q}(m_j-1)}p_0^\ell(\hat{D}^1)$. In particular, for $q=n$, we have $p_0^\ell(D)=z^{\sum_{1\le j\le n}(m_j-1)}p_0^\ell(\hat{D}^1)$.

\smallskip
Case 3. Let us now consider the last case $q=n+1$ (or $k=n-1$). Let $C_1$, $C_2$ and $D_1^{max}$, $D_2^{max}$ be as defined before and assume that there are $m_{1}\ge 2$ crossings between $C_1$, $C_2$. By the discussion earlier, we now have
$e(W(D_1^{max})H(D_1^{max}))=-s(D)-w(D)+1+2r^+(D)$ with $p_0^\ell(W(D_1^{max})H(D_1^{max}))=z^{m_{1}-2}p_0^\ell(\tilde{D}_1^{max})$ where in this case $\tilde{D}_1^{max}=\hat{D}\#T_2\#\cdots\#T_{n}$. It follows that  $p_0^\ell(W(D_1^{max})H(D_1^{max}))=(-1)^nz^{1+\sum_{1\le j\le n+1}(m_j-3)}p_0^\ell(\hat{D})$ by applying Theorem \ref{main_lemma} to the individual connected sum component $T_j$. On the other hand, we also have $e(W(D_2^{max})H(D_2^{max}))=-s(D)-w(D)+1+2r^+(D)$, but with $p_0^\ell(W(D_2^{max})H(D_2^{max}))=z^{m_{1}-1}p_0^\ell(\tilde{D}_2^{max})$. Since $\tilde{D}_2^{max}$ corresponds to the case of $q=n$, the result in Case 2 above applies to it. Hence $p_0^\ell(W(D_2^{max})H(D_2^{max}))=z^{\sum_{1\le j\le n+1}(m_j-1)}p_0^\ell(\hat{D}^1)$. By the construction of $\hat{D}^1$, we have $s(\hat{D}^1)=s(\hat{D})+1$, $c(\hat{D}^1)=c(\hat{D})+2$, $c^-(\hat{D}^1)=c^-(\hat{D})$ and $\sigma^+(\hat{D}^1)=\sigma^+(\hat{D})+1$. Since $\hat{D}^1$ is also a Type A link diagram, it follows that $p_0^\ell(\hat{D}^1)=-z^{-1}p_0(\hat{D})$ by Theorem \ref{main_lemma}. Thus $p_0^\ell(W(D_2^{max})H(D_2^{max}))=-z^{-1+\sum_{1\le j\le n+1}(m_j-1)}p_0^\ell(\hat{D})$. Since $n\ge 2$, $-1+\sum_{1\le j\le n+1}(m_j-1)>1+\sum_{1\le j\le n+1}(m_j-3)$. Thus $e(D)=-s(D)-w(D)+1+2r^+(D)$ and $p_0^\ell(D)=-z^{-1+\sum_{1\le j\le n+1}(m_j-1)}p_0^\ell(\hat{D})$.
\end{proof}

\subsection{Type M base link diagrams}

\smallskip
A {\em Type M link diagram} is an alternating link diagram $D$  obtained by attaching $m\ge 2$ strings of Seifert circles to two Seifert circles $C_1$ and $C_2$ belonging to two disjoint Type A link diagrams $D_1$ and $D_2$ such that all crossings involved in each string are positive (negative) lone crossings with one exception as depicted in Figure~\ref{TypeM}. That is, the Seifert circle in the string that is attached to $C_1$ may share multiple positive (negative) crossings. Notice that the case $m=2$ results in a base link diagram of Type B. Thus the case of interest is $m\ge 3$. Depending on the orientations of $C_1$ and $C_2$, the strings  will either all contain an even number of Seifert circles or all contain an odd number of Seifert circles, see Figure~\ref{TypeM}. 
We will call these two types of link diagrams by {\em Type M1} if $C_1$ and $C_2$ are parallel (so each string contains an even number of Seifert circles) or {\em Type M2} if $C_1$ and $C_2$ are antiparallel (so each string contains an odd number of Seifert circles). The two Seifert circles $C_1$ and $C_2$ used for this construction are called the {\em anchor Seifert circles}. We should point out that, strictly speaking, in the case when two or more strings in a Type M1 link diagram contain only one crossing each, then these crossings are actually not lone crossings since they are now multiple crossings between $C_1$ and $C_2$.

\begin{figure}[htb!]
\includegraphics[scale=.475]{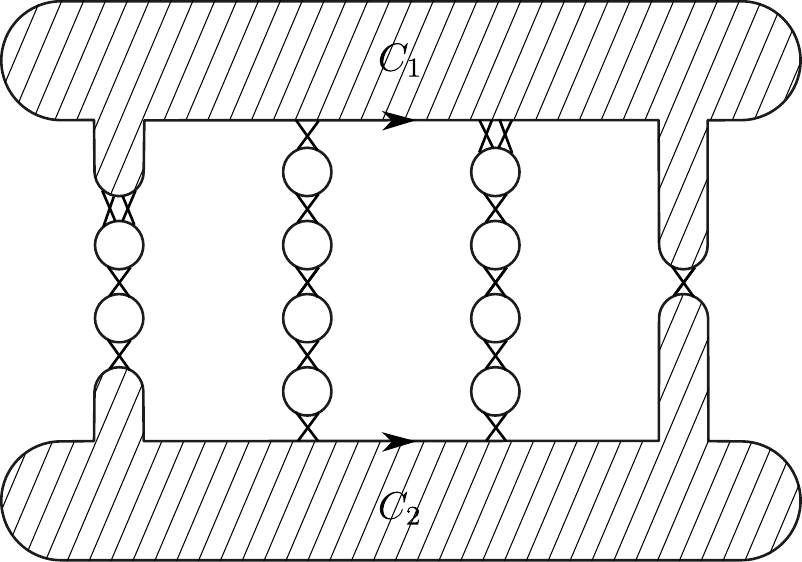}\qquad \includegraphics[scale=.54]{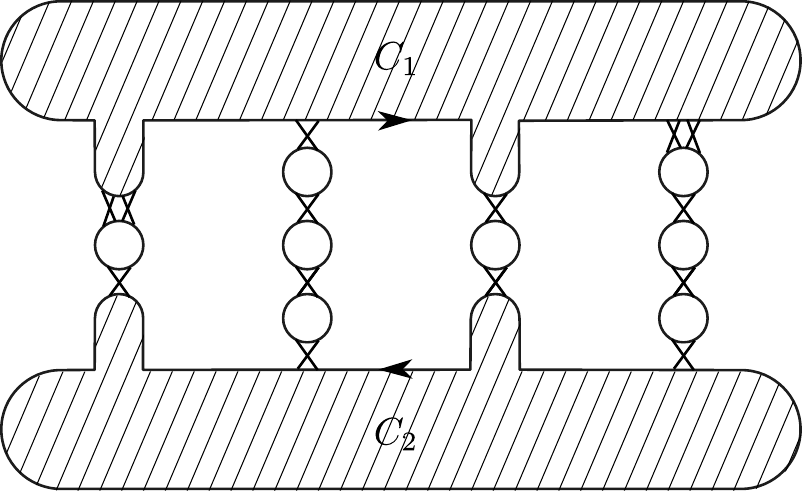}
\caption{Left: A Type M1 link diagram where the two anchor Seifert circles are parallel and the attached strings contain an even number of Seifert circles; Right: A Type M2 link diagram where the two anchor Seifert circles are antiparallel and the attached strings contain an odd number of Seifert circles. The anchor Seifert circles are shaded to indicate that they belong to two disjoint Type A link diagrams. Notice that a Seifert circle in a string attached to $C_1$ may share multiple crossings with $C_1$.
}
\label{TypeM} 
\end{figure}

In the case of Type M1, let each string contain $2k_j$ Seifert circles and at least $2k_j$ lone crossings. Then the number of Seifert circles in the string can be reduced by $k_j$ (using the reduction scheme illustrated in Figure \ref{Reduction}). If these crossings are positive (negative), then we have $r^+(D)=\sum_{1\le j\le m}k_j$, $r^-(D)=0$ ($r^+(D)=0$, $r^-(D)=\sum_{1\le j\le m}k_j$). In the case of Type M2, the situation is slightly different. Let each string contain $2k_j-1$ Seifert circles and at least $2k_j-1$ lone crossings. Then the number of Seifert circles in this string can be reduced by $k_j-1$. Figure \ref{caseM2} illustrates that an additional Seifert circle can be eliminated from the diagram. That is, if a Type M2 link diagram $D$ has $m\ge 2$ strings such that its $j$-th string ($k_j\ge 1$, $1\le j\le m$) contains $2k_j-1$ Seifert circles with at least $2k_j-1$ positive (negative) lone crossings, then we have $r^+(D)=-m+1+\sum_{1\le j\le m}k_j$, $r^-(D)=0$ ($r^+(D)=0$, $r^-(D)=-m+1+\sum_{1\le j\le m}k_j$). As we indicated in Remark \ref{r_est}, we shall see that these are indeed the reduction numbers of $D$ after we establish Theorem \ref{multipath_theorem1}.

\begin{figure}[htb!]
\includegraphics[scale=.35]{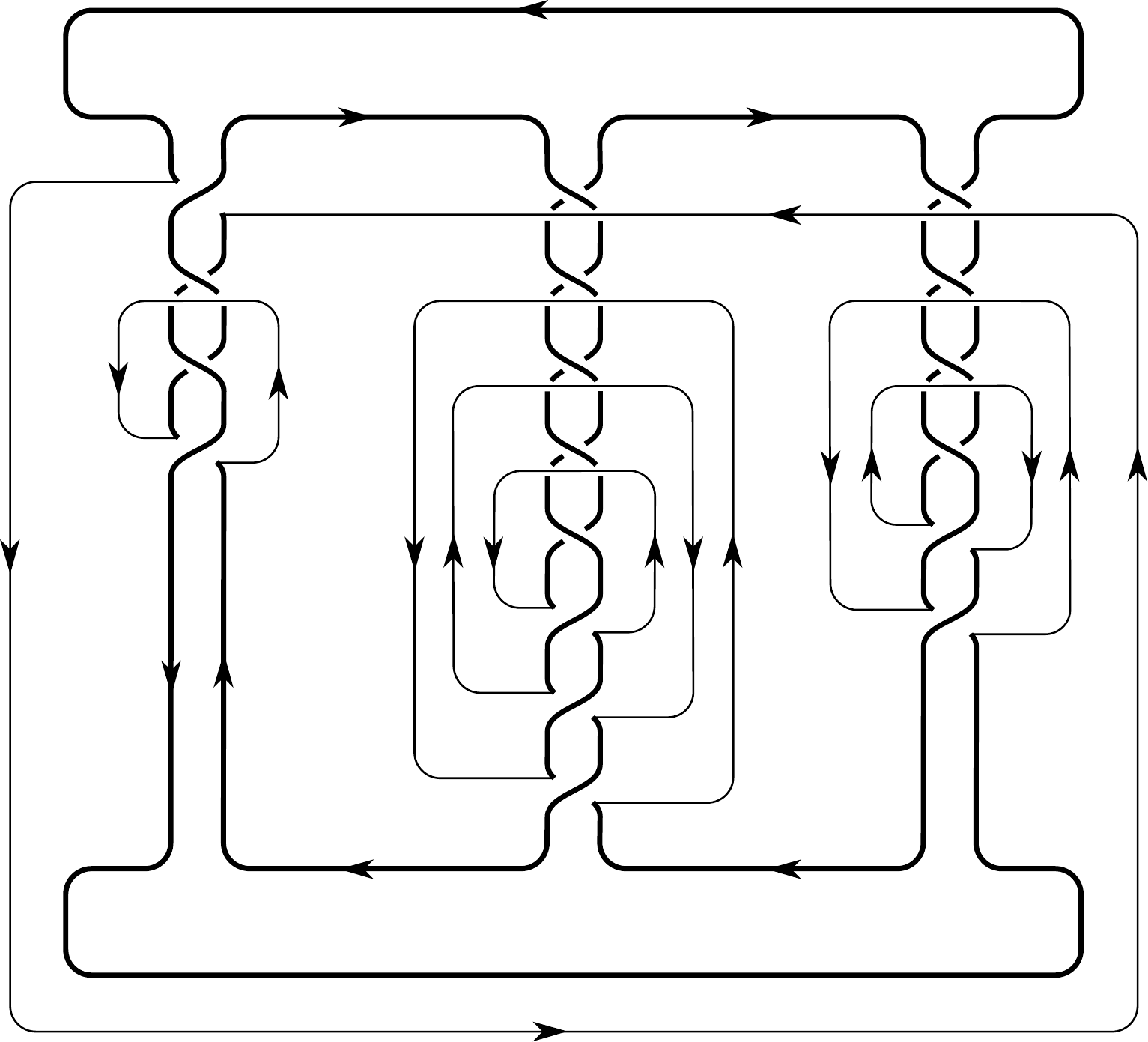}
\caption{How an additional Seifert circle can be eliminated from the first string in a Type M2 base link diagram.}
\label{caseM2} 
\end{figure}

\medskip
\begin{theorem}\label{multipath_theorem1}
A Type M link diagram $D$ is a base link diagram with the reduction numbers as defined above. Then $\textbf{b}(D)=s(D)-r(D)$, where $r(D)$ is the reduction number discussed in the previous paragraph.
\end{theorem}


\begin{proof}
We will only prove the case where all lone crossings in the attachments are positive so $r^-(D)=0$. 

\noindent
\underline{The case of Type M1 link diagrams.} Assume that we have $m\ge 3$ strings and the $j$-th string contains $2k_j$ Seifert circles with $2k_j$ ($k_j\ge 0$) lone crossings while the Seifert circle in the string attached to $C_1$ shares $\gamma_j\ge 1$ crossings with $C_1$. We have $r^-(D)=0$, $s(D)=s(D_1)+s(D_2)+2\sum_{1\le j\le m}k_j$, $w(D)=w(D_1)+w(D_2)+\sum_{1\le j\le m}(2k_j+\gamma_j)$ and we need to show that $E(D)=s(D)-w(D)-1-2r^-(D)=s(D)-w(D)-1$. We will use induction on $m\ge 2$. For $m=2$, $D$ is a Type B link diagram so the statement of the theorem follows from Theorem \ref{cycle_theorem}. Assume that the statement holds for some $m_0\ge 2$ and consider the case $m=m_0+1\ge 3$. If $k_1=k_2=\cdots=k_m=0$, then $D$ is a Type A base link diagram and the statement holds by Theorem \ref{main_lemma}. Assume that the claim is true for some $k_j=k_j^\p\ge 0$, $1\le j\le m$ and let us consider the case when one of the $k_j$'s has been increased by one. W.l.o.g. we assume that $k_m=k_m^\p+1$ and apply (\ref{Skein1}) to a lone crossing within the $m$-th string. 
Then $D_-$ is the diagram with $k_j=k_j^\p$, $1\le j\le m$, $s(D_-)=s(D)-2$, $w(D_-)=w(D)-2$, $r^-(D_-)=r^-(D)=0$, $r^+(D_-)=r^+(D)-1$ hence 
$-2+E(D_-)=-2+(s(D)-2)-(w(D)-2)-1=s(D)-w(D)-3= E(D)-2$ and 
$-2+e(D_-)=-2-(s(D)-2)-(w(D)-2)+1+2(r^+(D)-1)=-s(D)-w(D)+1+2r^+(D)$.
On the other hand, $D_0$ simplifies to $\tilde{D}_0=D^\p\#T_{\gamma_m}$, where $D^\p$ is the Type M1 link diagram with $m-1=m_0\ge 2$ strings (obtained from $D$ by removing the $m$-th string from it) and $T_{\gamma_m}$ is trivial if $\gamma_m=1$. If $\gamma_m=1$, then
$s(\tilde{D}_0)=s(D)-2k_m$, $w(\tilde{D}_0)=w(D)-2k_m-1$, $r^-(\tilde{D}_0)=r^-(D)=0$, $r^+(D_0)=r^+(D)-k_m$ hence by the induction hypothesis 
$-1+E(D_0)=-1+E(\tilde{D}_0)=-1+(s(D)-2k_m)-(w(D)-2k_m-1)-1=s(D)-w(D)-1=E(D),$ and
$
-1+e(D_0)=-1+e(\tilde{D}_0)=-1-(s(D)-2k_m)-(w(D)-2k_m-1)+1+2(r^+(D)-k_m)=-s(D)-w(D)+1+2k_m+2r^+(D)>-s(D)-w(D)+1+2r^+(D).$ Thus $E(D)=s(D)-w(D)-1$ and $e(D)=-s(D)-w(D)+1+2r^+(D)$.
Similarly, if $\gamma_m>1$, then 
$s(\tilde{D}_0)=s(D)-2k_m+1$, $w(\tilde{D}_0)=w(D)-2k_m$, $r^-(\tilde{D}_0)=r^-(D)=0$, $r^+(D_0)=r^+(D)-k_m$. Thus by the induction hypothesis we also have
$-1+E(D_0)=-1+E(\tilde{D}_0)=-1+(s(D)-2k_m+1)-(w(D)-2k_m)-1=s(D)-w(D)-1=E(D)$ and
$
-1+e(D_0)=-1+e(\tilde{D}_0)=-1-(s(D)-2k_m+1)-(w(D)-2k_m)+1+2(r^+(D)-k_m)=-s(D)-w(D)+1+2k_m+2r^+(D)>-s(D)-w(D)+1+2r^+(D).$
Thus we still have $E(D)=s(D)-w(D)-1$ and $e(D)=-s(D)-w(D)+1+2r^+(D)$, as desired.
This proves the case of Type M1 link diagrams. 

\noindent
\underline{The case of Type M2 link diagrams.} We will skip the case of $E(D)$ it is similar to the Type M1 link diagrams. 
Below we only provide the proof for $e(D)$, keep in mind that in the case of a Type M2 diagram, the crossings in the attached strings are all lone crossings.
Let us first consider the case $k_j=1$ for $1\le j\le m$. Here $r^+(D)=1$ for all $m\ge 2$. If $m=2$, the result holds by Theorem \ref{cycle_theorem}. Assume now that the formula in the theorem is true for $m=n\ge 2$, $k_1=k_2=\cdots=k_n=1$. Let us consider the case for $m=n+1$, $k_1=k_2=\cdots=k_n=k_{n+1}=1$. If the last string contains two lone crossings, then we apply (\ref{Skein1}) to a lone crossing in the last string attached. $D_-$ simplifies to a Type A base link diagram $\tilde{D}_-$ with $s(\tilde{D}_-)=s(D)-2$, $w(\tilde{D}_-)=w(D)-2$. Thus by Theorem \ref{main_lemma} we have $-2+e(D_-)=-2-s(\tilde{D}_-)-w(\tilde{D}_-)+1=-2-(s(D)-2)-(w(D)-2)+1=-s(D)-w(D)+3=-s(D)-w(D)+1+2r^+(D)$.  On the other hand, 
$D_0$ simplifies to the link diagram $\tilde{D}_0$ from the previous step with $m=n$, $s(\tilde{D}_0)=s(D)-1$, $w(\tilde{D}_0)=w(D)-2$ and $r^+(D_0)=1=r^+(D)$. By the induction hypothesis we have $-1+e(D_0)=-1+e(\tilde{D}_0)=-1-s(\tilde{D}_0)-w(\tilde{D}_0)+1+2r^+(\tilde{D}_0)=-1-(s(D)-1)-(w(D)-2)+1+2r^+(D)=-s(D)-w(D)+3+2r^+(D)>s(D)-w(D)+1+2r^+(D)$. Hence $e(D)=-s(D)-w(D)+1+2r^+(D)$. If on the other hand, the last string contains only one lone crossing and the Seifert circle in the string attached to $C_1$ shares $\gamma_m\ge 2$ crossings with $C_1$, then we resolve these crossings using the method depicted in Figure \ref{Case1Figure}. In this case $D_1^{max}$ simplifies to the diagram $D_n$ corresponding to $m=n$ with $s(D_n)=s(D)-1$, $w(D_n)=w(D)-\gamma_m-1$, hence (by the induction hypothesis)
$-\gamma_m+e(D_1^{max})=-m-(s(D)-1)-(w(D)-m-1)+1+2r^+(D)=-s(D)-w(D)+3+2r^+(D)$. On the other hand, $D_2^{max}$ corresponds to the diagram with $m=n+1$ where the last string contains two lone crossings, we have $s(D_2^{max})=s(D)$, $w(D_2^{max})=w(D)-\gamma_m+1$, hence by the previous step we have 
$-\gamma_m+1+e(D_2^{max})=-\gamma_m+1-s(D)-(w(D)-\gamma_m+1)+1+2r^+(D)=-s(D)-w(D)+1+2r^+(D)$. Thus we have $e(D)=-s(D)-w(D)+1+2r^+(D)$ as desired.
This completes the proof that the theorem holds for any number of $m$ strings as long as each string contains exactly one Seifert circle, that is, $k_1=k_2=\cdots=k_m=1$.

Assume now that the statement of the theorem holds for some $k_1\ge 1$, ..., $k_m\ge 1$. Notice that if $m=2$ then $r^+(D)=k_1+k_2-1$ and the statement of the theorem holds by Theorem \ref{cycle_theorem}. So we can further assume that $m\ge 3$ and consider the case when one of the $k_j$'s is increased by one.  W.l.o.g. assume that $k_m$ is increased to $k^\p_m=k_m+1\ge 2$ (so the number of lone crossings in the $m$-th string is increased by two). Apply (\ref{Skein1}) to a lone crossing within the $m$-th string. $D_-$ is the link diagram in the previous step and $D_0$ simplifies to $D_n\#T_{\gamma_m}$, where $D_n$ is the Type M2 base link diagram with $m-1=n$ strings obtained from $D$ by removing the last string of Seifert circles attached to it, and $\gamma_m\ge 1$ is the number of crossings between $C_1$ and the Seifert circle in the $m$-string that is attached to it. We have $r^+(D_-)=r^(D)-1$ and $r^+(D_n)=r^+(D)-k^\p_m+1$. In the case that $\gamma_m> 1$, $s(D_n\#T_{\gamma_m})=s(D)-2k^\p_m+2$, $w(D_n\#T_{\gamma_m})=w(D)-2k^\p_m+1$, thus we have (by the induction hypothesis) $-2+e(D_-)=-2-s(D_-)-w(D_-)+1+2r^+(D_-)=-2-(s(D)-2)-(w(D)-2)+1+2(r^+(D)-1)=-s(D)-w(D)+1+2r^+(D)$ and $-1+e(D_0)=-1-(s(D)-2k^\p_m+2)-(w(D)-2k^\p_m+1)+1+2(r^+(D)-k^\p_m+1)=-s(D)-w(D)+1+2r^+(D)+2(k^\p_m-1)>-s(D)-w(D)+1+2r^+(D)$ since $k^\p_m\ge 2$. Thus $e(D)=-s(D)-w(D)+1+2r^+(D)$ as desired. This finishes the proof.
\end{proof}

\subsection{Base link diagrams through further attachment operations}

In general, any alternating link diagram $D$ can be constructed from Type A link diagrams by attaching strings of lone crossings to Seifert circles in these link diagrams, or/and by attaching Seifert circles to one another (so that they will share lone crossings or multiple crossings). We will call these Type A link diagrams used to construct $D$ this way the {\em underlining} Type A link diagrams. The Type B or M link diagrams are two such examples. In the case of Type B or M link diagrams, our construction attached lone crossings to rather arbitrary Seifert circles in the Type A link diagrams and no particular properties of the Type A link diagrams were needed. Thus if we modify the underlining Type A link diagrams or replace the underlining Type A link diagrams with different ones to construct different Type B or M link diagrams, then the resulting (alternating) new link diagrams are still Type B or M link diagrams with the same reduction numbers. The only restriction on the modification or replacement of Type A link diagrams is that the new Type A link diagrams provide the needed Seifert circles with an unchanged orientation of these circles.
Such modifications include (but certainly are not limited to) attaching (not necessarily properly) a new Seifert circle to one of the Type A link diagrams, or adding more crossings between two Seifert circles (in one of the Type A link diagrams) that already share multiple crossings. This observation motivates the following definition.

\begin{definition}\label{strong_definition}{\em
Let $D$ be a base link diagram constructed from certain underlining Type A link diagrams. 
Let $D^\p$ be an alternating link diagram that is constructed in the same way as $D$ except that the construction uses different Type A link diagrams. If any such diagram $D^\p$ is a base link diagram with the same reduction numbers as that of $D$, then $D$ is called a {\em strong base link diagram}.}
\end{definition}

\begin{remark}{\em
While we do not know whether all base link diagrams are strong, we know at least a connected sum of an arbitrary number of the Type A, B and M link diagrams is also a strong base link diagram. 
}
\end{remark}

In this subsection we introduce an approach to construct new base link diagrams from an existing strong base link diagram by attaching strings of Seifert circles with lone crossings.  

\begin{definition}{\em 
Let $C$ be a Seifert circle in an alternating link diagram. A string $S$ consisting of $2k-1$ ($k\ge 1$) Seifert circles attached to $C$ is called a {\em Type I string} if any two consecutive Seifert circles in the string are connected by a lone crossing, and either the first or the last Seifert circle in the string is attached to $C$ by a lone crossing, while the other one is connected to $C$ either by a lone crossing or by multiple crossings.
}
\end{definition}
Notice that attaching $S$ adds $2k-1$ Seifert circles to the diagram and increases the reduction number (either $r^-$ or $r^+$) by $k-1$. Multiple Type I strings are allowed to be attached to the same Seifert circle. See Figure \ref{TypeIattachment} for an illustration. 

\begin{figure}[htb!]
\includegraphics[scale=.4]{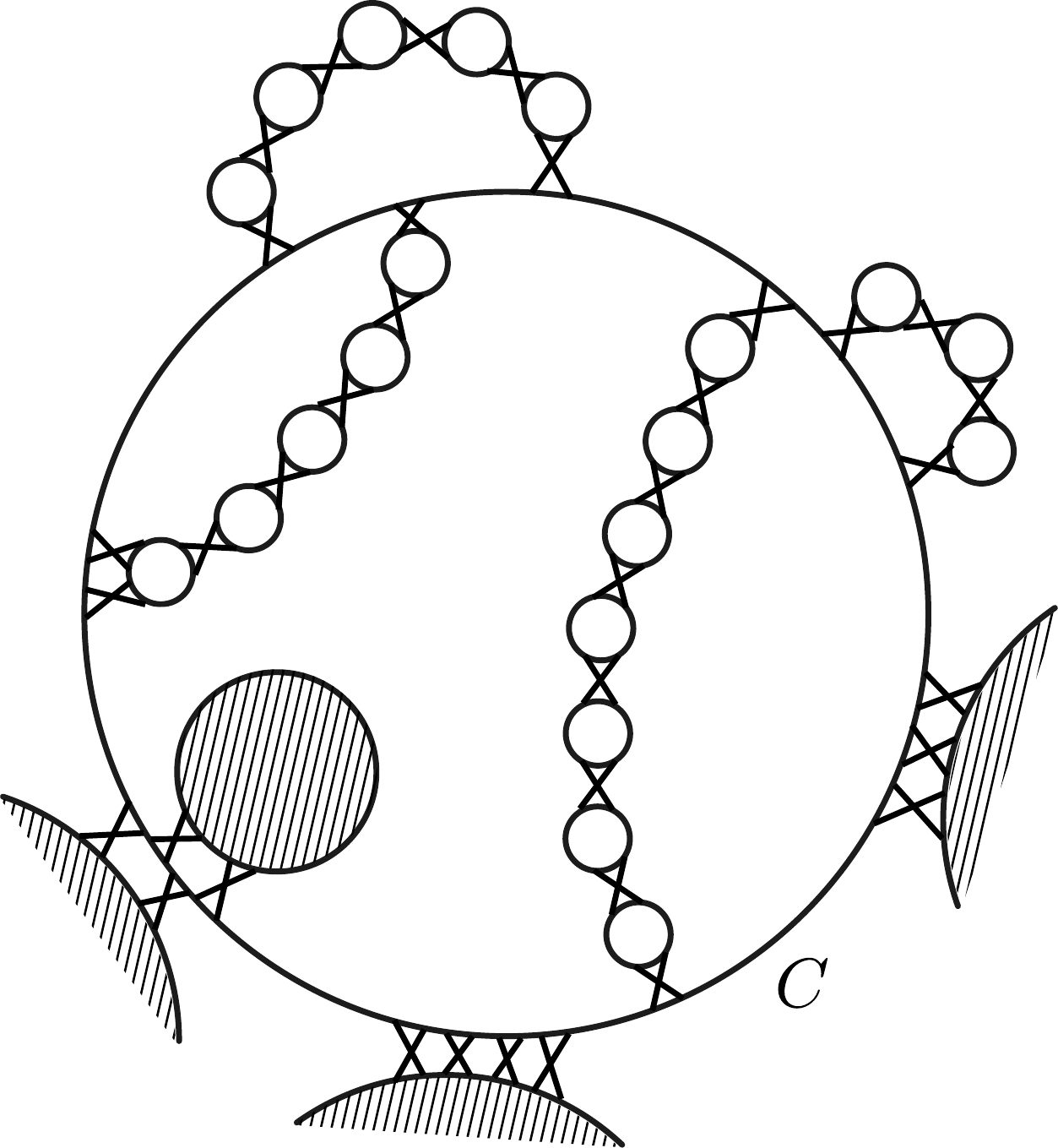}
\includegraphics[scale=.4]{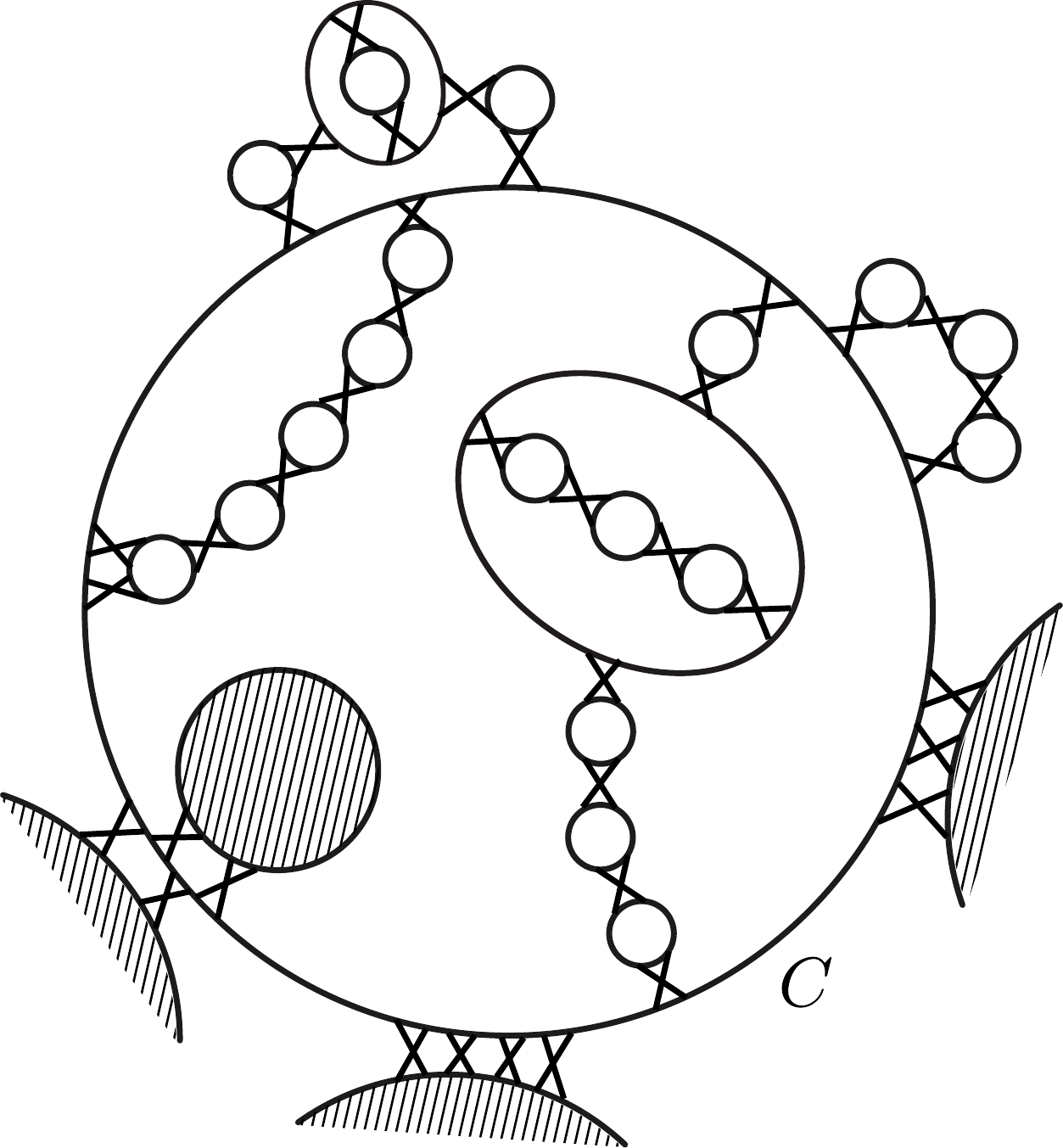}
\caption{ 
Type I attachment of several strings to a single Seifert circle. The shaded Seifert circles indicate other Seifert circles in the original (base) link diagram sharing crossings with $C$. Right: Repeated Type I attachment operations using Seifert circles created by previous Type I attachments.   
\label{TypeIattachment} } 
\end{figure}

\vspace{.1in}
 \begin{theorem}\label{TypeI_theorem}
Let $D_b$ be a strong base link diagram and let $D$ be a link diagram obtained by attaching Type I strings to Seifert circles of the underlining Type A diagrams used to construct $D_b$, then $D$ remains a strong base link diagram.
 Furthermore, the reduction numbers of $D$ are the summations of the reduction numbers of $D_b$ and those of the attached strings.
 \end{theorem}
 
\begin{proof} We will use induction on $n$, the number of Type I strings attached to $D_b$. 
For $n=0$, the result holds since the initial diagram is a strong base link diagram. Assume now that the statement of the theorem is true for $n=n_0\ge 0$. Now consider the case $n=n_0+1$. Without loss of generality, let us assume that the crossings in $S$ are of positive sign. For $k=1$, the operation attaches a single new Seifert circle $C^\p$ to $C$ with at least two crossings (so they are not lone crossings). Thus this attachment only modifies the underlining Type A link diagrams used to construct $D_b$ since $C$ belongs to one of the underlining Type A link diagrams. By the induction hypothesis, $D$ is a strong base link diagram and the reduction number remains unchanged. Now assume that the statement is true for some $k_0\ge 1$ and consider the case $k=k_0+1$. Keep in mind that one end Seifert circle $C^\p$ in $S$ may be attached to $C$ with multiple crossings. Let the number of crossings between $C$ and $C^\p$ be $m\ge 1$.
Apply \ref{Skein1} to a lone crossing in the added string. $D_-$ reduces (via a Reidemeister move II) to the link diagram $\tilde{D}_-$ which corresponds to $k=k_0$ while $D_0$ reduces (via Reidemeister I moves) to the link diagram $\tilde{D}_0=D^\p\#T_{m}$ where $T_m$ is trivial if $m=1$, and $D^\p$   corresponds to $n=n_0$. If $m=1$, then we have $s(D)=s(\tilde{D}_-)+2=s(\tilde{D}_0)+2k-1$, $w(D)=w(\tilde{D}_-)+2=w(\tilde{D}_0)+2k$, $r^-(D)=r^-(D_+)=r^-(\tilde{D}_-)=r^-(\tilde{D}_0)$ and $r^+(D)=r^+(\tilde{D}_-)+1=r^+(\tilde{D}_0)+k-1$. Thus
 \begin{eqnarray*}
 -2+E(D_-)&=& -2+s(\tilde{D}_-)-w(\tilde{D}_-)-1-2r^-(\tilde{D}_-)\\
 &=&
 -2+(s(D)-2)-(w(D)-2)-1-2r^-(D)\\
 &=&
 s(D)-w(D)-3-2r^-(D),\\
 -1+E(D_0)&=& -1+s(\tilde{D}_0)-w(\tilde{D}_0)-1-2r^-(\tilde{D}_0)\\
 &=&
 -1+(s(D)-2k+1)-(w(D)-2k)-1-2r^-(D)\\
 &=&
 s(D)-w(D)-1-2r^-(D).
 \end{eqnarray*}
It follows that $E(D)=s(D)-w(D)-1-2r^-(D)$ as desired. Similarly we have
 \begin{eqnarray*}
 -2+e(D_-)&=& -2-s(\tilde{D}_-)-w(\tilde{D}_-)+1+2r^+(\tilde{D}_-)\\
 &=&
 -2-(s(D)-2)-(w(D)-2)+1+2(r^+(D)-1)\\
 &=&
 -s(D)-w(D)+1+2r^+(D),\\
 -1+e(D_0)&=& -1-s(\tilde{D}_0)-w(\tilde{D}_0)+1+2r^+(\tilde{D}_0)\\
 &=&
 -1-(s(D)-2k+1)-(w(D)-2k)+1+2(r^+(D)-k+1)\\
 &=&
 -s(D)-w(D)+2k+1+2r^+(D),
 \end{eqnarray*}
hence $e(D)=-s(D)-w(D)+1+2r^+(D)$ since $k>1$. This proves that $D$ is still a base link diagram. Since the proof is valid for any underlining Type A link diagrams, $D$ is still a strong base link diagram. The case of $m>1$ can be similarly proved and is left to the reader.
\end{proof}

A {\em R-pattern} is a string of Seifert circles that is attached to one side of a Seifert circle $C$ in a strong base link diagram $D_b$ in a consecutive manner over an arc of $C$ (by an arc we mean that no other Seifert circles in $D_b$ share crossings with $C$ over this arc - with one exception described below in the case of {\em interlocked} R-patterns) where some of the Seifert circles in the string become attached to $C$ by a lone crossing or multiple crossings (such Seifert circles in the pattern attached to $C$ are called {\em attaching circles}). In particular, the first and the last Seifert circles in the string are attaching circles. The R in the name R-pattern stands for rational, see Section \ref{s5}. Since a cycle of Seifert circles must have even length, the string of lone crossings between two consecutive attaching circles must be even.
Two R-patterns are {\em interlocked} if they are attached to different sides of $C$ and overlap such that the ending Seifert circle in one of them is attached to $C$ between the space where the first two attaching Seifert circles of the other R-pattern are attached (from the other side of $C$), and the attachment is via a lone crossing (indicated by the arrows in Figure \ref{ValidAttachment}). 
\begin{figure}[htb!]
\includegraphics[scale=.6]{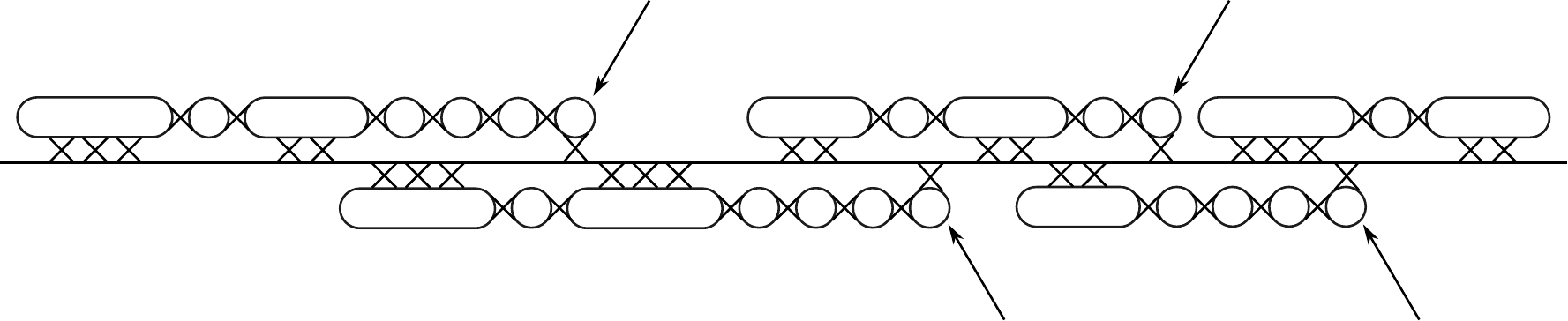}
\caption{Five interlocked R-patterns are shown in the figure where the horizontal segment is part of the Seifert circle $C$. The arrows indicate where lone crossings must be used for the attachment.
\label{ValidAttachment} }
\end{figure}
We require that no other overlaps between two R-patterns are allowed.
However, for the last R-pattern in an interlocked sequence of R-patterns, we allow the following two exceptions: 

(i) between the last  two (or first two, but not both) anchoring locations of the last $R$ pattern, $C$ may share crossings with other Seifert circles in $D$ and the attachment is via a lone crossing, see the left of Figure \ref{exceptions} for an illustration of this situation; 

(ii) the last (or first, but not both) Seifert circle in the sequence may be attached to a Seifert circle $C^\p$ of $D$ by a lone crossing, and in this case $C^\p$ shares multiple crossings with $C$ and one of these crossings is ``used" by this R-pattern as shown in the right side of Figure  \ref{exceptions}.
\begin{figure}[htb!]
\includegraphics[scale=.8]{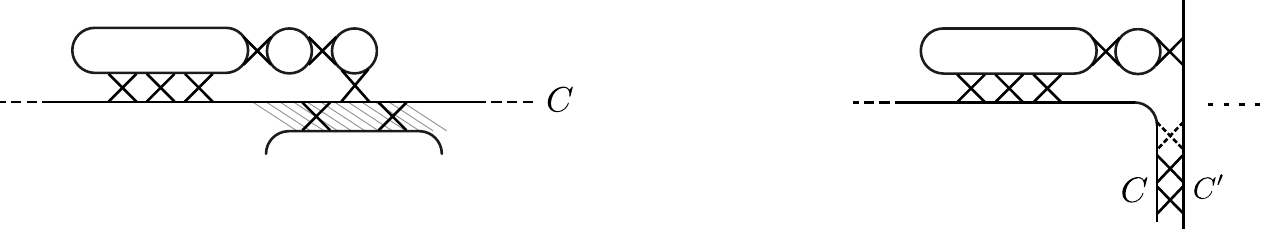}
\caption{The two exceptions to an end string in a sequence of interlocked R-patterns. Left: the shaded part indicates where other Seifert circles in the original base link diagram may be attached to $C$; Right: the dotted crossing indicates that there was an extra crossing in the original link diagram between $C$ and $C^\p$ (which is being ``used" to accommodate the attachment of the last string of Seifert circles).
\label{exceptions} }
\end{figure}

\begin{definition}\label{Rpattern}{\em 
A  {\em Type II} attachment is defined as a sequence of interlocked R-patterns attached to both sides of a Seifert circle $C$ in a strong base link diagram $D$. In particular we define a {\em Type II(i)} attachment or a {\em Type II(ii)} attachment as a Type II attachment of interlocked R-patterns that uses exception (i) or exception (ii) respectively.}
\end{definition}

For a R-pattern that contains $m$ ($m\ge 1$) strings with attaching Seifert circles $C_1$, $C_2$, ..., $C_{m+1}$, let $2k_j$ be the number of lone crossings in the $j$-th string that are between $C_j$ and $C_{j+1}$ ($k_j\ge 1$, $1\le j\le m$). The entire pattern contributes a total of $1+\sum_{1\le j\le m}2k_j$ Seifert circles to the resulting link diagram. The reduction numbers of the $j$-th string is $k_j$ and the total reduction number of this pattern is $\sum_{1\le j\le m}k_j$. Figure \ref{RpatternReduction} illustrates how these reduction numbers may be achieved. The details are left to our reader to verify. The theorem below then confirms that $\sum_{1\le j\le m}k_j$ is indeed the reduction number of a Type II attachment operation. 

\begin{figure}[htb!]
\includegraphics[scale=.35]{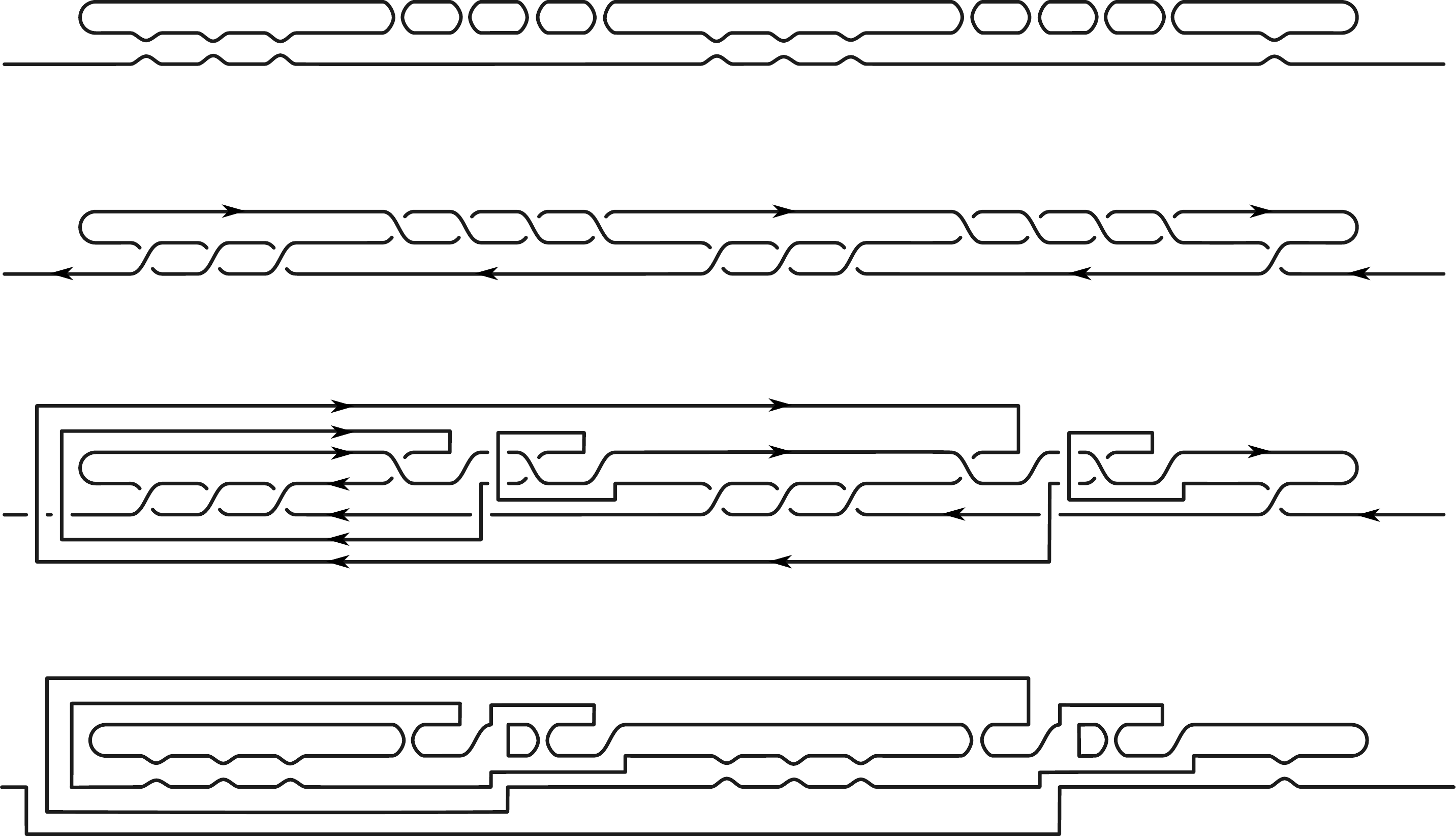}
\caption{An illustration of how the reduction numbers in a R-pattern may be achieved. Top: the Seifert circle decomposition of the original link diagram; upper middle: the original link diagram; bottom middle: the link diagram after the rerouting of strings; bottom: the Seifert circle decomposition of modified link diagram. Notice that the rerouting of the strand at the last lone crossing in a string needs to wrap around the strings that have been modified in earlier steps as shown.
\label{RpatternReduction} }
\end{figure}

\begin{theorem}\label{TypeD_theorem}
Let $D$ be a link diagram obtained by Type II attachment operations without exceptions or Type II(i) attachment operations on a strong base link diagram $D_b$ such that the attachments are made to Seifert circles in the underlining Type A link diagrams, then $D$ remains a strong base link diagram, where the reduction numbers of $D$ are the sum of the respective reduction numbers of $D_b$ and the reduction numbers of the attached interlocked R-pattern sequences as described above.
\end{theorem}

We first establish the Type II attachment operation for a single R-pattern.

\begin{lemma}\label{TypeD_theoremsingleR}
Let $D$ be a link diagram obtained by an attachment operation of a single R-pattern on a strong base link diagram $D_b$ with possibly exception (i), then $D$ remains a strong base link diagram.
\end{lemma}

\begin{proof} Consider the case of a sequence of interlocked R-patterns. We show that we can attach one R-pattern at a time and use induction on the number $n$ of attaching Seifert circles in the sequence Seifert circles for a single R-pattern. Keep in mind that we are only considering the cases of no exceptions or exception (i).

For $n=1$, we are attaching one Seifert circle with $m\ge 2$ multiple crossings to a Seifert circle in a strong base link diagram $D_b$, so we still have a base link diagram. 

For $n=2$, let $C_1$ and $C_2$ be the two attaching Seifert circles. If $C_1$ and $C_2$ are both attached to an arc of $C$ with no other Seifert circles share crossings with $C$ between the two attaching locations of $C_1$ and $C_2$ to $C$ then the attachment creates a connected sum of the strong base link diagram $D$ with a Type B link diagram. In the special case that both $C_1$ and $C_2$ are attached using a single crossing then this will be a connected sum of the strong base link diagram $D$ with a torus link $T(2k+2,2)$.
If both $C_1$ and $C_2$ are attached to $C$ with multiple crossings then the Type B link diagram contains $2k+2$ Seifert circles and $2k$ lone crossings ($k\ge 1$). If one of the two Seifert circles is attached using a lone crossing then the Type B link diagram contains $2k+2$ Seifert circles and $2k+1$ lone crossings ($k\ge 1$).
By Theorem \ref{connectedsumtheorem} and \ref{cycle_theorem}, the statement of the theorem holds. If all crossings in the string are lone crossings then it is in fact a Type I string and the statement holds by Theorem \ref{TypeI_theorem}. 

The remaining case is exceptions (i) in Definition \ref{Rpattern}. Let us assume that $C_2$ shares a lone crossing with $C$ and $C_1$ shares $m\ge 1$ crossings with $C$. If $m=1$ then this is a Type I attachment and by Theorem \ref{TypeI_theorem} attaching the R-pattern will form a strong base link diagram.
If $m>1$ then the string contains $2k+1$ lone crossings and $2k+1$ Seifert circles. For $k=1$,  apply (\ref{Skein1}) to one of the lone crossings. $D_-$ reduces to a link diagram corresponds to the case of $n=1$ and $D_0$ reduces to the connected sum of $D_b$ and an elementary torus link, which is a base link diagram by Theorem \ref{connectedsumtheorem}. We leave it to our reader to verify that $-2+E(D_-)=s(D)-w(D)-3-2r^-(D)$, $-1+E(D_0)=s(D)-w(D)-1-2r^-(D)$, $-2+e(D_-)=-s(D)-w(D)+1+2r^+(D)$ and $-1+e(D_0)=-s(D)-w(D)+3+2r^+(D)$. It follows that $E(D)=-1+E(D_0)=s(D)-w(D)-1-2r^-(D)$ and $e(D)=-2+e(D_-)=-s(D)-w(D)+1+2r^+(D)$. In general, if the statement is true for $k=k_0\ge 1$, then for $k=k_0+1$, the same argument yields $-2+E(D_-)=s(D)-w(D)-3-2r^-(D)$, $-1+E(D_0)=s(D)-w(D)-1-2r^-(D)$, $-2+e(D_-)=-s(D)-w(D)+1+2r^+(D)$ and $-1+e(D_0)=-s(D)-w(D)+1+2r^+(D)+2k$. Thus we also have $E(D)=-1+E(D_0)=s(D)-w(D)-1-2r^-(D)$ and $e(D)=-2+e(D_-)=-s(D)-w(D)+1+2r^+(D)$.

Assume now that the statement of the theorem holds for some $n=n_0\ge 2$ and let us consider the case $n=n_0+1\ge 3$. 
The R-pattern contains more than two attaching circles, and removing the last string of Seifert circles results in a Type II attachment to $D$ with $n_0$ attaching Seifert circles. Let this link diagram be $D^\p$ and by the induction hypothesis,  $D^\p$ is a strong base link diagram. This way we can view $D$ as being obtained from $D^\p$ by attaching a string of $2k$ Seifert circles (with at least $2k+1$ lone crossings in the string), one end to $C$ and the other end to the last Seifert circle $C^{\p\p}$ in the R-pattern. Consider the case $k=1$ and apply (\ref{Skein1}) to one of the lone crossings in the attachment. Both $D_0=D^\p$ and $D_-$ (which looks like $D^\p$ with one additional crossing between $C^{\p\p}$ and $C$) can be obtained from $D_b$ by a type II attachment using $n_0$ attaching Seifert circles. Thus by the induction principle both $D_0$ and $D_-$ are base link diagrams. If we assume the $2k$ lone crossings in the string are positive (the negative case is similar) then we have
$w(D_-)=w(D^\p)+1=w(D)-2$, $s(D_-)=s(D^\p)=w(D)-2$, $r^+(D)=r^+(D^\p)+1$ and $r^-(D)=r^-(D^\p)$. We have $-2+E(D_-)=s(D^\p)-w(D^\p)-4-2r^-(D^\p)$, $-2+e(D_-)=-s(D^\p)-w(D^\p)-2+2r^+(D^\p)$, $-1+E(D_0)=s(D^\p)-w(D^\p)-2-2r^-(D^\p)$ and $-1+e(D_0)=-s(D^\p)-w(D^\p)+2r^+(D^\p)$. It follows that $E(D)=-1+E(D_0)=s(D)-w(D)-1-2r^-(D)$, $e(D)=-2+e(D_-)=-s(D)-w(D)+1+2r^+(D)$ and $D$ is a base link diagram. Furthermore we have $p^h_0(D)=zp_0^h(D^\p)$ and $p_0^\ell(D)=p_0^\ell(D^\p)$. Now use induction on $k$, we can easily show that $E(D)=s(D)-w(D)-1-2r^-(D)$, $e(D)=-s(D)-w(D)+1+2r^+(D)$ with $p^h_0(D)=zp_0^h(D^\p)$ and $p_0^\ell(D)=p_0^\ell(D^\p)$ in general. We note that these formulas do not change if the $2k$ lone crossings in the string are negative. This concludes the induction and the theorem is proved. 
\end{proof}

We now can easily prove Theorem \ref{TypeD_theorem}.
\begin{proof}
Let $D$ be obtained by a type II attachment operation of $n$ interlocked R-patterns on a strong base link diagram $D_b$.
Then we apply Theorem \ref{TypeD_theoremsingleR} $n$ times. However we have to attach the different R-patterns in a particular order. For example in Figure \ref{ValidAttachment} the interlocked R-pattern consists of five R-patterns. We start by attaching the R-pattern on the right and move from right to left through the Figure. 
\end{proof}


Theorem \ref{TypeD_theorem} can be extended to include the case of exception (ii) as well with an additional condition.
Let $D_b$ be a strong base link diagram with two Seifert circles $C$ and $C^\p$ in $D_b$ sharing multiple crossings. If an interlocked R-pattern with exception (ii) is attached to $C$ such that one end Seifert circle of the pattern is attached to $C^\p$ via a lone crossing that is ``borrowed" from one of the multiple crossings between $C$ and $C^\p$ as shown  in the right side of Figure  \ref{exceptions}. In other word, in the resulting diagram (after the pattern is attached), $C$ and $C^\p$ share one less crossings than they do in $D_b$. Let $D_b^\p$ denote the link diagram obtained from $D_b$ by deleting one crossing between $C$ and $C^\p$, then we have

\begin{theorem}\label{TypeD2_theorem}
If $D_b$ and $D_b^\p$ are both strong base link diagrams, then attaching an interlocked R-pattern with exception (ii) using $C$ and $C^\p$ as defined above results in a strong base link diagram.
\end{theorem}

\begin{proof}
Assume that the interlocked R-pattern consists of several R-patterns. We begin by attaching the R-pattern with the exception first and then apply Theorem \ref{TypeD_theoremsingleR} repeatedly to attach the other R-patterns.
To attach the R-pattern with the exception we  are again using induction on the number $n$ of attaching circles. Let $D_b$ and $D^\p_b$ be as defined above and $D$ be the resulting diagram after the R-pattern is attached. Furthermore we assume that all crossings of the last attachment string of the R-pattern with an exception are positive and leave the negative case to the reader. Note that this implies $r^-(D_b)=r^-(D^\p_b)$ and $r^+(D^\p_b)-r^+(D_b)\le 1$ (that is the reduction number of $D^\p_b$ may have increased by one, if the two Seifert circles used to in the definition of $D_b$ and $D^\p_b$ share only two crossings). Assume $n=2$, $C_1$ is attached to $C$ and attach $C_2$ to a different Seifert circle $C^\p$ by a lone crossing which is ``borrowed" from one of the original crossings between $C$ and $C^\p$. Let $2k$ be the number of Seifert circles in the string ($k\ge 1$) and let $m\ge 1$ be the number of crossings between $C_1$ and $C$. 

Consider first the case $m=1$ and $k=1$. Apply (\ref{Skein1}) to one of the lone crossings. $D_-$ reduces to $D_b$ and $D_0$ reduces to $D_b^\p$. We leave it to our reader to verify that $-1+E(D_0)=s(D)-w(D)-1-2r^-(D)$ and $-2+E(D_-)=s(D)-w(D)-3-2r^-(D)$. Thus $E(D)=s(D)-w(D)-1-2r^-(D)$ and $p_0^h(za^{-1}H(D_0))=z p_0^h(D_b^\p)$.
We also have $-1+e(D_0)=-s(D)-w(D)+5+2r^+(D_b^\p)>-2+e(D_-)=-s(D)-w(D)+1+2r^+(D)$. Thus $e(D)=-s(D)-w(D)+1+2r^+(D)$,  $p_0^\ell(D)=p_0^\ell(D_b)$ and $D$ is a base link diagram.

Now we use induction on $k$. Let $D_{k_0}$ be the base link diagram with $m=1$ and $2k_0$ Seifert circles attached to $D_b$. We assume that all $D_{k_0}$ with $k_0\le k>1$ are base link diagrams. 
Now apply (\ref{Skein1}) to one of the lone crossing in the diagram $D=D_{k+1}$ where $m=1$ and $2(k+1)$ Seifert circles are attached $D_b$. We have that $D_-=D_k$ and $D_0=D^\p_b$ and therefore $s(D_0)=s(D)-2(k+1)$,$w(D_0)=w(D)-2(k+1)-1$, $s(D_-)=s(D)-2$ and $w(D_-)=w(D)-2$. Furthermore $r^-(D)=r^-(D_0)=r^-(D_-)$, $r^+(D)=r^+(D_-)+1$ and $r^+(D_0)+k+1 \le r^+(D)\le r^+(D_0)+k+2$.
Now we can easily show that $-1+E(D_0)=s(D)-w(D)-1-2r^-(D)$ and $-2+E(D_-)=s(D)-w(D)-3-2r^-(D)$. Thus $E(D)=s(D)-w(D)-1-2r^-(D)$ and $p_0^h(za^{-1}H(D_0))=z p_0^h(D_b^\p)$.
We also have $-1+e(D_0)=-s(D)-w(D)+4(k+1)+2r^+(D_b^\p)>-2+e(D_-)=-s(D)-w(D)+1+2r^+(D)$. Thus $e(D)=-s(D)-w(D)+1+2r^+(D)$,  $p_0^\ell(D)=p_0^\ell(D_b)$ and $D$ is a base link diagram. This proves the case for $m=1$.

If $m=2$ for any $k\ge 1$, apply (\ref{Skein1}) to one of the two crossings between $C_1$ and $C$. $D_-$ reduces to $D_b^\p$ and $D_0$ corresponds to the case $m=1$. 
Therefore $s(D_0)=s(D)=s(D_-)+2k$,$1+w(D_0)=w(D)=w(D_-)+2(k+1)$. Furthermore $r^-(D)=r^-(D_0)=r^-(D_-)$.
Notice that $r^+(D)=r^+(D_0)$ since the cycle of Seifert Circles of length four containing $C_1$ and $C$ has reduction number of one regardless of how many crossings are connecting $C_1$ and $C$. We also have $r^+(D_-)+k \le r^+(D)\le r^+(D_-)+k+1$. Now we can easily show that $-1+E(D_0)=s(D)-w(D)-1-2r^-(D)=-2+E(D_-)$. However $p_0^h(a^{-2}H(D_-))=p_0^h(D_b^\p)$ and $p_0^h(za^{-1}H(D_0))=z^2p_0^h(D_b^\p)$ by the results for $m=1$. Thus we have $E(D)=s(D)-w(D)-1-2r^-(D)$ with $p_0^h(D)=z^2p_0^h(D_b^\p)$.
We also have $-1+e(D_0)=-s(D)-w(D)+1+2r^+(D)$ and $-2+e(D_-)=-s(D)-w(D)+4k+1+2r^+(D^\p_b)\ge-s(D)-w(D)+2k-1+2r^+(D)> -1+e(D_0)$.
Thus $e(D)=-s(D)-w(D)+1+2r^+(D)$,  $p_0^\ell(D)=z p_0^\ell(D_b)$ and $D$ is a base link diagram. This proves the case for $m=2$.

Now assume that $D_{m,k}$ is the diagram created by attaching a string of $2k$ Seifert circles to $D$ where the first Seifert circle $C_1$ is attached with $m\ge 3$ crossings to $C$ and the last Seifert circle is attached to a different Seifert circle $C^\p$ by a lone crossing which is ``borrowed" from one of the original crossings between $C$ and $C^\p$. In addition we assume that all $D_{m,k}$ are base link diagrams for all $k$ and all $m\le m_0$. Furthermore we assume that $p_0^h(D_{m,k})=z^m p_0^h(D_b^\p)$ and $p_0^\ell(D_{m,k})=z^{m-1} p_0^\ell(D_b)$. Now let $D= D_{m_0+1,k}$ be given and apply (\ref{Skein1}) to one of the $m_0+1$ crossings between $C_1$ and $C$. We now have $D_0 = D_{m_0,k}$ and $D_- = D_{m_0-1,k}$. Therefore $s(D_0)=s(D)=s(D_-)$, $1+w(D_0)=w(D)=w(D_-)+2$. Furthermore $r^-(D)=r^-(D_0)=r^-(D_-)$ and $r^+(D)=r^+(D_0)=r^+(D_-)$. We have that $-1+E(D_0)=-2+E(D_-)=s(D)-w(D)-1-2r^-(D)$, $p_0^h(D_0)=z^{m_0} p_0^h(D_b^\p)$ and $p_0^h(D_-)=z^{m_0-1} p_0^h(D_b^\p)$.
Therefore $E(D)=s(D)-w(D)-1-2r^-(D)$ and $p_0^h(D)=z^{m_0+1} p_0^h(D_b^\p)$.
We also have that $-1+e(D_0)=-2+e(D_-)=-s(D)-w(D)+1-2r^+(D)$, $p_0^\ell(D_0)=z^{m_0-1} p_0^\ell(D_b)$ and $p_0^\ell(D_-)=z^{m_0-2} p_0^\ell(D_b)$.
Therefore $E(D)=s(D)-w(D)-1-2r^-(D)$ and $p_0^\ell(D)=z^{m_0} p_0^h(D_b)$.  This concludes the case for $n=2$.

Assume now that the statement of the theorem holds for some $n=n_0\ge 2$ and let us consider the case $n=n_0+1\ge 3$. 
The R-pattern with exception contains more than two attaching circles, and removing the initial string of Seifert circles (that does not contain the exception) results in a Type II attachment to $D$ with $n_0$ attaching Seifert circles. Let this link diagram be $D^\p$ and by the induction hypothesis,  $D^\p$ is a strong base link diagram. This way we can view $D$ as being obtained from $D^\p$ by attaching a string of $2k$ Seifert circles (with at least $2k+1$ lone crossings in the string), one end to $C$ and the other end to the first Seifert circle $C^{\p\p}$ in the R-pattern. Consider the case $k=1$ and apply (\ref{Skein1}) to one of the lone crossings in the attachment. Both $D_0=D^\p$ and $D_-$ (which looks like $D^\p$ with one additional crossing between $C^{\p\p}$ and $C$) can be obtained from $D_b$ by a type II attachment of a R-pattern with exception using $n_0$ attaching Seifert circles. Thus by the induction principle both $D_0$ and $D_-$ are base link diagrams. If we assume the $2k$ lone crossings in the string are positive (the negative case is similar) then we have
$w(D_-)=w(D^\p)+1=w(D)-2$, $s(D_-)=s(D^\p)=w(D)-2$, $r^+(D)=r^+(D^\p)+1$ and $r^-(D)=r^-(D^\p)$. We have $-2+E(D_-)=s(D^\p)-w(D^\p)-4-2r^-(D^\p)$, $-2+e(D_-)=-s(D^\p)-w(D^\p)-2+2r^+(D^\p)$, $-1+E(D_0)=s(D^\p)-w(D^\p)-2-2r^-(D^\p)$ and $-1+e(D_0)=-s(D^\p)-w(D^\p)+2r^+(D^\p)$. It follows that $E(D)=-1+E(D_0)=s(D)-w(D)-1-2r^-(D)$, $e(D)=-2+e(D_-)=-s(D)-w(D)+1+2r^+(D)$ and $D$ is a base link diagram. Furthermore we have $p^h_0(D)=zp_0^h(D^\p)$ and $p_0^\ell(D)=p_0^\ell(D^\p)$. Now use induction on $k$, we can easily show that $E(D)=s(D)-w(D)-1-2r^-(D)$, $e(D)=-s(D)-w(D)+1+2r^+(D)$ with $p^h_0(D)=zp_0^h(D^\p)$ and $p_0^\ell(D)=p_0^\ell(D^\p)$ in general. We note that these formulas do not change if the $2k$ lone crossings in the string are negative. This concludes the induction and the theorem is proved. 
\end{proof}

\section{Applications and examples}\label{s5}

A formula for the braid index of two bridge links already exists and can be found in a standard textbook on knot theory \cite{Crom}. However this formula is based on particular diagrams of the of the two bridge link that are often non minimal diagrams. In this section, we offer a new approach that is always based on a minimal diagram of the two bridge link and that can be extended to a larger class of links - namely the alternating Montesions links.

\medskip
\subsection{Application to two bridge links}
\label{twobridgelinks}

Let $K=b(\alpha,\beta)$ be a two-bridge link (or 4-plat or rational link), where $0<\beta<\alpha$
and $\alpha$, $\beta$ are co-prime integers. A vector $(a_1,a_2,...,a_n)$
is called a {\it standard continued fraction decomposition} of $\frac{\beta}{\alpha}$ if $n$ is odd and all $a_i>0$ and
$$
\frac{\beta}{\alpha}=\frac{1}{a_1+\frac{1}{a_{2}+\frac{1}{.....\frac{1}{a_n}}}}. 
$$
It may be necessary to allow $a_n=1$ in order to guarantee that the length of vector $(a_1,a_2,...,a_n)$ is odd and under these conditions the standard continued fraction expansion of $\frac{\beta}{\alpha}$ is unique.
A standard diagram of a two bridge link $b(\alpha,\beta)$ given by the vector $(a_1,a_2,...,a_n)$ is shown in Figure \ref{2bridgeone}, where the rightmost block of crossings corresponds to the $a_1$ entry. Such a diagram is also called a standard 4-plat diagram. Furthermore, without loss of generality for a standard diagram we will assign the component corresponding to the long arc at the bottom of Figure \ref{2bridgeone} the orientation as shown. This is immaterial in the case of a two bridge knot since two bridge knots and links are invertible. In the case when the two bridge link has two components, there are two choices for the orientation of the other component. We will assume that an orientation has been given to the other component, but there is no need for us to specify which one since our goal is to develop a formula for the braid index that works for both cases. We also note that both a two bridge knot or link $L$ and its mirror image $\bar{L}$ have standard continued fraction decomposition using only positive values in the vector $(a_1,a_2,...,a_n)$.

%
\begin{figure}[htb!]
\includegraphics[scale=1]{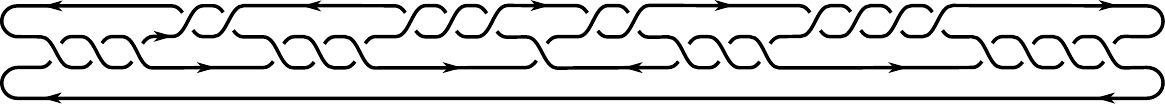}
\caption{The two bridge link $b(17426,4117)$ given by $(4,4,3,2,1,3,3,2,3)$.}
\label{2bridgeone}
\end{figure}

Since all crossings corresponding to a given $a_i$ have the same crossing sign under the given orientation, we will define a signed vector $(b_1,b_2,...,b_n)$ where $b_i =\pm a_i$ with its sign given by the crossing sign of the crossings corresponding to $a_i$. For example, for $K=b(17426,4117)$ with the orientation shown in Figure \ref{2bridgeone} we obtain the signed vector $(-4,4,-3,-2,-1,3,3,2,3)$. Let us group the consecutive $b_j$'s with the same signs together and call these groups {\em blocks} denoted by $B_i$. For example, we have four blocks $B_1=(-4)$, $B_2=(4)$,  $B_3=(-3,-2,-1)$ and $B_4=(3,3,2,3)$ for the link given in Figure \ref{2bridgeone}. 

Next we consider the Seifert circles of the standard diagram $D$ of a two bridge link $K=(b_1,b_2,...,b_n)$ (where $n=2k+1$ is odd), see Figure \ref{2bridgeone} for an example. Let $C$ be the Seifert circle that contains the long arc on the bottom of the diagram. $D$ can be realized as a Type II attachment to $C$ satisfying the following conditions: 

\noindent
\begin{itemize}
\item Each R-pattern attached to the outside of $C$ corresponds to a positive block in $(b_1,b_2,...,b_n)$.

\item Each R-pattern attached to the inside of $C$ corresponds to a negative block in $(b_1,b_2,...,b_n)$.

\item The R-patterns are interlocked.

\item  Each $b_{2j+1}>0$, and each $b_{2j}<0$ corresponds to crossings between a Seifert circle in an R-pattern and $C$. That is, each of these corresponds to an attaching circle that is attached with $|b_j|$ crossings to $C$.

\item  For each $b_{2j}>0$ (and each $b_{2j+1}<0$) there are $|b_{2j}|$ ($|b_{2j+1}|$) lone crossings between the attaching Seifert circles.

\item Each positive block starts and ends with a positive $b_{2j}$ unless it is the first or the last block in $(b_1,b_2,...,b_n)$.

\item Each negative block starts and ends with a negative  $b_{2j+1}$.

\item A middle positive block either contains a single positive $b_{2j}$ that is even, or is of the form $(b_{2j},b_{2j+1},...,b_{2j+2j^\p})$ where $j^\p\ge 1$ and both $b_{2j}$ and $b_{2j+2j^\p}$ are odd.

\item A middle negative block either contains a single negative $b_{2j+1}$ that is even, or is of the form $(b_{2j+1},b_{2j+1},...,b_{2j+2j^\p+1})$ where $j^\p\ge 1$ and both $b_{2j+1}$ and $b_{2j+2j^\p+1}$ are odd.
\end{itemize}



The above statements can be explained as follows:
Since the orientation of $C$ is fixed by the orientation of the long arc at the bottom of a standard diagram, it is easy to see that the crossings correspond to positive $b_{2j+1}$'s and negative $b_{2j}$'s must be smoothed in the direction parallel (but with opposite direction) to the long arc at the bottom as shown in Figure \ref{2bridgeone}. Thus the positive $b_{2j+1}$'s and negative $b_{2j}$'s contribute to the ``medium"  Seifert circles (namely the attaching circles in the R-patterns). The orientation of the arcs in the negative $b_{2j+1}$'s and positive $b_{2j}$'s causes the smoothing in the direction ``vertical" to the large Seifert circle $C$ as shown in Figure \ref{2bridgetwo}. They will form the lone crossings in the $R$ patterns. Thus all crossings with a positive sign are on the outside of the large Seifert circle $C$ while all crossings with a negative sign are on the inside of $C$. It is thus clear that these must form blocks of crossings of the same sign, corresponding to the blocks in $(b_1,b_2,...,b_n)$ and each block $B_i$ is an R-pattern attached to $C$.
The R-patterns are interlocked because each time we switch the signs of the $b_j$ (that is we switch from a block $B_i$ to $B_{i+1}$) the large Seifert circle $C$ switches its position in the 4-plat diagram between the top string and the second string counting from the bottom, see Figure \ref{2bridgetwo}. These causes exactly one lone crossing of the block $B_i$ to be ``interlocked" with the next block $B_{i+1}$.
Furthermore if we move from right to left along the top string of the Seifert circle $C$ then when a positive (negative) bock $B_i$ ends $C$ must move from the second string (top string)) to the top string (second string from the bottom), thus the positive (negative) block $B_i$ must end with a positive $b_{2j}$ (negative $b_{2j+1}$) and the negative (positive) block $B_{i+1}$ must start with a negative $b_{2j+1}$ (positive $b_{2j}$), see Figure \ref{orientation_rational}. Thus we can reconstruct $D$ by attaching the R-patterns to $C$ one at a time as we move from right to left through the top portion of $C$ as depicted in Figure \ref{2bridgetwo}. We can use the above information to prove the following theorem.


\begin{figure}[htb!]
\includegraphics[scale=1]{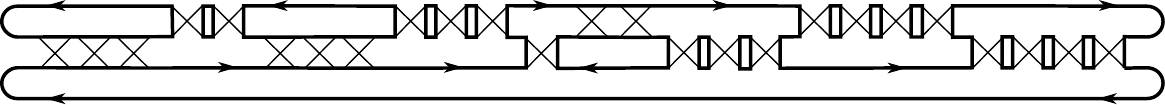}
\caption{The Seifert circle decomposition of the two bridge link in Figure \ref{2bridgeone}. We see that the patterns interlock with a lone crossing along $C$ between $B_1$ and $B_2$, $B_2$ and $B_3$, and between $B_3$ and $B_4$. We also note that there are 15 Seifert circles: the large Seifert circle $C$, 3 medium Seifert circles and 11 small Seifert circles.
}
\label{2bridgetwo}
\end{figure}

\begin{theorem}\label{2bridge_theorem} 
Let $K=b(\alpha,\beta)$ be a two bridge link diagram with signed vector  $(b_1,b_2,...,b_{2k+1})$ in the normal form, then the braid index is given by
\begin{equation}\label{2bridgeformula}
\textbf{b}(K)=1+\frac{2+\sign(b_1)+\sign(b_{2k+1})}{4}+\sum_{b_{2j}>0,1\le j\le k}\frac{b_{2j}}{2}+\sum_{b_{2j+1}<0,0\le j\le k}\frac{|b_{2j+1}|}{2}.
\end{equation}
\end{theorem}

\begin{proof} 
By Theorem \ref{TypeD_theorem}, $\textbf{b}(K)$ is given by the total number of Seifert circles in it minus its reduction number. Let us consider the following (exhaustive list of) possibilities for the contributions of a block to the total count of Seifert circles in $K$ and to the total reduction number of $K$. Formula (\ref{2bridgeformula}) follows once we combine these cases.

(1) $K$ contains a single R-pattern with only positive crossings. By the counting given in the paragraph before Theorem \ref{TypeD_theorem}, the total number of Seifert circles in the R-pattern is $1+\sum_{b_{2j}>0,1\le j\le k}{b_{2j}}$ and its reduction number is $\sum_{b_{2j}>0,1\le j\le k}\frac{b_{2j}}{2}$. Since $C$ is the only Seifert circle in the base link diagram, we have $\textbf{b}(K)=2+\sum_{b_{2j}>0,1\le j\le k}\frac{b_{2j}}{2}$, which is the same as (\ref{2bridgeformula}) as there are no negative blocks.

(2) $(b_1,b_2,...,b_{2k+1})$ contains more than one block and the first block $B_1=(b_1,b_2,...,b_{2j})$ is positive. In this case it is necessary that $b_{2j}$ is odd, see Figure \ref{orientation_rational}. (If $b_{2j}$ is even then there will be a cycle of odd length in the Seifert graph of $K$, which is not possible since the Seifert graph is bipartite.) In this case $B_1$ contributes a total of $\sum_{1\le i\le j}{b_{2i}}$ Seifert circles with a reduction number $(-1+\sum_{1\le i\le j}{b_{2i}})/2$. Thus the contribution of $B_1$ to $\textbf{b}(K)$ is $(1+\sum_{1\le i\le j}{b_{2i}})/2=\frac{1+\sign(b_1)}{4}+\sum_{1\le i\le j}({b_{2i}}/2)$. 

(3) Similarly, if $(b_1,b_2,...,b_{2k+1})$ contains more than one block and the last block $B_m$ (which is of the form $(b_{2j^\p},b_{2j^\p+1},...,b_{2k+1})$) is positive, then $B_m$ contributes $\frac{1+\sign(b_{2k+1})}{4}+\sum_{j^\p\le i\le k}({b_{2i}}/2)$ to $\textbf{b}(K)$.

(4) In all other cases a positive block $B_{j_1}$ either contains a single term $b_{2j}>0$ with $b_{2j}$ being even, or it is of the form $B_{j_1}=(b_{2j},b_{2j+1},...,b_{2j+2j_1})$ with $b_{2j}$ and $b_{2j+2j_1}$ both odd. We leave it to our reader to verify that the contribution of $B_{j_1}$ to $\textbf{b}(K)$ is $\sum_{0\le i\le j_1}({b_{2j+2i}}/2)$ (notice that this includes the case $j_1=0$).

(5) Similarly, any negative block $B_{j_2}=(b_{2j^\p+1},b_{2j^\p+2},...,b_{2j^\p+1+2j_2})$ (with the possibility that $j_2=0$) contributes a term of the form 
$\sum_{0\le i\le j_2}({b_{2j^\p+1+2i}}/2)$ to $\textbf{b}(K)$. Notice that in the case when $B_{j_2}$ is the first or the last block, then the terms $\frac{1+\sign(b_{1})}{4}$ or $\frac{1+\sign(b_{2k+1})}{4}$ will be zero and do not change formula (\ref{2bridgeformula}).
\end{proof}

\begin{remark}{\em
If a minimum two bridge link diagram $K$ is not represented in its standard form as shown in Figure \ref{2bridgeone}, but rather in the mirror image of Figure \ref{2bridgeone}, then a slight modification of the above proof leads to the following formulation, assuming that $(b_1^\p,b_2^\p,...,b_{2k+1}^\p)$ is the signed vector of $K$ in this presentation:
\begin{equation}\label{2bridgeformula_2}
\textbf{b}(K)=1+\frac{2-\sign(b_1^\p)-\sign(b_{2k+1}^\p)}{4}+\sum_{b^\p_{2j}<0,1\le j\le k}\frac{|b^\p_{2j}|}{2}+\sum_{b^\p_{2j+1}>0,0\le j\le k}\frac{b_{2j+1}^\p}{2}.
\end{equation} 
In particular, if $K$ has signed vector $(b_1,b_2,...,b_{2k+1})$ when it is in the standard form, then the above formula applies to its mirror image with signed vector $(b_1^\p,b_2^\p,...,b_{2k+1}^\p)=(-b_1,-b_2,...,-b_{2k+1})$. We have
\begin{eqnarray*}
\textbf{b}(\overline{K})&=&1+\frac{2-\sign(b_1^\p)-\sign(b_{2k+1}^\p)}{4}+\sum_{b^\p_{2j}<0,1\le j\le k}\frac{|b^\p_{2j}|}{2}+\sum_{b^\p_{2j+1}>0,0\le j\le k}\frac{b_{2j+1}^\p}{2}\\
&=&1+\frac{2+\sign(b_1)+\sign(b_{2k+1})}{4}+\sum_{b_{2j}>0,1\le j\le k}\frac{b_{2j}}{2}+\sum_{b_{2j+1}<0,0\le j\le k}\frac{|b_{2j+1}|}{2}\\
&=&\textbf{b}(K),
\end{eqnarray*}
as expected.
}
\end{remark}

\begin{figure}[htb!]
\includegraphics[scale=.5]{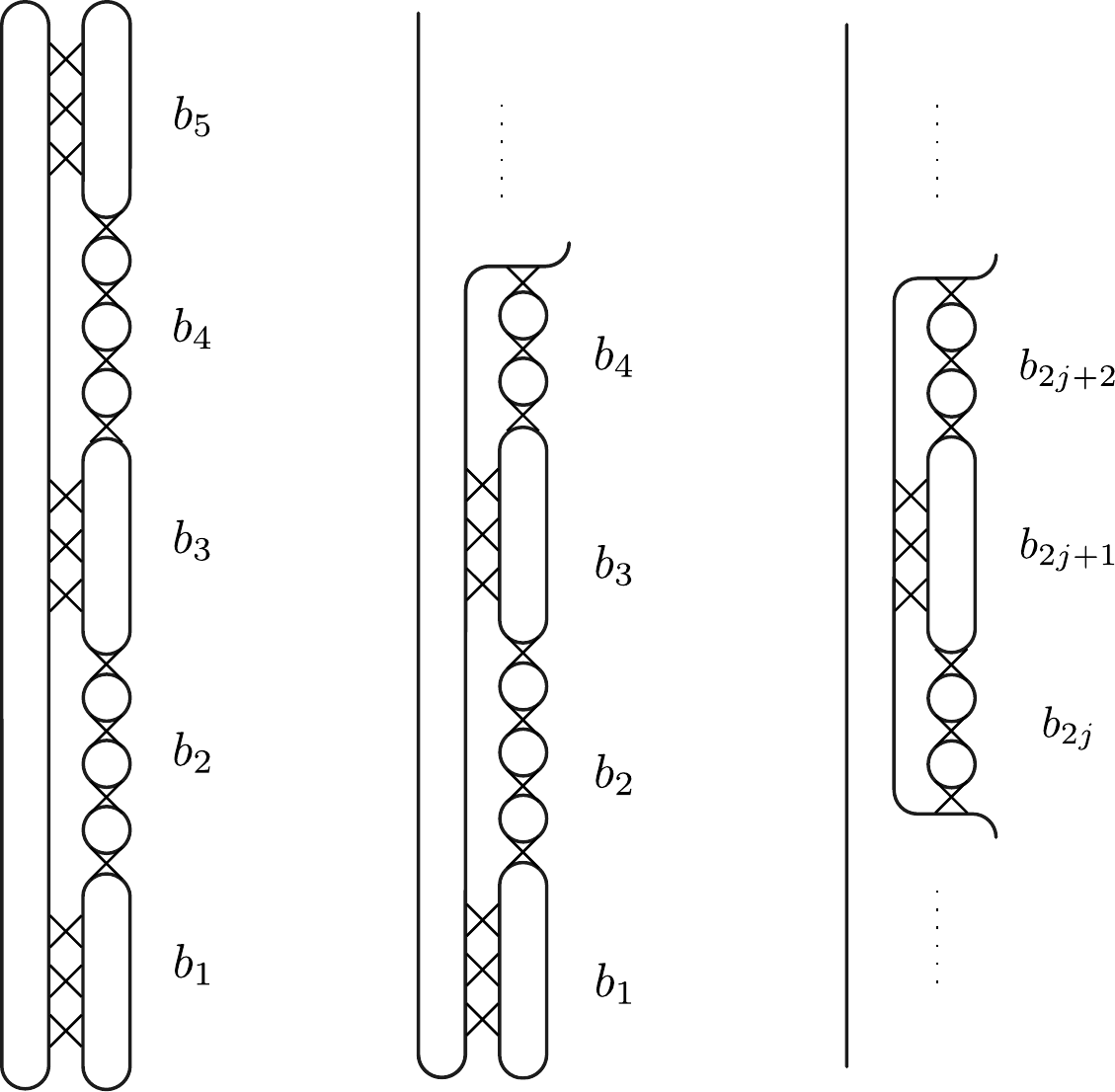}
\caption{From left to right: 1. $(b_1,b_2,...,b_{2k+1})$ consists of a single positive block; 2. A positive first block followed by a negative block must end with an odd $b_{2j}$; 3. A middle positive block must start and end with odd $b_{2j}$'s.
}
\label{orientation_rational}
\end{figure}

\begin{example}{\em
Consider the two bridge link with signed vector $(-4,4,-3,-2,-1,3,3,2,3)$ given in Figure \ref{2bridgeone}. We have 
\begin{eqnarray*}
\textbf{b}(K)&=&1+(2+\sign(b_1)+\sign(b_{2k+1}))/4+\sum_{b_{2j}>0,1\le j\le k}b_{2j}/2+\sum_{b_{2j+1}<0,0\le j\le k}|b_{2j+1}|/2\\
&=&1+1/2+(2+3+4)+(1+3+4)/2=10.
\end{eqnarray*}
If we reverse the orientation of one component in the link then $b(17426,4117)$ has a signed vector
$(4,4,3,2,1,3,-3,-2,-3)$. In which case we have
\begin{eqnarray*}
\textbf{b}(K)&=&1+(2+\sign(b_1)+\sign(b_{2k+1}))/4+\sum_{b_{2j}>0,1\le j\le k}b_{2j}/2+\sum_{b_{2j+1}<0,0\le j\le k}|b_{2j+1}|/2\\
&=&1+1/2+(3+2+4)/2+(3+3)/2=9.
\end{eqnarray*}}
\end{example}

\begin{remark}{\em 
A different formulation of the braid index of a two bridge link was obtained by Murasugi \cite{Mu} using an even decomposition of the rational number $\beta/\alpha$ that defines the link. In a future paper we shall establish the relationship between the two formulations using a direct combinatorical approach \cite{DEH2018}. One advantage of the formula in Theorem \ref{2bridge_theorem}  is that it uses a minimal diagram of a rational link that is directly based on the Conway notation in the knot table. More significantly, the real advantage of this formulation it that it provides crucial step toward the complete formulation of the braid index of alternating Montesinos links, see the next subsection.}
\end{remark}

\subsection{Application to alternating Montesinos links}
\underline{The definition of a Montesinos link.} The theorems we proved in Section \ref{s4}, as well as in the last subsection, have prepared us to tackle a much larger family of links, namely the (oriented) alternating Montesinos links. In general, a Montesinos link $L=M(\beta_1/\alpha_1,\ldots, \beta_k/\alpha_k,e)$ is a link with a diagram as shown in Figure \ref{Montesinos}, where each diagram within a topological circle (which is only for the illustration and not part of the diagram) is a rational tangle $A_j$ that corresponds to some rational number $\beta_j/\alpha_j$ with $|\beta_j/\alpha_j|<1$ and $1\le j\le k$ for some positive integer $k$, and $e$ is an integer that stands for an arbitrary number of half-twists, see Figure \ref{Montesinos}.
If the Montesinos link is alternating then all fractions $\beta_j/\alpha_j$ have the same sign and this is matched by the sign of $e$  representing the $|e|$ half-twists. As in the case of two bridge knots the sign of $e$ and the $\beta_j/\alpha_j$ should not be confused with the sign of individual crossings, i.e. for example the signs of the crossings represented by $e$ may not coincide with the sign of $e$. The sign of the crossings represented by $e$ dependents on the orientation that the two strings in the $e$-half twists have.
Since we are only concerned with the braid index of an alternating Montesinos link in this paper, we will assume that $\beta_j/\alpha_j>0$ for each $j$ and that the crossings in the tangle diagrams are as chosen in a standard drawing of two bridge links - for more details see below. Notice that if $k=2$ then the Montesinos link is actually a 2-bridge link. However we shall not require that $k\ge 3$ since our formula will hold for the case $k=2$ as well. 
A classification of Montesinos links exists including both alternating and non-alternating Montesinos links and can be found in \cite{B}.
Without loss of generality, we will assume that the top long strand in a Montesinos link diagram is oriented from right to left as shown in Figure \ref{Montesinos} since reversing the orientations of all components in a link does not change its braid index. 

\begin{figure}[htb!]
\includegraphics[scale=.4]{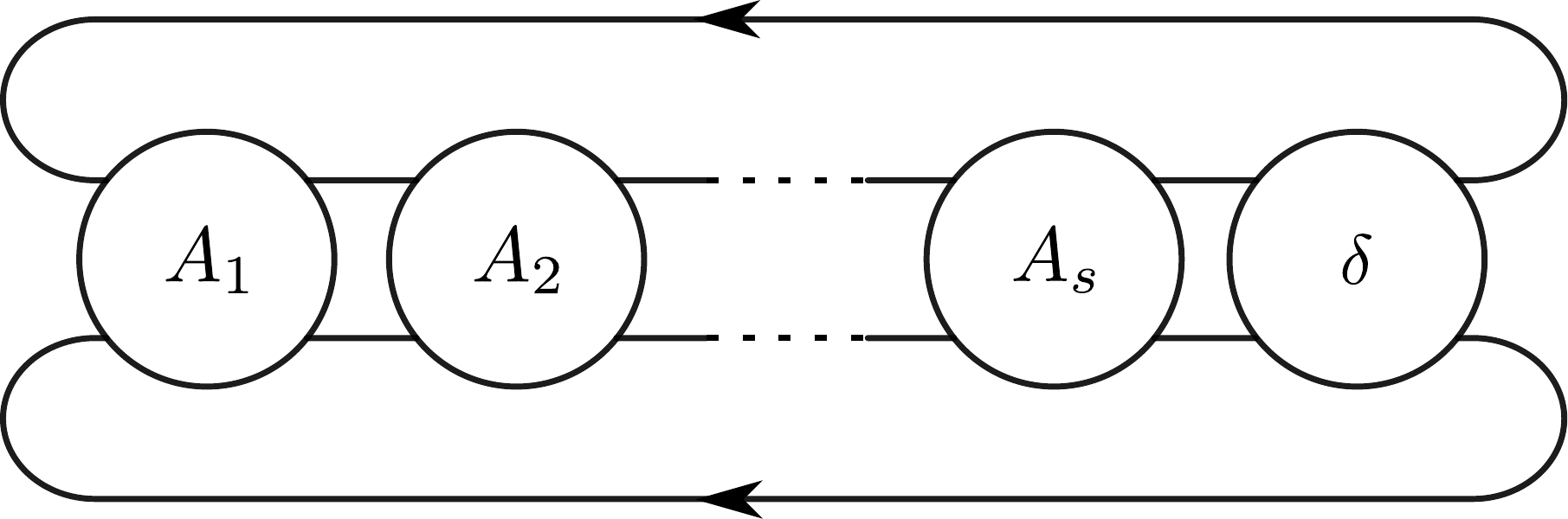}
\caption{A diagram depicting a general Montesinos link with $s$ rational tangles and $\delta$ horizontal half-twists.}
\label{Montesinos}
\end{figure}

\noindent
\underline{Standard diagrams, notations and terminology.} We will use a standard drawing for each rational tangle $A_j$ which is given by the continued fraction of the rational number $\beta_j/\alpha_j$ and contains an odd number of positive entries, exactly like what we did in the case of two bridge links in the last section. That is, we assume that $0<\beta_j<\alpha_j$ and $\beta_j/\alpha_j$ has a continued fraction decomposition of the form $(a_{1}^j,a^j_2,...,a_{2q_j+1}^j)$. The four strands that entering/exiting each tangle are marked as NW, NE, SW and SE. One example is shown at the left of Figure \ref{tangle}. Here we note that $a_{2q_j+1}^j$ is allowed to equal one if needed to make the vector $(a_{1}^j,a^j_2,...,a_{2q_j+1}^j)$ of odd length. We have
$$
\frac{\beta_j}{\alpha_j}=\frac{1}{a_{1}^j+\frac{1}{a_{2}^j+\frac{1}{.....\frac{1}{a_{2q_j+1}^j}}}}. 
$$
The closure of a rational tangle is obtained by connecting its NW and SW end points by a strand and connecting its NE and SE end points with another strand (as shown at the left side of Figure \ref{tangle}). This closure is called the denominator $D(A_j)$ of the rational tangle $A_j$. Notice that $D(A_j)$ results in a normal standard diagram of  the two bridge link $K(\alpha_j, \beta_j)$ given by the vector $(a_{1}^j,a^j_2,...,a_{2q_j+1}^j)$ (as shown at the right side of Figure \ref{tangle}) and discussed in the previous subsection. We note that the requirement that $\beta_j/\alpha_j<1$ means that each rational tangle ends with a vertical row of twists. This requirement guarantees a unique number of $e$ horizontal twists in the diagram of the Montesinos link. Finally, we define $(b_{1}^j,b^j_2,...,b_{2q_j+1}^j)$ as a signed vector similarly to the last subsection: $|b^j_m|=a^j_m$ with its sign matching the signs of the corresponding crossings in $L$ under the given orientation of $L$. We will use the notation $A_j(b_{1}^j,b^j_2,...,b_{2q_j+1}^j)$ to denote the tangle $A_j$ and the signed vector associated with it. We note that crossing signs cannot be determined by only looking at the fraction $\beta_j/\alpha_j$. One needs to take into account the orientation of the tangle $\beta_j/\alpha_j$ inherited from the orientation of $L$. 

\begin{figure}[htb!]
\includegraphics[scale=.4]{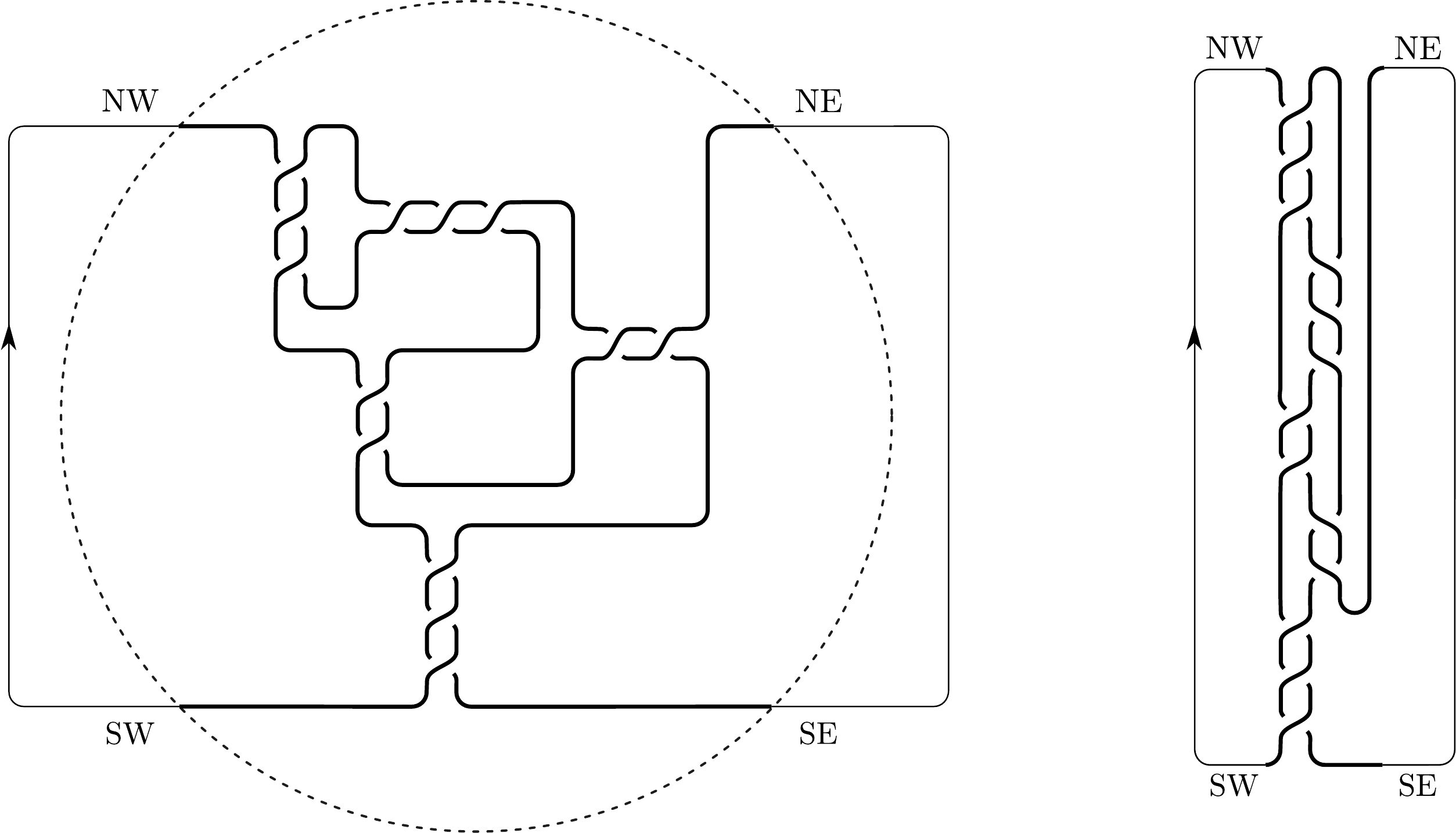}
\caption{Left: A standard drawing of the rational tangle $56/191 = A(3,2,2,3,3)$; Right: The denominator $D(A(3,2,2,3,3))$ is a standard diagram of the two bridge knot $K(191,56)$.}
\label{tangle}
\end{figure}

\noindent
\underline{Seifert circle decomposition of $L$.} Let us now consider the Seifert circle decomposition of $L$ by first examining how the arcs of Seifert circles entering and exiting each $A_j(b_{1}^j,b^j_2,...,b_{2q_j+1}^j)$ might look. Figure \ref{decomp} lists all eight possibilities for these arcs. Of course there might be lots of smaller complete Seifert circles within each tangle, but these are not shown in Figure \ref{decomp}. Observing (from Figure \ref{tangle}) that the SW--NW and SE--NE strands meet at the first crossing in $b_{1}^j$, therefore if these two strands belong to two different Seifert circles, then they must have parallel orientation (and the corresponding crossings are negative). Thus (vi) and (viii) are not possible. Furthermore, since we have assigned the top long arc in the Montesinos link diagram the orientation from right to left, (iii) is not possible either. We say that $A_j$ is of {\em Seifert Parity 1} if it decomposes as (i) in Figure \ref{decomp}, of {\em Seifert Parity 2} if it decomposes as (ii) or (iv) in Figure \ref{decomp} and of {\em Seifert Parity 3} if it decomposes as (v) or (vii) in Figure \ref{decomp}. Notice that $A_j$ is of Seifert Parity 3 if and only if $b_1^j>0$. Also, the Seifert Parity of a tangle $A_j$ depends on the orientation it inherits from $L$ and it should not be confused with the term parity of a tangle, which refers how the arcs in a tangle are connected and is a property that depends on the tangle itself alone.

\begin{figure}[htb!]
\includegraphics[scale=.4]{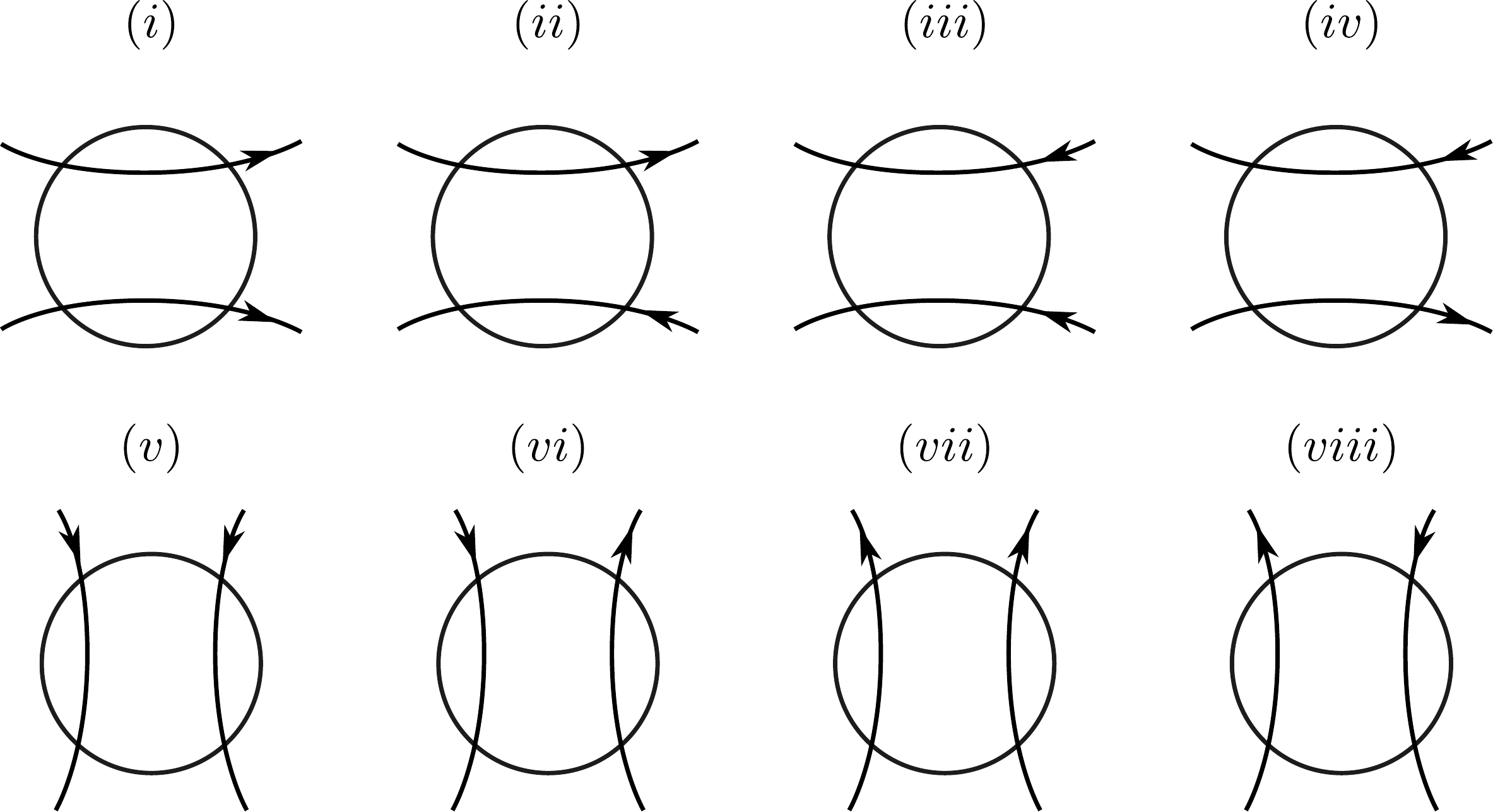}
\caption{Of the eight cases listed, (iii), (vi) and (viii) are not possible.}
\label{decomp}
\end{figure}

Thus, it becomes clear that the Seifert circle decomposition of $L$ contains the following: the Seifert circle(s) that contain the top and bottom long strands in $L$, the Seifert circles that do not contain these long strands, but contain strands that entering/exiting  one or more tangles, and Seifert circles within a tangle corresponding to the medium and small Seifert circles as defined in the last subsection. For the sake of convenience, we will call these {\em huge, large, medium} and {\em small} Seifert circles. $L$ can be classified into one of the following three classes. 

Class M1. The top long strand and the bottom long strand belong to two different huge Seifert circles and the bottom long strand also has the orientation from right to left. Notice that if $L$ is of Class M1, then every $A_j$ is of Seifert Parity 1 and all crossings in $e$ (if there are any) have negative signs. On the other hand, if one of the $A_j$'s is of Seifert Parity 1 or the crossings in $e$ are negative, then $L$ must be of Class M1. 

Class M2. The top long strand and the bottom long strand belong to two different huge Seifert circles and the bottom long strand has the orientation from left to right. Notice that $L$ is of Class M2 if and only if every $A_j$ is of Seifert Parity 2 (more precisely case (ii) in Figure \ref{decomp}) and in this case $e=0$. However, the condition $e=0$ and one of the $A_j$'s is of Seifert Parity 2 does not necessarily mean that $L$ is of Class M2.

Class B. The top long strand and the bottom long strand belong to the same huge Seifert circle. Notice that $L$ is of Class B if and only if at least one $A_j$ is of Seifert Parity 3 and all crossings in $e$ (if there are any) have positive signs.

Let $L=M(\beta_1/\alpha_1,\ldots, \beta_k/\alpha_k,e)$ be given such that $0<\beta_j/\alpha_j<1$ and that $e$ is positive (in the sense of the standard diagram drawing, not the crossing signs). In the following we will explain how to construct an alternating Montesinos link from a strong base link diagram by attaching interlocked R-patterns, which then leads us to the conclusion that all alternating Montesinos link diagrams are base link diagrams.

\noindent
\underline{Class M1 case.} Here every tangle $A_j$ is of Seifert Parity 1 and $L$ can be constructed from a Type M1 strong base link diagram $M_1$ with $k+\delta$ strings attached where all crossings in the strings are negative. Notice that in this case we have $b_1^j<0$ for all $1\le j\le k$ and all the $\delta$ crossings in $e$ are negative. More precisely, in the last $\delta$ strings of $M_1$, each one contains a single negative crossing (corresponding to the crossings in $e$). If $b_1^j$ is odd, then $b_2^j$ must be positive (in order for $A_j$ to have Seifert Parity 1), and the crossings in $b_1^j$ and $b_2^j$ smooth as shown in the right side of Figure \ref{TypeII(i)attachment}. In this case the $j$-th string contains exactly $|b_1^j|$ negative lone crossings, a valid Type M1 string. It is easy to see in this case that the rest of the tangle is attached to the top huge Seifert circle $C_1$ as Type II(i) attachment, see Figure \ref{TypeII(i)attachment}

\begin{figure}[htb!]
\includegraphics[scale=.4]{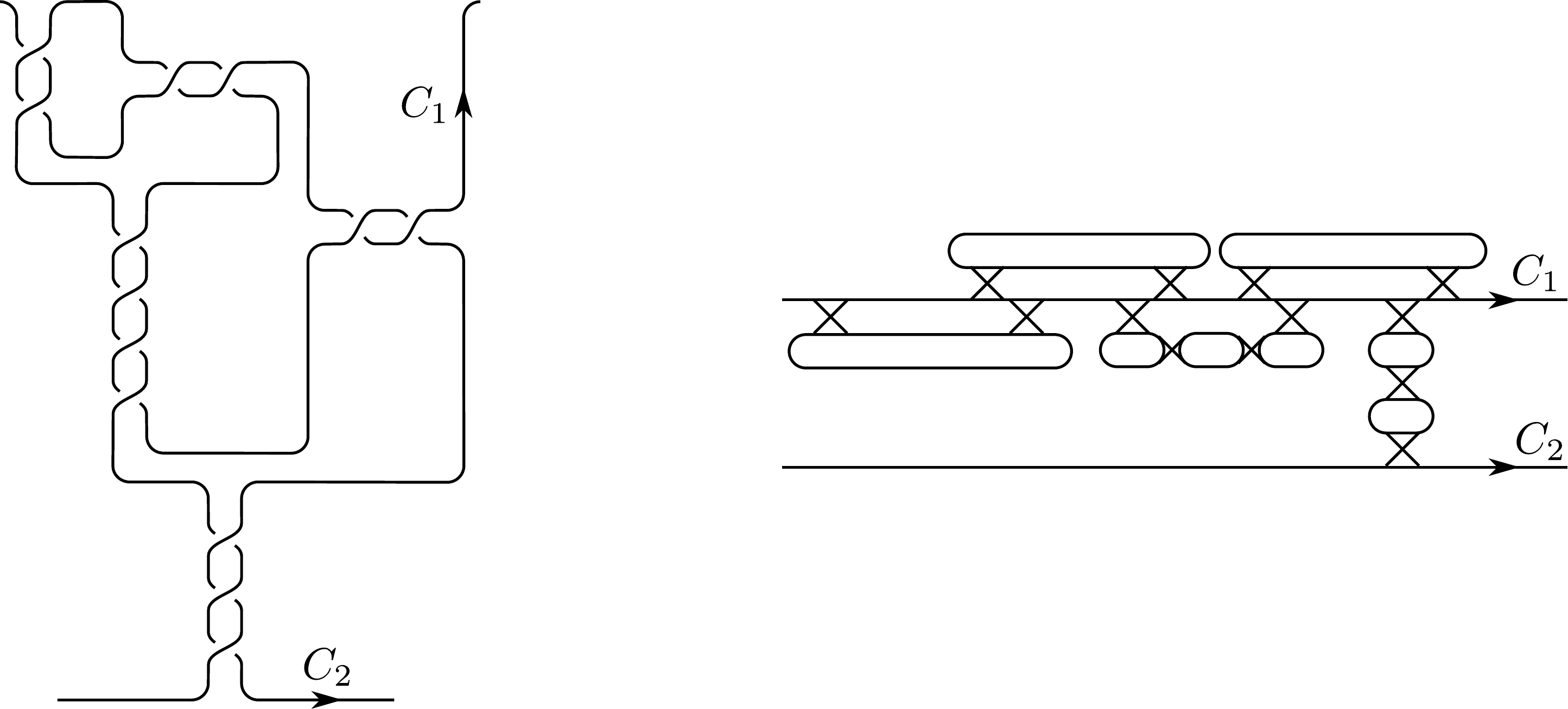}
\caption{A Type II(i) attachment. Left: The tangle $49/169=A(-3,+2,-4,+2,-2)$ with Seifert Parity 1 and odd $b_1^j$ ; Right: The realization of the tangle as an interlocked R-patterns attached to $C_1$ with exception (i), where $C_1$, $C_2$ are the two Seifert circles used to define the corresponding Type M1 link diagram. 
\label{TypeII(i)attachment}}
\end{figure} 

On the other hand, if $b_1^j$ is even, then $b_2^j$ must be negative and the crossings in $b_1^j$ and $b_2^j$ smooth as shown in Figure \ref{TypeII(ii)attachment}. In this case the $j$-th string contains $|b_1^j|$ Seifert circles and $|b_1^j|$ lone crossings (corresponding to the crossings in $b_1^j$), in addition the Seifert circle attached to $C_1$ shares $|b_2^j|+1$ negative crossings with $C_1$. Again this is a valid Type M1 string, and in this case the rest of the tangle is attached to the top huge Seifert circle $C_1$ as Type II(ii) attachment, where the added crossing is ``borrowed" by the pattern, recovering the original link diagram $D$. In Figure \ref{TypeII(ii)attachment} an arrow points to the crossing that was borrowed to create the type (ii) exception.
Since the resulting Type M1 link diagram remains a strong base link diagram with or without this additional crossing between $C_1$ and the Seifert circle in the $j$-th string attached to it, the condition of Theorem \ref{TypeD2_theorem} is met. Thus a Class M1 Montesinos link diagram remains a base link diagram.

\begin{figure}[htb!]
\includegraphics[scale=.4]{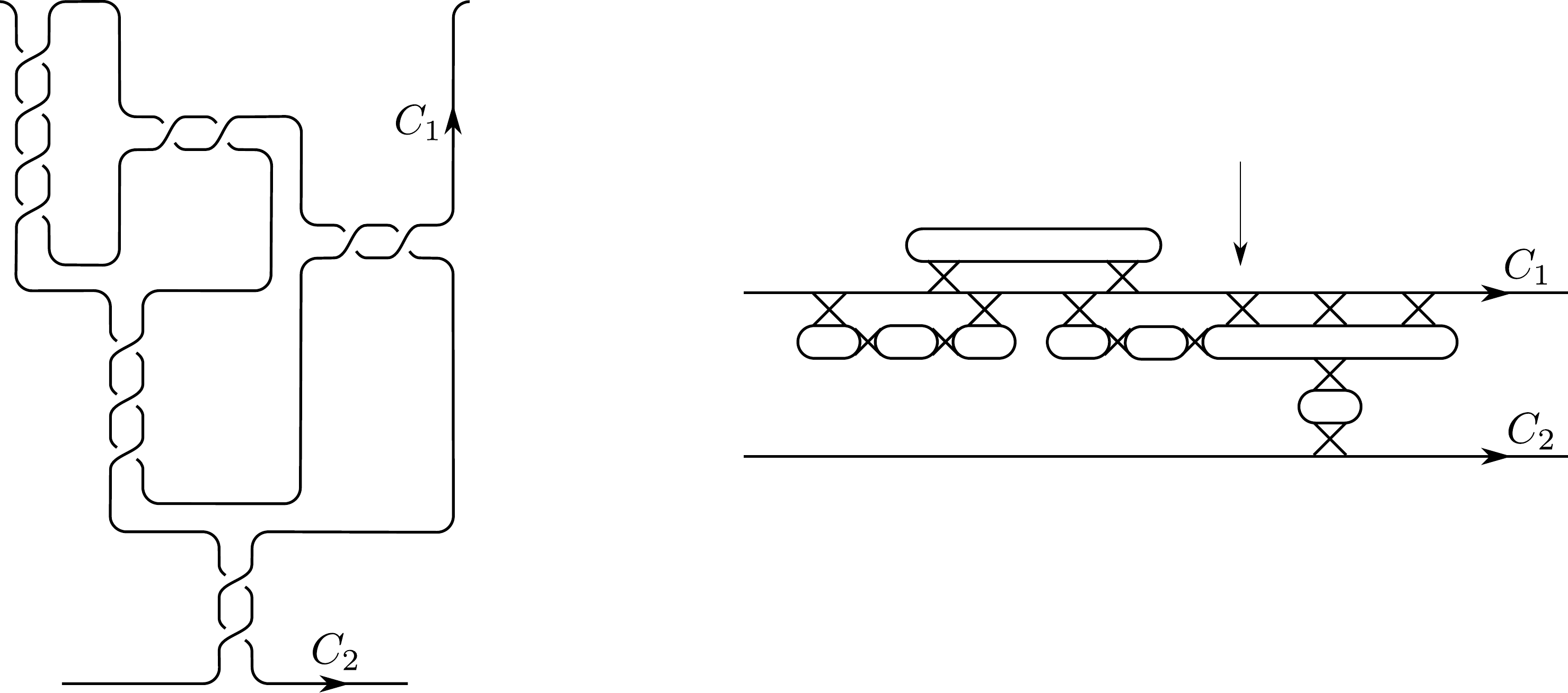}
\caption{A Type II(ii) attachment. Left: The tangle $71/173=A(-2,-2,-3,+2,-4)$ with Seifert Parity 1 and even $b_1^j$ ; Right: The realization of the tangle as an interlocked R-patterns attached to $C_1$ with exception (ii). The arrow points to the crossing that was borrowed to create the type (ii) exception. 
\label{TypeII(ii)attachment}}
\end{figure} 

\noindent
\underline{Class M2 case.} Here every tangle $A_j$ is of Seifert Parity 2 and $e=0$. As in the Class M1 case $b_1^j<0$ for all $1\le j\le k$. Overall this case is quite similar to the M1 Case. $L$ can be constructed from a Type M2 strong base link diagram $M_2$ with $k$ strings attached where all crossings in the strings are negative. More precisely, if $b_1^j$ is odd, then $b_2^j$ must be negative (in order for $A_j$ to have Seifert Parity 2), and the crossings in $b_1^j$ and $b_2^j$ smooth horizontally (similar to the case shown in Figure \ref{TypeII(ii)attachment}).  In this case the $j$-th string contains $|b_1^j|$ Seifert circles and is of length $|b_1^j|+1$. It contains $|b_1^j|$ lone crossings (corresponding to the crossings in $b_1^j$), and the Seifert circle attached to $C_1$ shares $|b_2^j|+1$ negative crossings with $C_1$. Since the resulting Type M2 link diagram remains a strong base link diagram with or without this additional crossing between $C_1$ and the Seifert circle in the $j$-th string attached to it, the condition of Theorem \ref{TypeD2_theorem} is met. Thus this is a valid Type M2 string, and in this case the rest of the tangle is attached to the top huge Seifert circle $C_1$ as Type II (ii), where the added crossing is ``borrowed" by the pattern. On the other hand, if $b_1^j$ is even, then $b_2^j$ must be positive and the crossings in $b_2^j$ smooth vertically (similar to the case shown in Figure \ref{TypeII(i)attachment}). 
In this case the $j$-th string contains exactly $|b_1^j|$ negative lone crossings, a valid Type M2 string. It is easy to see in this case that the rest of the tangle is attached to the top huge Seifert circle $C_1$ as Type II (i). Thus a Class M2 Montesinos link diagram remains a base link diagram.

\noindent
\underline{Class B case.} In this case at least one tangle is of Seifert Parity 3, there are no tangles of Seifert Parity 1 and the crossings in $e$ (if there are any) are positive and are smoothed vertically. Further more, the two long strands in the diagram belong to the same huge Seifert circle and the strands entering and exiting consecutive tangles of Seifert Parity 2 also belong to the same (large or huge) Seifert circles. In this case $L$ can be constructed from a Type B link diagram (denoted by $D_B$) with Type I attachments as follows. To construct $D_B$ we at first obtain the Seifert circles of $D_B$ by replacing each tangle with the tangle with two simple, no-intersection strands that match its Seifert Parity as given in Figure \ref{decomp}, and smooth the crossings in $e$. Thus the total number of Seifert circles in $D_B$ equals the number of crossings in $e$ plus the number of tangles with Seifert Parity 3. At this stage of the construction $D_B$ contains no crossings, for an example see the class B Montesinos link $M(17/44,7/10,19/26,2)$ in Figure \ref{classBconstructionone}.
Now we will add crossings starting with $e$ lone crossings to $D_B$ (these might not be the only lone crossings in $D_B$).

\begin{figure}[htb!]
\includegraphics[scale=.3]{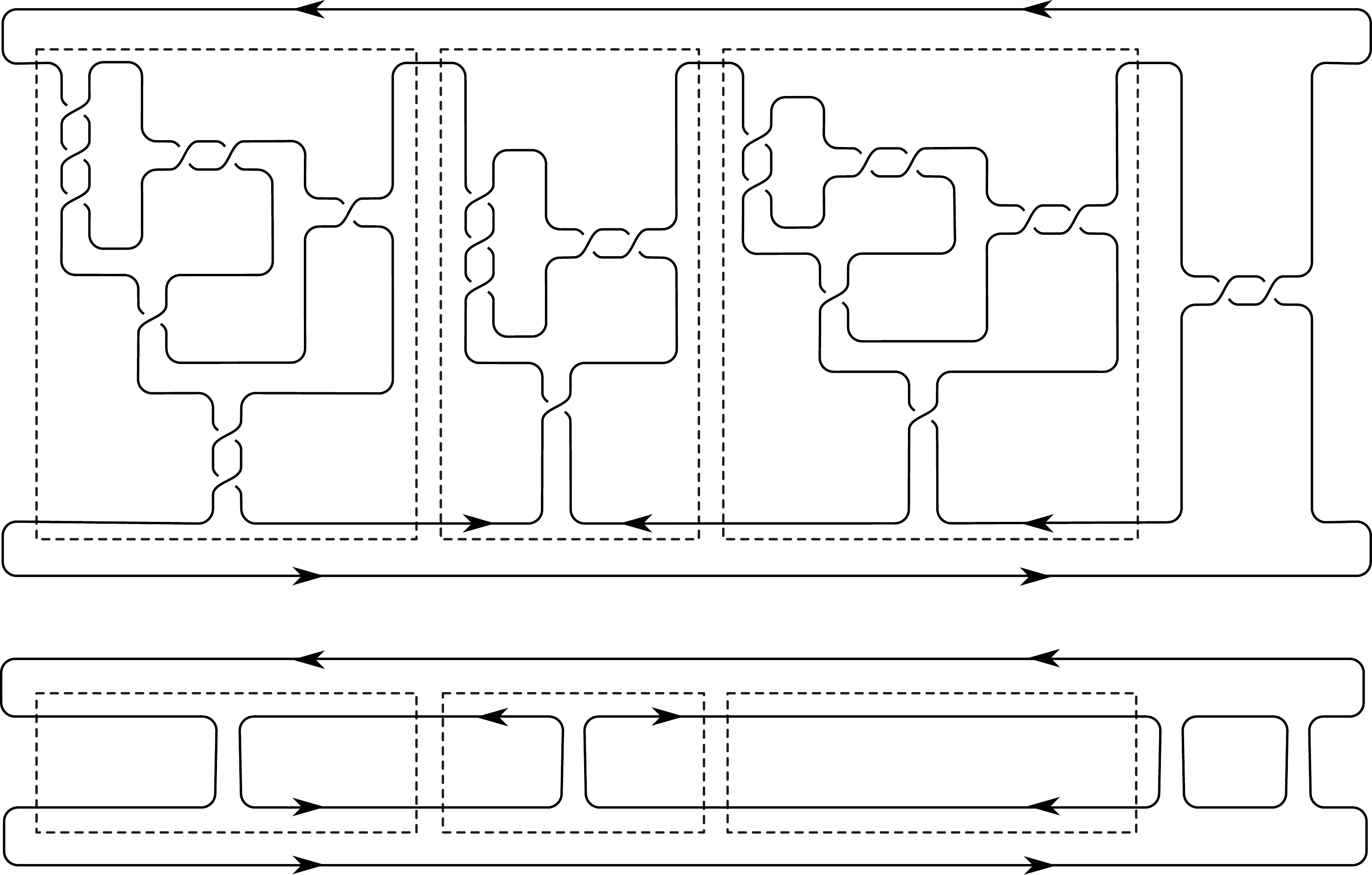}
\caption{Top: The class B Montesinos link $M(17/44,7/10,19/26,2)$. The first two tangles on the left  are of Seifert parity 3 and the third tangle is of Seifert parity 2. Bottom: The initial construction of $D_B$ with no crossings.
\label{classBconstructionone}}
\end{figure} 

This leaves the case of a tangle $A_j$  of Seifert Parity 3 for a detailed discussion. In this case $b_1^j>0$ and $b_2^j>0$, and in order for this to be a Seifert Parity 3 tangle the Seifert circle on the right of the tangle only touches a single crossing of  $b_2^j$ and all crossings of $b_1^j$ (that is the crossings in $b_1^j$ smooth vertically and the crossings in $b_2^j$ smooth horizontally). If $b_2^j=1$, then the strands after smoothing the crossings in $b_1^j$ and $b_2^j$ belong to different Seifert circles of $D_B$ if there is more than one tangle of Seifert Parity 3, and they belong to the single huge Seifert circle if there is exactly one tangle of Seifert Parity 3. The two strands share $b_1^j+b_2^j=b_1^j+1\ge 2$ crossings and the rest of the tangle can be attached to the large (or huge) Seifert circle on the left side via a Type II (i) attachment. On the other hand, if $b_2^j>1$, then we also place $b_1^j+1\ge 2$ crossings between the two corresponding Seifert circles in $D_B$ (created using the two simple strands that define the Seifert Parity of $A_j$), as the rest of the tangle (other than the portion containing $b_1^j$) is attached to the large (or huge) Seifert circle on the left side via a Type II (ii) attachment, which would borrow the added crossing back, recovering the original diagram, see Figure \ref{tangleRpatternwithexception}. In Figure \ref{classBconstructiontwo} we continue the construction of $D_B$ that was started in Figure \ref{classBconstructionone}. The two left most tangles are of Seifert parity 3. For the first we add $|b_1^1+1|$ crossings since $b_2^1=1$ and for the second we add $|b_1^2|$ crossings since $b_2^2>1$. At this stage $D_B$ is a Type B link diagram.

\begin{figure}[htb!]
\includegraphics[scale=.45]{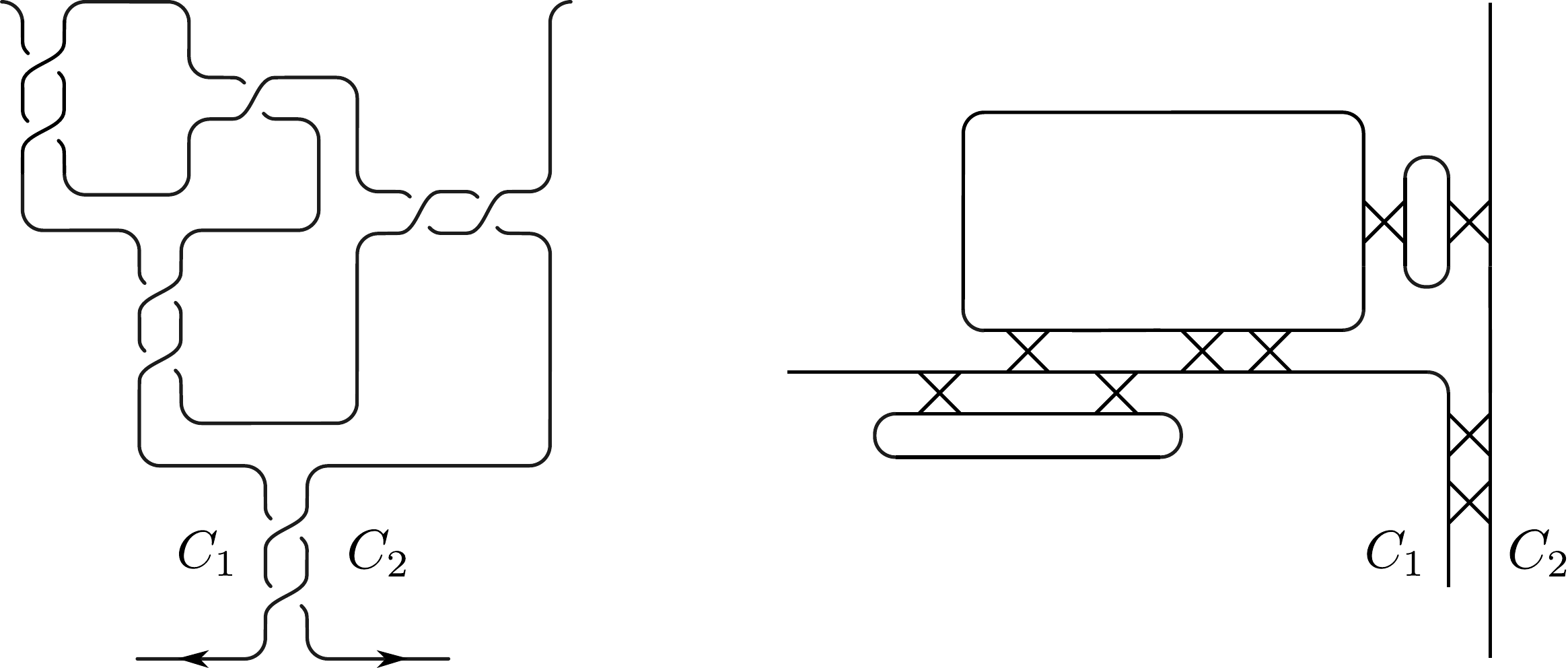}
\caption{Left: A standard drawing of the rational tangle $19/46 = A(2,2,2,1,2)$ with its Seifert circles; Right: Redrawn as an interlocked R-pattern with exception as displayed in Figure \ref{exceptions} on the right.}
\label{tangleRpatternwithexception}
\end{figure}

\begin{figure}[htb!]
\includegraphics[scale=.3]{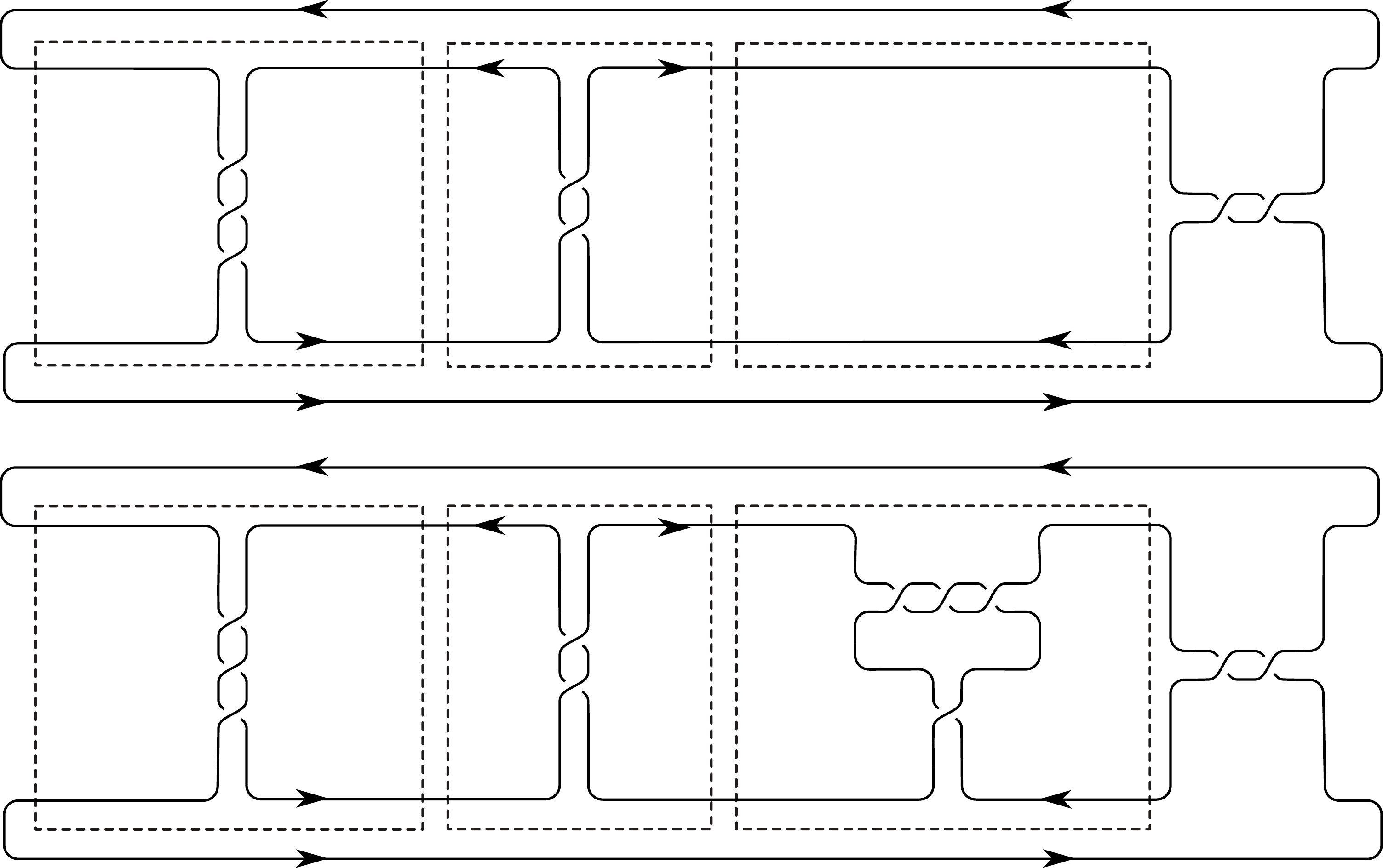}
\caption{The continuation of the construction of $D_B$ started in Figure \ref{classBconstructionone}.
Top: $D_B$ with $e$, $|b_1^1+1|$ and $|b_1^2|$ crossings added. A this stage $D_B$ is a Type B link diagram.  Bottom: A type I attachments is added to $D_B$ for the third tangle of Seifert parity 2.
\label{classBconstructiontwo}}
\end{figure} 

In the case of a tangle of Seifert Parity 2, the discussion  is almost identical to the Class M2 case discussed above. If $b_1^j$ is even we add a string of $|b_1^j|$ lone negative crossings to the large or huge Seifert circle that contains the Seifert Parity 2 tangle via a type I attachment. The rest of the tangle will become a Type II(i) attachment.  If If $b_1^j$ is odd we add a string of length $|b_1^j|+1$ consisting of $|b_1^j|$ negative lone crossings and is attached via $|b_2^j|+1$ crossings to the large or huge Seifert circle via a type I attachment. The rest of the tangle will become a Type II(ii) attachment. Figure \ref{classBconstructiontwo} shows that a string with $|b_1^3|$ Seifert circles with $|b_2^3|+1$ crossing attached at the top. Now the construction of $D_B$ is complete and the rest of the diagram can be created using Type II attachments. Figure \ref{classBconstructionthree} shows the complete diagram where three interlocked pattern we used to reconstruct the three tangles.

\begin{figure}[htb!]
\includegraphics[scale=.3]{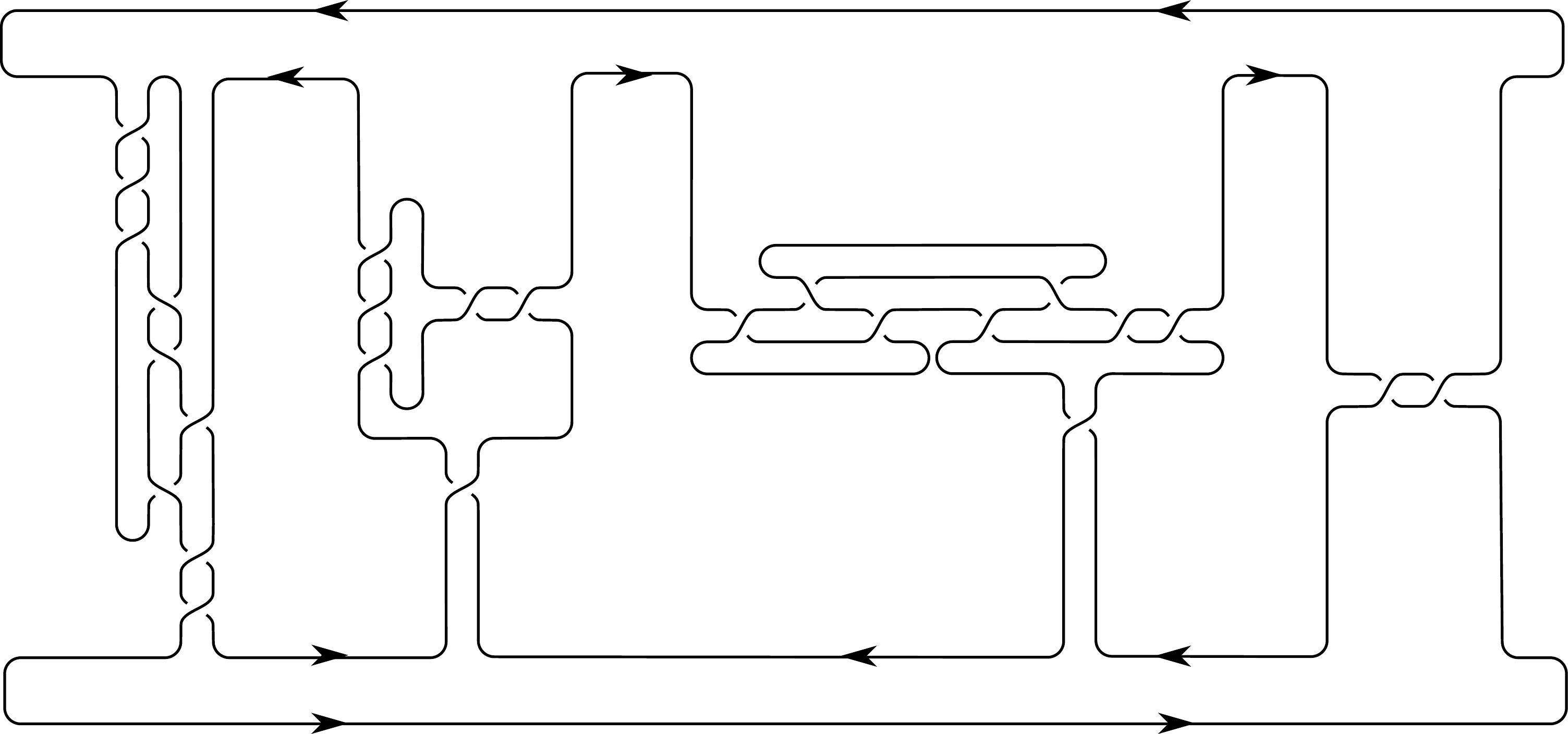}
\caption{The continuation of the construction of $D_B$ started in Figure \ref{classBconstructionone}.
For the first tangle a Type II(i) attachment, for the second a Type II(ii) attachment and for the third a Type II(ii) attachment is used.
\label{classBconstructionthree}}
\end{figure} 

Thus,
by the result on R-patterns from the Section \ref{s4} (Theorems \ref{TypeD_theorem}, \ref{TypeD_theoremsingleR}  and \ref{TypeD2_theorem}), we have proved the following theorem.

\begin{theorem}\label{Montesinos_theorem}
An alternating Montesinos link $L=M(\beta_1/\alpha_1,\ldots, \beta_k/\alpha_k,e)$ is a strong base link diagram. In particular, we have $\textbf{b}(L)=(E(L)-e(L))/2+1=s(L)-r(L)$, where $s(L)$ is the number of Seifert circles in the normal standard diagram of $L$ and $r(L)=r^-(L)+r^+(L)$ is the reduction number of the diagram.
\end{theorem}

\begin{remark}\label{braidindexusingreduction} 
{\em  For the example in Figures \ref{classBconstructionone}, \ref{classBconstructiontwo} and \ref{classBconstructionthree} there are 12 Seifert circles in the original diagram. From the first tangle we obtain $r^-(A_1)=1$, for the second $r^+(A_1)=1$, while the third does not contribute any increase in reduction numbers. The completed class $B$ link diagram $D_B$ has a reduction number $r^+(D_B)=1$. Thus $r(L)=r^-(L)+r^+(L)=1+2=3$ and we can redraw the diagram of the Montesinos link with $12-3=9$ Seifert circles. A computation of the HOMFLY polynomial yields that $E(L)=8$ and $e(L)=-8$ and we obtain a braid index of $b(L)=9$, as claimed by Theorem \ref{Montesinos_theorem}.}
\end{remark}

Theorem \ref{Montesinos_theorem} allows us to derive an explicit formula for the braid index of $L$ as the difference between $s(L)$ and $r(L)$, since we can compute $s(L)$ and $r(L)$ by computing the contributions of the individual tangles to $s(L)-r(L)$. We will spend the rest of this section on this task.

\noindent
\underline{Formulation of the braid index for an alternating Montesinos link.} Let $i=1$, 2 or 3 be the Seifert Parity type of $A_j$ and let $\Delta_i(A_j)$ be the contribution of the medium and small Seifert circles in $A_j$ to $s(L)-r(L)$. 

\noindent
The case of Seifert Parity 2. We will first determine $\Delta_i(A_j)$ when $A_j$ is of Seifert Parity 2. In this case the denominator $D(A_j)$ (see Figure \ref{tangle}) is a standard two bridge diagram as discussed in Subsection \ref{twobridgelinks} and $b_1^j<0$. Since we are not counting the large Seifert circle $C$, using Theorem \ref{2bridgetwo} we have:

\begin{eqnarray*}
\Delta_2(A_j)&=&(2+\sign(b^j_1)+\sign(b^j_{2q_j+1}))/4+\sum_{b^j_{2m}>0,1\le m\le q_j}b^j_{2m}/2+\sum_{b^j_{2m+1}<0,0\le m\le q_j}|b^j_{2m+1}|/2\\
&=&(1+\sign(b^j_{2q_j+1}))/4+\sum_{b^j_{2m}>0,1\le m\le q_j}b^j_{2m}/2+\sum_{b^j_{2m+1}<0,0\le m\le q_j}|b^j_{2m+1}|/2.
\end{eqnarray*}

\noindent
The case of Seifert Parity 1. If $A_j$ is of Seifert Parity 1, then $b^j_{1}<0$ but we cannot take the denominator directly (due to the orientations of the strands inherited from $L$). However if we add one more (negative) crossing to the crossings in $b^j_{1}$, then the orientation of the bottom strand is reversed and the result is a tangle $A_j^\p$ of Seifert Parity 2. We can compare $A_j$ with $A_j^\p$. In the above we just computed a formula that can be used to obtain $\Delta_2(A_j^\p)$. We note that $A_j^\p$ has the same reduction number as $A_j$ but has one more Seifert circle than $A_j$ does. Thus the contribution of the medium and small Seifert circles in $A_j$ to $s(L)-r(L)$ is:

\begin{eqnarray*}
\Delta_1(A_j)&=&-1+\Delta_2(A_j^\p)= -1+(1+\sign(b^j_{2q_j+1}))/4+\sum_{b^j_{2m}>0,1\le m\le q_j}b^j_{2m}/2\\
&+&(|b^j_{1}|+1)/2+\sum_{b^j_{2m+1}<0,1\le m\le q_{j}}|b^j_{2m+1}|/2\\
&=&(-1+\sign(b^j_{2q_j+1}))/4+\sum_{b^j_{2m}>0,1\le m\le q_j}b^j_{2m}/2+\sum_{b^j_{2m+1}<0,0\le m\le q_j}|b^j_{2m+1}|/2.
\end{eqnarray*}

\noindent
The case of Seifert Parity 3. Finally, if $A_j$ is of Seifert Parity 3, then $D(A_j)$ is a two bridge link diagram $K(A_j)$ (in its normal standard form) with $b^j_{1}>0$ and $b^j_{2}>0$. Notice that in this case $L$ must be of Class B. Furthermore, the NW and SW strands and and the NE and SE strands belong to large (or huge) Seifert circles in $D_B$. Then the rest of the Seifert circles form an interlocked R-pattern with exception (i) if $b_2^j=1$ or (ii) if $b_2^j>1$ (see the discussion of the class B case in the proof of Theorem \ref{Montesinos_theorem}). If we take the denominator of the tangle, then we obtain a (minimum) rational link diagram in which the large Seifert circle (the one containing the NW-SW strand) does not contain the NE-SE strand. Thus, the contribution of the medium and small Seifert circles to $s(L)-r(L)$ is the same as the braid index of the above denominator minus 2 since we are only counting the contribution of the medium and small Seifert circles (which do not include the ones containing the  NW-SW strand and the NE-SE strand). Thus using Theorem \ref{Montesinos_theorem} we have:

\begin{eqnarray*}
\Delta_3(A_j)&=&-1+(2+\sign(b^j_1)+\sign(b^j_{2q_j+1}))/4+\sum_{b^j_{2m}>0,1\le m\le q_j}b^j_{2m}/2+\sum_{b^j_{2m+1}<0,0\le m\le q_j}|b^j_{2m+1}|/2\\
&=&(-1+\sign(b^j_{2q_j+1}))/4+\sum_{b^j_{2m}>0,1\le m\le q_j}b^j_{2m}/2+\sum_{b^j_{2m+1}<0,0\le m\le q_j}|b^j_{2m+1}|/2.
\end{eqnarray*}

In the case when $L$ is of Class B, let $\eta$ be the number of Seifert Parity 3 $A_j$'s in $L$, then $D_B$ consists of a cycle of Seifert circles of length $2n=\eta+e$, where $e$ is the number of lone crossings. Recall that by the construction of $D_B$, the lone crossings in $D_B$ are precisely those in $e$. Combining the above and Theorems \ref{cycle_theorem}, \ref{multipath_theorem1}, \ref{TypeI_theorem} and \ref{TypeD_theorem}, we obtain the following complete formulation for the braid index of an alternating Montesinos link (presented in a normal standard diagram).

\begin{theorem}\label{Montesinos_formula}
Let $L=M(\beta_1/\alpha_1,\ldots, \beta_k/\alpha_k,\delta)=M(A_1,A_2,\ldots, A_k,e)$ be an alternating Montesinos link with a normal standard diagram and  the signed vector $(b_{1}^j,b^j_2,...,b_{2q_j+1}^j)$ for $A_j$, we have
\begin{eqnarray}
\textbf{b}(L)&=&2+\sum_{1\le j\le k}\Delta_1(A_j)\ \rm{if}\ L\ \rm{is\ of\ Class\ M1};\label{Mformula1}\\
\textbf{b}(L)&=&1+\sum_{1\le j\le k}\Delta_2(A_j) \ \rm{if}\ L\ \rm{is\ of\ Class\ M2};\label{Mformula2}\\
\textbf{b}(L)&=&\Delta_0(L)+\sum_{A_j \in \Omega_2}\Delta_2(A_j)+\sum_{A_j \in \Omega_3}\Delta_3(A_j) \ \rm{if}\ L\ \rm{is\ of\ Class\ B},\label{Mformula3}
\end{eqnarray}
where $\Omega_2$, $\Omega_3$ are the sets of Seifert Parity 2 and Seifert Parity 3 $A_j$'s respectively, $\Delta_0(L)=\eta+e-\min\{(\eta+e)/2-1,e\}$ and $\eta=\vert \Omega_3\vert$. 
\end{theorem}

We end this section with a few examples. These are relatively small Montesinos knots whose HOMFLY polynomials (hence $E(L)-e(L))/2+1$) can be computed directly to verify the results. 

\medskip
\begin{example}{\em Let $K=12a_{304}=M(7/19,1/3,1/2,0)$ in Figure \ref{Me1}. This is an alternating Montesinos link of Class B with $e=0$ and $\eta=2$. $7/19=(2,1,2,1,1)$ has a signed vector of $(2,1,-2,1,1)$ so it is of Seifert Parity 3, $1/3=(3)$ has a signed vector of $(3)$ so it is of Seifert Parity 3, $1/2=(2)$ has a signed vector of $(-2)$ so it is of Seifert Parity 2. We obtain $\Delta_3(7/19)=1+1=2$, $\Delta_3(1/3)=0$, $\Delta_2(1/2)=1$ (here we count $\sign(b^j_1)=\sign(b^j_{2q_j+1})=-1$) and $\Delta_0(L)=2-0=2$.
By (\ref{Mformula3}) we obtain $\textbf{b}(12a_{304})=2+1+2=5$.}
\end{example}

\begin{figure}[htb!]
\includegraphics[scale=.3]{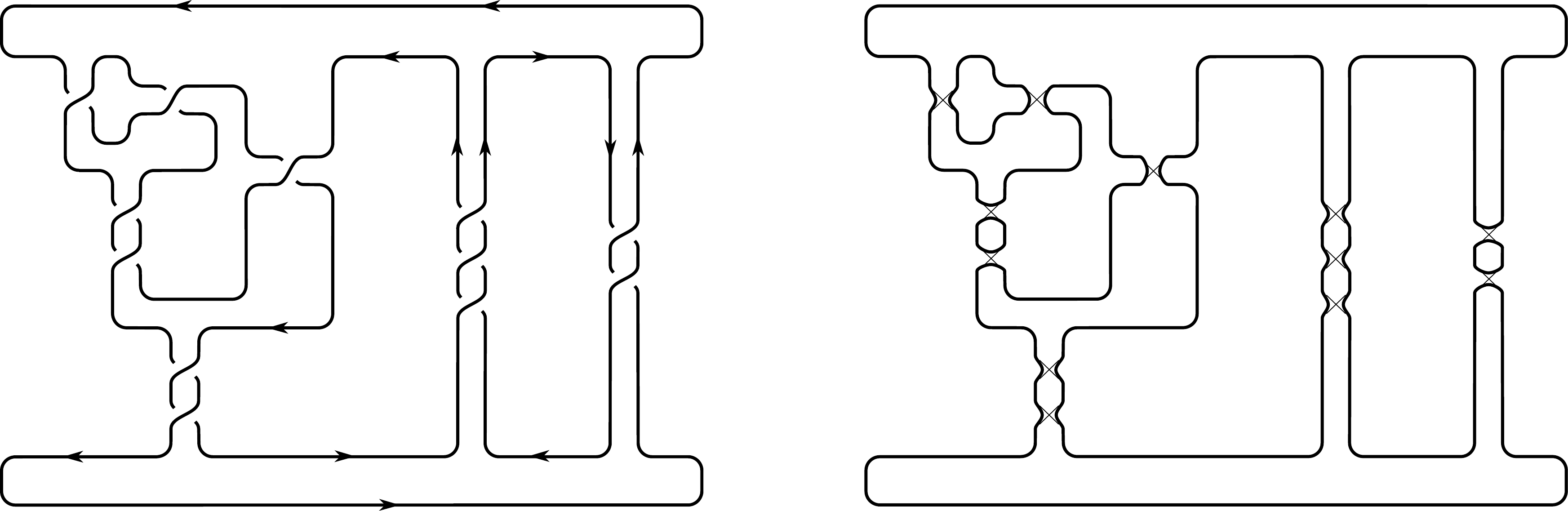}
\caption{Left: The Montesinos knot $K=12a_{304}=M(7/19,1/3,1/2,0)$ in its normal standard form; Right: The Seifert circle decomposition of $K$.}
\label{Me1}
\end{figure}

\medskip
\begin{example} {\em Let $K=12a_{252}=M(1/4,3/5,1/3,1)$ in Figure \ref{Me3}. This is also a Montesinos link of Class B with $e=1$ and $\eta=3$. $1/4=(4)$ has a signed vector of $(4)$ and $\Delta_3(1/4)=0$, $3/5=(1,1,2)$ has a signed vector of $(1,1,-2)$ and $\Delta_3(3/5)=-1/2+1/2+1=1$, $1/3=(3)$ has a signed vector of $(3)$ and $\Delta_3(1/3)=0$.
By (\ref{Mformula3}) we obtain $\textbf{b}(12a_{252})= 1+(4-1)=4$.}
\end{example}

\begin{figure}[htb!]
\includegraphics[scale=.4]{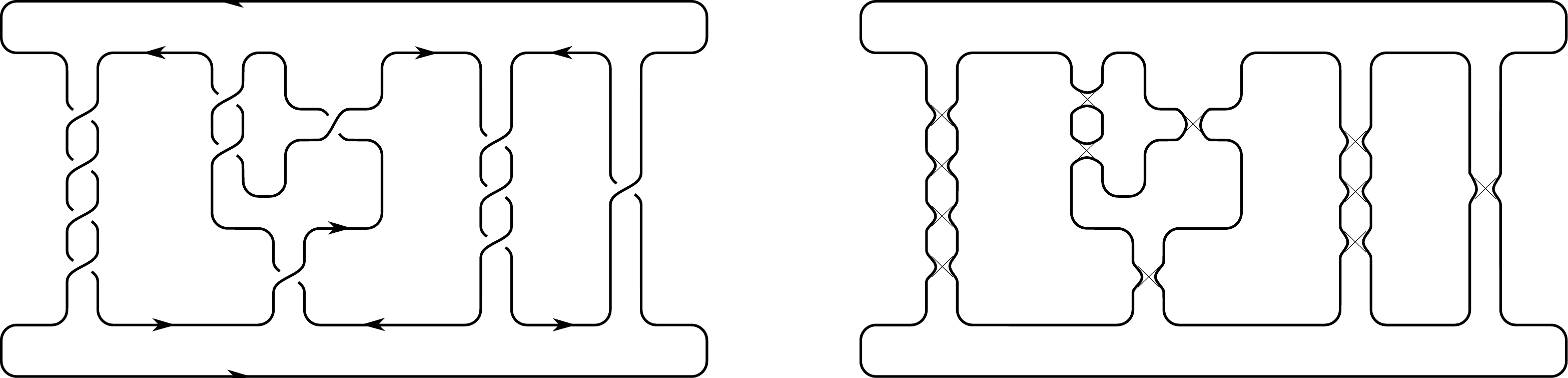}
\caption{Left: The Montesinos knot $K=12a_{252}=M(1/4,3/5,1/3,1)$ in its normal standard form; Right: The Seifert circle decomposition of $K$.}
\label{Me3}
\end{figure}

\medskip
\begin{example}{\em  Let $L=M(12/19,2/3,2)=b(188,79)$ in Figure \ref{Me4}. This is a Montesinos link of Class M1 with $e=2$ and the two tangles are of Seifert Parity 1: $12/19=(1,1,1,2,2)$ has a signed vector of $(-1,1,1,2,2)$ and $2/3=(1,1,1)$ has a signed vector of $(-1,1,1)$. We have $\Delta_1(12/19)=0+3/2+1/2=2$, $\Delta_1(2/3)=0+1/2+1/2=1$.
By (\ref{Mformula1}) we obtain $\textbf{b}(b(188,79))=2+2+1=5$. Since there are only two rational tangles this link is actually a two bridge link. We can redraw the diagram to see that the vector is $(2, 2, 1, 1, 1, 2, 1, 1, 1)$ which gives the two bridge link $b(188,79)$.
The signed vector is $(2, 2, 1, 1, -1, -2, -1, 1, 1)$ and applying Formula (\ref{2bridgeformula}) also yields $1+4/4+(1+1/2+1/2)+(1/2+1/2)=5$, as expected.}
\end{example}

\begin{figure}[htb!]
\includegraphics[scale=.3]{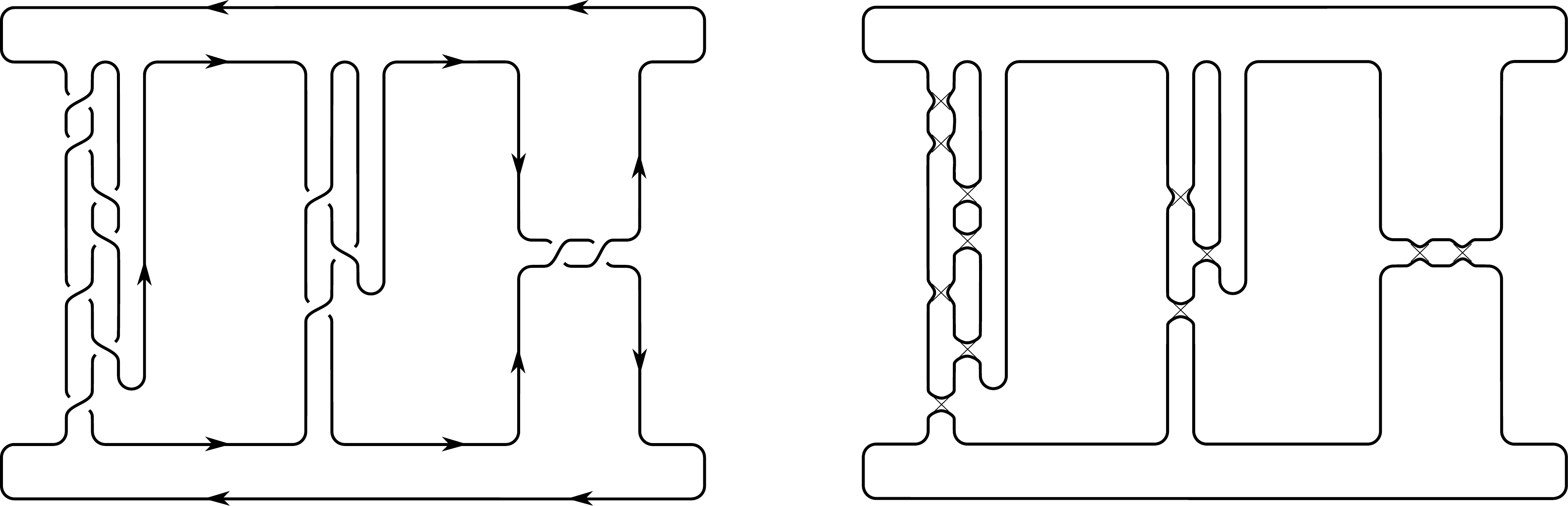}
\caption{Left: The Montesinos knot $K=M(12/19,2/3,2)=b(188,79)$ in its normal standard form; Right: The Seifert circle decomposition of $K$.}
\label{Me4}
\end{figure}

\medskip
\begin{example}{\em  Let $K=12a_{83}=M(12/19,2/3,1/2)$ in Figure \ref{Me5}. This is a Montesinos link of Class M2 with $e=0$ and three tangles of Seifert Parity 2: $12/19=(1,1,1,2,2)$ has a signed vector of $(-1,-1,-1,2,-2)$, $2/3=(1,1,1)$ has a signed vector of $(-1,-1,-1)$ and $1/2=(2)$ has a signed vector of $(-2)$. We have $\Delta_2(12/19)=0+1+2=3$, $\Delta_2(2/3)=0+0+1=1$ and $\Delta_2(1/2)=0+0+1=1$.
By (\ref{Mformula2}) we obtain $\textbf{b}(12a_{83})=1+3+1+1=6$.}
\end{example}

\begin{figure}[htb!]
\includegraphics[scale=.3]{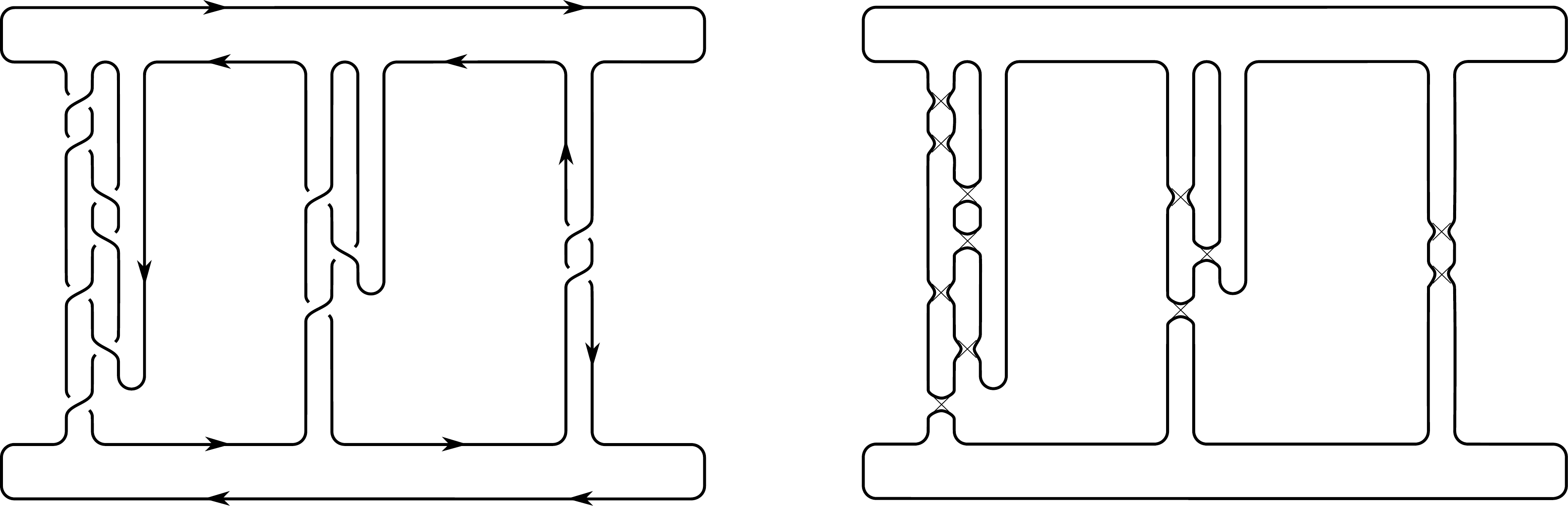} 
\caption{Left: The Montesinos knot $K=12a_{83}=M(12/19,2/3,1/2)$ in its normal standard form; Right: The Seifert circle decomposition of $K$.}
\label{Me5}
\end{figure}

\medskip
\begin{example}{\em  Let $L=M(17/44,7/10,19/26,2)$ be the earlier example shown in Figure \ref{classBconstructionone}. This is a Montesinos link of Class B with $e=2$ and two tangles of Seifert Parity 3 and one tangle of Seifert Parity 2: $17/44=(2,1,1,2,3)$ has a signed vector of $(2,1,-1,-2,-3)$, $7/10=(1,2,3)$ has a signed vector of $(1,2,3)$ and $19/26=(1,2,1,2,2)$ has a signed vector of $(-1,-2,-1,2,-2)$. We have $\Delta_0=3$, $\Delta_3(17/44)=2$, $\Delta_3(7/10)=1$ and $\Delta_2(19/26)=3$.
By (\ref{Mformula2}) we obtain $\textbf{b}(L)=9$ which agrees with Example \ref{braidindexusingreduction}.}
\end{example}

\section{Summary and future work}\label{s6}

In this paper we presented algorithms that allow the determination of the braid index of an oriented alternating link $L$ directly from a minimal diagram $D$ of the link $L$. We introduced several classes of link diagrams (with lone crossings) that can be constructed from alternating link diagrams without lone crossings. We show that for each of these link diagrams we can determine the maximum number of the (well known) Seifert circle reduction moves (see for example \cite{MP}). We show that the equality of the Morton-William-Frank inequality holds for the diagram obtained after the reduction of $r(D)$ Seifert circles. As applications of our methods and results, we derived a new formulation of the braid index for rational links based on minimum diagrams. We also show that the braid index of any alternating Montesinos link satisfies the equation $\textbf{b}(D)=s(D)-r(D)$. Using this formula, we are able to derive an explicit formula for computing the braid index of any alternating Montesinos link from a minimal diagram. We point out that the techniques in this paper can be used for even larger classes of alternating diagram, such as subfamilies of the arborescent knot family (or Conway algebraic knots) \cite{Bo}, and can be extended to some non-alternating link families as well. This will be a topic of future work of the authors.

\end{document}